\documentclass[leqno,11pt]{article}
\usepackage[utf8]{inputenc}
\usepackage[T1]{fontenc}
\usepackage{microtype}

\usepackage[dvipsnames]{xcolor}
\usepackage{xcolor-solarized}

\usepackage{calc}
\usepackage[a4paper]{geometry}

\ifdefined\screenLayout
  \pagecolor{solarized-base3}
  \color{solarized-base03}
  \usepackage{fancyhdr}
\fi

\usepackage{amsmath}
\usepackage{amsthm}

\usepackage{tikz-cd}
\usetikzlibrary{arrows} 
\tikzset{
  commutative diagrams/.cd, 
  arrow style=tikz, 
  diagrams={>=stealth}
}
\usetikzlibrary{matrix,decorations.pathreplacing,calc}
\usetikzlibrary{graphs,graphs.standard}
\usepackage{pgfplots}
\usepgfplotslibrary{external}

\usepackage{afterpage}
\usepackage{pdflscape}

\usepackage{colortbl}
\usepackage{adforn}
\usepackage{nameref}

\usepackage{textcomp}
\usepackage[sb]{libertine}
\usepackage[varqu,varl]{zi4}%
\usepackage[libertine,bigdelims,vvarbb]{newtxmath}

\IfPackageAtLeastTF{superiors}{2.0}{%
  \usepackage[supsfam=LibertinusSerif-Sup,supscaled=1.2,raised=-.13em]{superiors}
}{%
  \usepackage[supstfm=libertinesups,supscaled=1.2,raised=-.13em]{superiors}
}

\useosf
\usepackage[scr=boondox,cal=euler]{mathalfa}
\usepackage{marvosym}
\usepackage{slashed}
\usepackage{esint} 
\usepackage[ngerman,english]{babel}
\usepackage{imakeidx}
\makeindex[intoc]
\usepackage{csquotes}
\usepackage[
  backend=biber,
  hyperref=true,
  backref=true,
  isbn=false,
  doi=true,
  natbib=true,
  eprint=true,
  useprefix=true,
  maxcitenames=99,
  maxbibnames=99,  
  maxalphanames=99, 
  minalphanames=99,
  safeinputenc,
  style=alphabetic,
  citestyle=alphabetic,
  block=space,
  datamodel=preamble/ext-eprint,
  sorting=nyt
]{biblatex}
\usepackage[
  bookmarksnumbered = true,
  hypertexnames = false,
  colorlinks    = true,
  citecolor     = solarized-blue,
  linkcolor     = solarized-blue,
  urlcolor      = solarized-blue,
  breaklinks
]{hyperref}

\DeclareFieldFormat{url}{%
  \href{#1}{\ComputerMouse}
}
\DeclareFieldFormat{doi}{%
  \mkbibacro{DOI}\addcolon\space\href{https://doi.org/#1}{#1}
}
\makeatletter
\DeclareFieldFormat{arxiv}{%
  arXiv\addcolon\space\href{http://arxiv.org/\abx@arxivpath/#1}{#1}
}
\makeatother
\DeclareFieldFormat{mr}{%
  MR\addcolon\space\href{http://www.ams.org/mathscinet-getitem?mr=MR#1}{#1}
}
\DeclareFieldFormat{zbl}{%
  Zbl\addcolon\space\href{http://zbmath.org/?q=an:#1}{#1}
}
\renewbibmacro*{eprint}{%
  \printfield{arxiv}%
  \newunit\newblock
  \printfield{mr}%
  \newunit\newblock
  \printfield{zbl}%
  \newunit\newblock
  \iffieldundef{eprinttype}
  {\printfield{eprint}}
  {\printfield[eprint:\strfield{eprinttype}]{eprint}}
}

\AtEveryBibitem{%
  \clearlist{address}%
}
\DeclareFieldFormat[article,inproceedings,inbook,incollection,thesis]{title}{\textit{#1}}
\renewbibmacro{in:}{}
\usepackage[inline,shortlabels]{enumitem}

\usepackage{subcaption}

\usepackage[yyyymmdd]{datetime}

\usepackage{etoolbox}
\ifundef{\abstract}{}{\patchcmd{\abstract}%
    {\quotation}{\quotation\noindent\ignorespaces}{}{}}

\usepackage[super]{nth}

\usepackage{thmtools}
\usepackage[framemethod=TikZ]{mdframed}

\numberwithin{equation}{section}

\renewcommand{\qedsymbol}{$\blacksquare$}

\newcommand{\CorollaryQED}{\qedsymbol}
\newcommand{\ConjectureQED}{$\square$}
\newcommand{\SituationQED}{$\times$}
\newcommand{\DefinitionQED}{$\bullet$}
\newcommand{\NotationQED}{$\circ$}
\newcommand{\ExampleQED}{$\spadesuit$}
\newcommand{\RemarkQED}{$\clubsuit$}
\newcommand{\ExerciseQED}{?!}

\ifdefined\screenLayout
  \declaretheoremstyle[
  bodyfont=\itshape,
  mdframed={    
    backgroundcolor=solarized-base3!90!solarized-blue,
    linewidth=0,
    innerleftmargin=.5em,
    innerrightmargin=.5em,
    innertopmargin=.5em,
    innerbottommargin=.5em,
    leftmargin=-.5em,
    rightmargin=-.5em,
  }
]{theorem}

\declaretheoremstyle[
mdframed={
  backgroundcolor=solarized-base3!90!solarized-green,
  linewidth=0,
  innerleftmargin=.5em,
  innerrightmargin=.5em,
  innertopmargin=.5em,
  innerbottommargin=.5em,
  leftmargin=-.5em,
  rightmargin=-.5em,
  }
]{definition}

\declaretheoremstyle[
  mdframed={
    backgroundcolor=solarized-base3!90!solarized-yellow,
    linewidth=0,
    innerleftmargin=.5em,
    innerrightmargin=.5em,
    innertopmargin=.5em,
    innerbottommargin=.5em,
    leftmargin=-.5em,
    rightmargin=-.5em,
  }
]{example}

\declaretheoremstyle[
  mdframed={
    backgroundcolor=solarized-base3!90!solarized-orange,
    linewidth=0,
    innerleftmargin=.5em,
    innerrightmargin=.5em,
    innertopmargin=.5em,
    innerbottommargin=.5em,
    leftmargin=-.5em,
    rightmargin=-.5em,
  }
]{remark}
\else
  \declaretheoremstyle[
  bodyfont=\itshape
  ]{theorem}
  \declaretheoremstyle[]{definition}
  \declaretheoremstyle[]{example}
  \declaretheoremstyle[]{remark}
\fi

\declaretheorem[numberlike=equation,style=theorem]{theorem}
\declaretheorem[numbered=no,name=Theorem,style=theorem]{theorem*}
\declaretheorem[numberlike=equation,name=Lemma,style=theorem]{lemma}
\declaretheorem[numberlike=equation,name=Proposition,style=theorem]{prop}
\declaretheorem[numberlike=equation,name=Corollary,qed=\CorollaryQED,style=theorem]{cor}
\declaretheorem[numberlike=equation,name=Conjecture,qed=\ConjectureQED,style=theorem]{conjecture}

\declaretheorem[numberlike=equation,name=Situation,style=definition,qed=\SituationQED]{situation}
\declaretheorem[numberlike=equation,name=Definition,style=definition,qed=\DefinitionQED]{definition}
\declaretheorem[numbered=no,name=Definition,style=definition,qed=\DefinitionQED]{definition*}
\declaretheorem[numberlike=equation,name=Notation,style=definition,qed=\NotationQED]{notation}

\declaretheorem[numberlike=equation,style=definition,qed=\ExampleQED]{example}

\declaretheorem[numberlike=equation,style=remark,qed=\RemarkQED]{remark}
\declaretheorem[numbered=no,style=remark,name=Remark,qed=\RemarkQED]{remark*}

\declaretheorem[numberlike=equation,style=definition]{question}

\def\makeautorefname#1#2{\AtBeginDocument{\expandafter\def\csname#1autorefname\endcsname{#2}}}
\makeautorefname{table}{Table}        
\makeautorefname{chapter}{Chapter}
\makeautorefname{section}{Section}
\makeautorefname{subsection}{Section}
\makeautorefname{subsubsection}{Section}
\makeautorefname{footnote}{Footnote}
\AtBeginDocument{\def\itemautorefname~#1\null{(#1)\null}}
\AtBeginDocument{\def\equationautorefname~#1\null{(#1)\null}}

\numberwithin{substep}{step}
\makeautorefname{step}{Step}
\makeautorefname{substep}{Step}

\makeautorefname{case}{Case}
\makeautorefname{substep}{Step}

\setlist[description]{leftmargin=!,labelindent=1em}
\setlist[enumerate]{label={\rm (\arabic*)},ref=\arabic*}
\setlist[enumerate,2]{label={\rm (\alph*)},ref=\theenumi.\alph*}
\setlist[enumerate,3]{label={\rm (\roman*)},ref=\theenumii.\roman*}

\let\C\undefined

\usepackage{bm}
\usepackage{mathtools} 
\usepackage{stmaryrd} 

\DeclareFontFamily{U}{mathx}{\hyphenchar\font45}
\DeclareFontShape{U}{mathx}{m}{n}{
      <5> <6> <7> <8> <9> <10>
      <10.95> <12> <14.4> <17.28> <20.74> <24.88>
      mathx10
      }{}
\DeclareSymbolFont{mathx}{U}{mathx}{m}{n}
\DeclareFontSubstitution{U}{mathx}{m}{n}
\DeclareMathAccent{\widecheck}{0}{mathx}{"71}
\DeclareMathAccent{\wideparen}{0}{mathx}{"75}

\DeclareMathOperator{\Aut}{Aut}

\DeclareMathOperator{\flow}{flow}
\DeclareMathOperator{\Fr}{Fr}

\DeclareMathOperator{\graph}{graph}
\DeclareMathOperator{\HF}{\HF}

\DeclareMathOperator{\Hom}{Hom}

\DeclareMathOperator{\Sym}{Sym}

\DeclareMathOperator{\area}{area}

\DeclareMathOperator{\coker}{coker}

\DeclareMathOperator{\im}{im}
\DeclareMathOperator{\ind}{index}
\DeclareMathOperator{\inj}{inj}

\DeclareMathOperator{\sign}{sign}

\DeclareMathOperator{\supp}{supp}

\DeclarePairedDelimiter\floor{\lfloor}{\rfloor}
\DeclarePairedDelimiter\paren{\lparen}{\rparen}
\DeclarePairedDelimiter\sqparen{[}{]}
\DeclarePairedDelimiter{\Abs}{\|}{\|}

\DeclarePairedDelimiter{\Inner}{\langle}{\rangle}

\DeclarePairedDelimiter{\abs}{\lvert}{\rvert}

\DeclarePairedDelimiter{\set}{\lbrace}{\rbrace}
\def\({\left(}
\def\){\right)}
\def\<{\left\langle}
\def\>{\right\rangle}

\newcommand{\C}{{\mathbf{C}}}

\newcommand{\N}{{\mathbf{N}}}

\newcommand{\Q}{\mathbf{Q}}

\newcommand{\R}{\mathbf{R}}

\newcommand{\Z}{\mathbf{Z}}

\newcommand{\co}{\mskip0.5mu\colon\thinspace}

\newcommand{\defined}[2][\key]{\def\key{#2}\textbf{#2}\index{#1}}
\newcommand{\delbar}{\bar{\del}}

\newcommand{\del}{\partial}
\newcommand{\ev}{\mathrm{ev}}

\newcommand{\id}{\mathrm{id}}
\newcommand{\incl}{\hookrightarrow}
\newcommand{\emb}{\hookrightarrow}

\newcommand{\loc}{\mathrm{loc}}

\newcommand{\pr}{\mathrm{pr}}
\newcommand{\qandq}{\quad\text{and}\quad}

\newcommand{\qand}{\quad\text{and}}

\newcommand{\qwithq}{\quad\text{with}\quad}

\newcommand{\vol}{\mathrm{vol}}

\renewcommand{\Im}{\operatorname{Im}}

\renewcommand{\emptyset}{\varnothing}
\renewcommand{\epsilon}{\varepsilon}
\renewcommand{\setminus}{{\backslash}}

\renewcommand{\leq}{\leqslant}
\renewcommand{\geq}{\geqslant}

\makeatletter
\renewcommand*\env@matrix[1][*\c@MaxMatrixCols c]{%
  \hskip -\arraycolsep
  \let\@ifnextchar\new@ifnextchar
  \array{#1}}

\renewcommand\xleftrightarrow[2][]{%
  \ext@arrow 9999{\longleftrightarrowfill@}{#1}{#2}}
\newcommand\longleftrightarrowfill@{%
  \arrowfill@\leftarrow\relbar\rightarrow}
\makeatother

\newcommand{\rd}{{\rm d}}

\newcommand{\rH}{{\rm H}}

\newcommand{\rII}{{\rm II}}

\newcommand{\bF}{{\mathbf{F}}}

\newcommand{\bI}{{\mathbf{I}}}
\newcommand{\bJ}{{\mathbf{J}}}
\newcommand{\bK}{{\mathbf{K}}}

\newcommand{\bM}{{\mathbf{M}}}

\newcommand{\bO}{{\mathbf{O}}}

\newcommand{\bR}{{\mathbf{R}}}

\newcommand{\cA}{\mathcal{A}}

\newcommand{\cC}{\mathcal{C}}

\newcommand{\cH}{\mathcal{H}}

\newcommand{\cJ}{\mathcal{J}}
\newcommand{\cK}{\mathcal{K}}

\newcommand{\cM}{\mathcal{M}}
\newcommand{\cN}{\mathcal{N}}
\newcommand{\cO}{\mathcal{O}}
\newcommand{\cP}{\mathcal{P}}
\newcommand{\cQ}{\mathcal{Q}}

\newcommand{\cS}{\mathcal{S}}
\newcommand{\cT}{\mathcal{T}}
\newcommand{\cU}{\mathcal{U}}
\newcommand{\cV}{\mathcal{V}}
\newcommand{\cW}{\mathcal{W}}

\newcommand{\cZ}{\mathcal{Z}}

\newcommand{\sH}{\mathscr{H}}

\newcommand{\sO}{\mathscr{O}}

\newcommand{\fd}{{\mathfrak d}}

\newcommand{\fp}{{\mathfrak p}}
\newcommand{\fq}{{\mathfrak q}}

\newcommand{\fz}{{\mathfrak z}}

\newcommand{\fH}{{\mathfrak H}}

\newcommand{\dF}{{\rd F}}

\renewcommand{\emb}{\mathrm{emb}}
\newcommand{\isol}{\mathrm{isol}}
\newcommand{\GW}{\mathrm{GW}}
\newcommand{\BPS}{\mathrm{BPS}}
\newcommand{\vir}{\mathrm{vir}}
\newcommand{\bbJ}{\mathbb{J}}
\newcommand{\vdim}{\mathrm{vdim}}
\newcommand{\simple}{\mathrm{si}}
\renewcommand{\emb}{\mathrm{emb}}

\newcommand{\SpaceOfHermitianStructures}{\cH}
\newcommand{\ModuliSpaceOfMaps}{\cM}
\newcommand{\ModuliSpaceOfNodalMaps}{\overline\ModuliSpaceOfMaps}
\newcommand{\ModuliSpaceOfSimpleMaps}{\ModuliSpaceOfMaps^\simple}
\newcommand{\ModuliSpaceOfEmbeddedMaps}{\ModuliSpaceOfMaps^\emb}
\newcommand{\SpaceOfCompactSubsets}{\cK}
\newcommand{\SpaceOfCycles}{\cZ}
\newcommand{\SpaceOfConnectedCycles}{\cC}
\newcommand{\SpaceOfCurves}{\SpaceOfCycles^\simple}
\newcommand{\SpaceOfConnectedCurves}{\SpaceOfConnectedCycles^\simple}
\newcommand{\SpaceOfEmbeddedCurves}{\SpaceOfCycles^\emb}
\newcommand{\SpaceOfConnectedEmbeddedCurves}{\SpaceOfConnectedCycles^\emb}
\newcommand{\SpaceOfSubmanifolds}{\cS}
\newcommand{\SpaceOfAlmostComplexStructures}{\cJ}
\newcommand{\SpaceOfAlmostComplexStructuresRegular}{\SpaceOfAlmostComplexStructures_*}
\newcommand{\SpaceOfAlmostComplexStructuresDiscrete}{\SpaceOfAlmostComplexStructures_\isol}
\newcommand{\genus}{\mathrm{g}}
\newcommand{\HurwitzSpace}[2]{\overline\cH_{#1,#2}}

\newcommand{\OtherQED}{$\bullet$}
\renewcommand{\ConjectureQED}{\OtherQED}
\renewcommand{\SituationQED}{\OtherQED}
\renewcommand{\DefinitionQED}{\OtherQED}
\renewcommand{\NotationQED}{\OtherQED}
\renewcommand{\ExampleQED}{\OtherQED}
\renewcommand{\RemarkQED}{\OtherQED}

\renewcommand{\N}{{\mathbb N}}
\renewcommand{\Z}{{\mathbb Z}}
\renewcommand{\Q}{{\mathbb Q}}
\renewcommand{\R}{{\mathbb R}}

\newcommand{\edited}[1]{#1} 

\author{
  Aleksander Doan
  \and
  Eleny-Nicoleta Ionel
  \and
  Thomas Walpuski
}
\title{
  The Gopakumar--Vafa finiteness conjecture
}
\date{2025-11-29}

\begin{document}

\maketitle

\begin{abstract}
  The Gopakumar--Vafa conjecture predicts that the BPS invariants of a symplectic $6$--manifold, defined in terms of the Gromov--Witten invariants, are integers and all but finitely many vanish in every homology class.
  The integrality part of this conjecture was proved earlier by Ionel and Parker.
  This article proves the finiteness part.
  The proof relies on a modification of Ionel and Parker's cluster formalism using results from geometric measure theory.
\end{abstract}

\tableofcontents


\section{Introduction}
\label{Sec_Introduction}

Using ideas from $M$--theory
\citet{Gopakumar1998,Gopakumar1998a} predicted that there exist integer invariants $\BPS_{A,\genus}(X,\omega)$ associated with
a closed symplectic $6$--manifold $(X,\omega)$;
a Calabi--Yau class $A$, that is: $A \in \rH_2(X,\Z)$ with $c_1(A) \coloneq \Inner{c_1(X,\omega),A} = 0$; and 
$\genus \in \N_0$.
These invariants are interpreted in physics as the count of BPS states supported on $J$--holomorphic curves representing $A$ and of genus $\genus$.
\citeauthor{Gopakumar1998} conjectured that their invariants are related to the Gromov--Witten invariants $\GW_{A,\genus}(X,\omega)$ by the marvelous formula
\begin{equation}
  \label{Eq_GopakumarVafa}
  \sum_{A \in \Gamma}\sum_{\genus=0}^\infty
  \GW_{A,\genus}(X,\omega)
  \cdot
  t^{2\genus-2} q^A
  =
  \sum_{A \in \Gamma}\sum_{\genus=0}^\infty
  \BPS_{A,\genus}(X,\omega)
  \cdot
  \sum_{k=1}^\infty
  \frac1k
  \paren*{2\sin\paren*{kt/2}}^{2\genus-2}
  q^{kA}
\end{equation}
with $\Gamma \coloneq \set{ A \in \rH_2(X,\Z) : A \neq 0, c_1(A) = 0 }$;
see \cite[(3.2)]{Gopakumar1998a}.
This formula is to be understood as an equality of formal power series in variables $q^A$ whose coefficients are Laurent series in $t$.

\citeauthor{Gopakumar1998} did not give a direct mathematical definition of their invariants.
Indeed,
despite valiant efforts---especially by algebraic geometers \cite{Hosono2001, Pandharipande2009, Pandharipande2010, Kiem2012, Maulik2018}---mathematicians still do not know how to define them directly.
Turning the problem on its head and regarding \autoref{Eq_GopakumarVafa} as the \emph{definition} of $\BPS_{A,\genus}(X,\omega)$ led to the following conjecture.

\begin{conjecture}[{The Gopakumar--Vafa conjecture \cites{Gopakumar1998,Gopakumar1998a}[Conjecture 1.2]{Bryan2001}}]
  \label{Conj_GopakumarVafa}
  Let $(X,\omega)$ be a closed symplectic $6$--manifold.
  For every $A \in \rH_2(X,\Z)$ with $A \neq 0$ and $c_1(A) = 0$
  the numbers $\BPS_{A,\genus}(X,\omega)$ defined by \autoref{Eq_GopakumarVafa} satisfy:
  \begin{description}[labelwidth=\widthof{\bfseries (integrality)}]
  \item[(integrality)]
    $\BPS_{A,\genus}(X,\omega) \in \Z$ for every $\genus\in\N_0$.
  \item[(finiteness)]
    There is $\genus_A \in \N_0$ such that $\BPS_{A,\genus}(X,\omega) = 0$ for every $\genus\geq \genus_A$.
    \qedhere
  \end{description}
\end{conjecture}

\begin{remark}
  The integrality part of \autoref{Conj_GopakumarVafa} was proved by \citet{Ionel2018}.
\end{remark}

\begin{remark}
  There is an analogue of \autoref{Conj_GopakumarVafa} for Fano classes;
  that is: $A \in \rH_2(X,\Z)$ with $c_1(A) > 0$;
  see \autoref{Sec_GopakumarVafaFano}.
  This case is significantly easier because multiple covers can be avoided.
  \citet[Theorem 1.5]{Zinger2011} has proved integrality for Fano classes.
  \citet[Corollary 1.18]{Doan2019} have proved finiteness for Fano classes and primitive Calabi--Yau classes.
\end{remark}

\begin{remark}
  The finiteness part of \autoref{Conj_GopakumarVafa} implies that the coefficients of $q^A$ in the Gromov--Witten series \eqref{Eq_GopakumarVafa} are $t^{-2}$ times analytic functions of $t$ and rational functions of $u = - e^{it}$;
  cf.~\cite[Conjectures 3.2, 3.3, and 3.28]{Pandharipande2009}. 
\end{remark}

\begin{question}
  \label{Rmk_BPSCastelnuovo}
  \edited{
  Define \defined{BPS Castelnuovo number} associated with a class $A \in \Gamma$ by 
  \begin{align*}
    \gamma_A^\BPS(X,\omega)
    \coloneq
    \sup \set{
    	\genus \in \N_0 
    	:
    	\BPS_{A,\genus}(X,\omega) \neq 0
    }
  \end{align*}
  (and $\gamma_A^\BPS(X,\omega) \coloneqq -1$ if $\BPS_{A,g}(X,\omega)=0$ for all $g \in \N_0$.)
  This is an invariant of $(X,\omega)$.
  The finiteness part of \autoref{Conj_GopakumarVafa} implies that $\gamma_A^\BPS(X,\omega)<\infty $ for all $A \in \Gamma$.
  }
  It is interesting to ask:
  \emph{are there effective bounds on $\gamma_A^\BPS(X,\omega)$ analogous to Castelnuovo's bound for the genus of an irreducible degree $d$ curve in $\mathbb{P}^n$} \cites{Castelnuovo1889}[Chapter III Section 2]{Arbarello1985}?
\end{question}

The purpose of this article is to prove the \edited{finiteness part of the Gopakumar--Vafa conjecture.}

\begin{theorem}
  \label{Thm_GopakumarVafa}
  \autoref{Conj_GopakumarVafa} holds.
\end{theorem}

\edited{ 
The strategy of the proof is similar to that in \cite{Ionel2018}.
The new insight of this article is that Gromov's compactness theorem for $J$--holomorphic \emph{maps} can be replaced by the compactness theorem for $J$--holomorphic \emph{cycles} (i.e. currents), combined with  other geometric analysis results proved in \autoref{Sec_SpacesOfPseudoHolomorphicCurves}, such as a version of Allard's regularity theorem \cite{Allard1972}. 

To explain the new challenges and how we resolve them, let us first discuss the most natural approach to proving the Gopakumar--Vafa conjecture, inspired by Taubes' work on the Gromov invariants of symplectic $4$--manifolds. 
Gromov's compactness theorem together with Wendl's super-rigidity theorem \cite[Theorem A]{Wendl2016}  imply that for a generic $\omega$--tamed almost complex structure $J$ on $X$ the moduli space of simple $J$--holomorphic maps of bounded genus and energy is finite, and that all such maps are \emph{super-rigid}.
In fact, by replacing Gromov's theorem with the compactness theorem for pseudo-holomorphic cycles, \cite[Theorem 1.6]{Doan2018a} establishes the same result without a genus bound. 
This reduces the Gopakumar--Vafa conjecture to its local version for contributions from super-rigid curves.

However, such contributions depend on $J$ and are essentially impossible to calculate directly. 
When $J$ is \emph{elementary}, i.e. takes a special form around each $J$--holomorphic curve, the contribution can be computed by work of \citet{Pandharipande1999,Bryan2008,Lee2009a,Zinger2011}.
The Gopakumar--Vafa conjecture holds for elementary $J$.
Therefore, one can try to prove the conjecture for arbitrary $J$  by analyzing what happens to these contributions as $J$ is deformed to an elementary one.
This approach, similar to Taubes' work in dimension four, requires understanding the structure of codimension one walls in the space of $J$'s where super-rigidity fails, and bifurcation analysis of multiply covered $J$--holomorphic maps along paths of $J$'s crossing these walls. 
Unfortunately, this is, in general, a difficult problem. 
While this paper was under review,  \citet{Bai2021}, building on \cite{Wendl2016}, were able to analyze the wall-crossing caused by double covers of fixed genus,
\edited{and their results could be used to reprove the integrality part of the Gopakumar–Vafa conjecture for homology classes of divisibility two.}
Similar analysis for higher degree covers seems significantly more challenging.

The situation is even more complicated for the finiteness part of the conjecture, as bifurcation analysis such as \cite{Wendl2016,Bai2021} deals with curves of fixed genus. 
While for a generic $J$, the genus of embedded pseudo-holomorphic curves with bounded energy is bounded \cite[Theorem 1.6]{Doan2018a}, the proof does not generalize to families of $J$'s.
(Indeed, proving such a generalization is a difficult open problem, closely tied to the the question whether Allard's regularity theory can be extended to higher multiplicity currents.)
Without a bound on genus, it is possible that as $J$ is deformed in a generic $1$--parameter family, infinitely many embedded $J$--holomorphic curves in the same homology class appear or disappear in $X$. 
In such a situation, one cannot conclude the finiteness part of the conjecture from the computation for elementary $J$, even assuming that the bifurcation analysis in the spirit of \cite{Wendl2016,Bai2021} has been carried out for covers of arbitrary degree.

In this paper, as in  \cite{Ionel2018}, we take a different approach, entirely bypassing the bifurcation analysis of multiple covers, and working instead with the notion of a \emph{cluster} introduced in \cite{Ionel2018}. 
A cluster is a collection of $J$--holomorphic curves which are close to a given curve, called the \emph{core} of the cluster, and whose energy and genus are below a certain threshold.
Clusters are both open and closed subsets of the moduli space of $J$--holomorphic maps, thus have a well defined contribution to the Gromov--Witten invariants. 
Therefore, truncations in $q$ and $t$ of the Gromov--Witten series $\GW(X,\omega)$ appearing in \autoref{Eq_GopakumarVafa} can be decomposed into contributions of clusters, thus reducing the Gopakumar--Vafa conjecture to its local version for clusters. 
The conjecture holds for clusters which are \emph{elementary}, meaning that $J$ is elementary around the core curve.

In \cite{Ionel2018}, after fixing a truncation in $q$ and $t$ (i.e. a bound on energy and genus), the contributions of general clusters are shown to agree with those of the elementary clusters up to
contributions of clusters whose core is a curve of higher level.
The proof requires understanding deformations of $J$--holomorphic embeddings in a given homology class, but not their interactions with multiple covers, as higher level curves are ignored.
This allows \citeauthor{Ionel2018} to recursively prove integrality, but not finiteness because of the truncation in $t$.
This truncation is necessary as their cluster formalism relies on Gromov's compactness theorem for $J$--holomorphic maps and on the local wall-crossing model for the moduli space of such maps (around an embedding). 
%

In our proof, also inspired by Taubes' work on symplectic $4$--manifolds, we use the space of pseudo-holomorphic \emph{cycles}, equipped with the topology of \emph{geometric convergence}, to define clusters.
Since pseudo-holomorphic cycles do not have a specified genus, this allows us to drop the truncation in $t$.
We then use geometric measure theory to prove various topological properties of clusters in the space of cycles, such as compactness and openess, which are needed to carry out \citeauthor{Ionel2018}'s proof in this setting. 
In fact, we work with three different topological spaces containing pseudo-holomorphic embeddings: 
\begin{itemize}
	\item the space of pseudo-holomorphic {\em maps} with the Gromov topology,
	\item the space of pseudo-holomorphic {\em cycles} (i.e. currents) with the topology  of geometric convergence,
	\item the space of {\em compact subsets} with the Hausdorff distance.  
\end{itemize}
The relationship between these three spaces is rather subtle, as shown by various counterexamples discussed in \autoref{Sec_SpacesOfPseudoHolomorphicCurves} (cf.  \autoref{Rem.cont.inj.not.emb},  \autoref{Rem_ZeroAreaCurves} and  \autoref{Rem.Z.si.cont.inj.not.emb}).
The natural forgetful maps between them are continuous, but in general are not open, nor proper, nor injective. 
The desired properties of clusters, of which openess is the most difficult, are obtained by comparing these topologies when restricting to the subspaces of embedded curves. 
This is the content of  \autoref{Sec_SpacesOfPseudoHolomorphicCurves}, in particular \autoref{Prop_CurvesAreOpenInCycles}, \autoref{Thm_EmbeddedCurvesAreOpenInCycles}, and \autoref{Thm_SpaceOfEmbeddedCurvesEmbedsIntoSpaceOfSubmanifolds}.
The proofs of these results rely on geometric analysis arguments inspired by Allard's regularity theorem and White's regularity theorem for mean curvature flow.

%

Based on the results in  \autoref{Sec_SpacesOfPseudoHolomorphicCurves}, the upgraded cluster formalism is developed in \autoref{Sec_GromovWittenInvariants}, \autoref{Sec_Clusters}, and \autoref{Sec_ClusterIsotopyTheorem}.
Once it is in place,
\citeauthor{Ionel2018}'s argument proves both integrality and finiteness.
The technical part of the proof is ensuring that finiteness continues to hold in the wall-crossing formulae for the cluster contributions in a generic path of $J$s.
This is far from obvious: the naive approach using compactness for pseudo-holomorphic cycles, super-rigidity and bifurcation analysis fails to account for this, as we explained above.
However, the results of \autoref{Sec_SpacesOfPseudoHolomorphicCurves} allow us to conclude that as long as the homology class is fixed, there cannot be any sequence of higher genus pseudo-holomorphic curves converging in Hausdorff distance to an {\em embedded} pseudo-holomorphic curve, cf.  \autoref{Cor_ModuliSpaceOfEmbeddedMapsEmbedsIntoSets}. 
This is a delicate result, crucially depending on the fact that the limit is embedded and has multiplicity one when regarded as a cycle; it is false if the limit is singular or of higher multiplicity. 
It implies that if such a sequence geometrically converges to an embedded limit, then, in fact, it converges in the stronger Gromov topology. 
In that case, the Kuranishi local model in the space of embedded pseudo-holomorphic maps can be used to understand how the cluster contributions change as we vary the almost complex structure, and in particular prove finiteness for cluster deformations.}


For completeness' sake \autoref{Sec_GopakumarVafaFano} summarises the work of \citet[Theorem 1.5]{Zinger2011} and \citet[Corollary 1.18]{Doan2019} on the analogue of \autoref{Conj_GopakumarVafa} for Fano classes.
The theory developed in \autoref{Sec_SpacesOfPseudoHolomorphicCurves} allows for an alternative proof of \cite[Theorem 1.1]{Doan2019} as well as a partial strengthening of \cite[Theorem 1.6]{Doan2018a}.
This is discussed in \autoref{Rmk_NoGhosts} and \autoref{Sec_CastelnuovoOneParameterFamilies}.

\paragraph{Acknowledgements.}
We thank Camillo De Lellis for discussions about Allard's regularity theorem,
Jesse Madnick and Gavin Ball for pointing out \cite{Gray1965}, and
Robert Bryant for pointing out \cite{Hashimoto2004}. \edited{We also thank the referees for their meticulous reading of the paper and their numerous insightful suggestions.}
This material is based upon work supported by: 
\begin{itemize}
\item
  \href{https://www.simonsfoundation.org/simons-society-of-fellows/}{the Simons Society of Fellows} (AD);
\item
  \href{https://www.nsf.gov/awardsearch/showAward?AWD_ID=1905361&HistoricalAwards=false}{the National Science Foundation under Grant No.~1905361} (EI);
\item
  \href{https://www.nsf.gov/awardsearch/showAward?AWD_ID=1754967&HistoricalAwards=false}{the National Science Foundation under Grant No.~1754967},   
  \href{https://sloan.org/grant-detail/8651}{an Alfred P. Sloan Research Fellowship}, and
  \href{https://sites.duke.edu/scshgap/}{the Simons Collaboration on Special Holonomy in Geometry, Analysis, and Physics} (TW).
\end{itemize} 



\section{The space of pseudo-holomorphic cycles}
\label{Sec_SpacesOfPseudoHolomorphicCurves}

Throughout this section,
assume the following.

\begin{situation}
  \label{Sit_PseudoHolomorphicCycles}
  Let $X$ be a smooth manifold.
  A \defined{Hermitian structure} on $X$ is a pair $(J,g)$ consisting of an almost complex structure $J$ and a Riemannian metric $g$ with respect to which $J$ is orthogonal.
  Let $\SpaceOfHermitianStructures$ be a topological space of Hermitian structures $(J,g)$ on $X$ which are at least $C_\loc^3$.
  \edited{Suppose that the topology on $\SpaceOfHermitianStructures$ is metrizable and at least as fine as the $C_\loc^3$ topology.}
\end{situation}

\begin{example}
  \label{Ex_SymplecticAlmostHermitianStructures}
  If $(X,\omega)$ is a symplectic manifold,
  then there are two natural choices for $\SpaceOfHermitianStructures$:
  \begin{enumerate}
  \item
    \label{Ex_SymplecticAlmostHermitianStructures_Compatible}
    $\SpaceOfAlmostComplexStructures(\omega)$,
    the space of almost complex structures $J$ which are compatible with $\omega$;
    that is: $g \coloneq \omega(\cdot,J\cdot)$ defines a Riemannian metric.
  \item
    \label{Ex_SymplecticAlmostHermitianStructures_Tamed}
    $\SpaceOfAlmostComplexStructures_\tau(\omega)$,
    the space of almost complex structures $J$ which are tamed by $\omega$;
    that is: $g \coloneq \frac12\paren*{\omega(\cdot,J\cdot) - \omega(J\cdot,\cdot)}$ defines a Riemannian metric.
  \end{enumerate}
  In either case, $J$ is orthogonal with respect to $g$.
\end{example}

Denote by $\ModuliSpaceOfNodalMaps$ the space of pairs $(J,g;[u])$ consisting of $(J,g) \in \SpaceOfHermitianStructures$ and an equivalence class $[u]$ of stable nodal $J$--holomorphic maps \edited{equipped with the Gromov topology}.
Denote by $\ModuliSpaceOfSimpleMaps$ and 
$\ModuliSpaceOfEmbeddedMaps$ the subsets of those $(J,g;[u])$ with $u$ being simple and an embedding respectively.
Denote by $\pr_{\SpaceOfHermitianStructures} \co \ModuliSpaceOfNodalMaps \to \SpaceOfHermitianStructures$ the canonical projection map and by 
$\genus \co \ModuliSpaceOfNodalMaps \to \N_0$ and
$E \co \ModuliSpaceOfNodalMaps \to \edited{[0,\infty)}$
the maps which assign to a nodal pseudo-holomorphic map its \edited{arithmetic} genus and energy.
Gromov's compactness theorem asserts that if $X$ is compact,
then the map
\begin{equation*}
  (\pr_\SpaceOfHermitianStructures,\genus,E) \co \ModuliSpaceOfNodalMaps \to \SpaceOfHermitianStructures \times \N_0 \times [0,\infty)
\end{equation*}
is proper \cites[§1]{Gromov1985}{Hummel1997}.
\emph{The genus component is crucial};
indeed:
the map $(\pr_\SpaceOfHermitianStructures,E) \co  \ModuliSpaceOfNodalMaps \to \SpaceOfHermitianStructures \times [0,\infty)$ fails to be proper.
A trivial reason for the failure properness are ghosts components;
that is: components of the domain of a nodal map on which the map is constant.
Evidently, there are ghosts components of arbitrary genus.
A more interesting reason for the failure properness are multiple covers.
If $u \co (\Sigma,j) \to (X,J)$ is a pseudo-holomorphic map and $\pi \co (\tilde \Sigma,\tilde j) \to (\Sigma,j)$ is a branched cover,
then $u \circ \pi$ is pseudo-holomorphic and $E(u \circ \pi) = \deg(\pi) \cdot E(u)$.
Furthermore, for every $d \geq 2$ and $\genus_0 \in \N$ there is a branched cover with $\deg(\pi) = d$ and $\genus(\tilde \Sigma) \geq \genus_0$.
In either case, the unboundedness of the genus is not reflected in the subsets $\im u$ parametrized by $[u]$.

These issues can be partially resolved by considering pseudo-holomorphic cycles instead of pseudo-holomorphic maps.
The purpose of this section is to summarize the salient parts of the theory of pseudo-holomorphic cycles and add to it a few observations,
which might appear to be minor but are crucial for the proof of the Gopakumar--Vafa finiteness conjecture.
\edited{The main results of this section are:
\begin{itemize}
  \item \autoref{Thm_SpaceOfCycles_Proper}, a compactness theorem for pseudo-holomorphic cycles,
  \item \autoref{Prop_CurvesAreOpenInCycles} and \autoref{Thm_EmbeddedCurvesAreOpenInCycles},  which assert that the subsets of pseudo-holomorphic curves and embedded pseudo-holomorphic curves  are open in the space of pseudo-holomorphic cycles, 
  \item \autoref{Thm_SpaceOfEmbeddedCurvesEmbedsIntoSpaceOfSubmanifolds} and \autoref{Cor_ModuliSpaceOfEmbeddedMapsEmbedsIntoSpaceOfSubmanifolds}, which assert  that on the subset of embedded pseudo-holomorphic curves the topology of geometric convergence of cycles agrees with that of $C^1$ convergence.
\end{itemize}
In addition, these results are used to prove \autoref{Prop_SpaceOfCurvesEmbedsIntoSpaceOfCompactSubsets}, and \autoref{Cor_ModuliSpaceOfEmbeddedMapsEmbedsIntoSets}, which are crucial in the proof of the Gopakumar--Vafa finiteness conjecture in \autoref{Sec_GopakumarVafaConjecture}.
}

\subsection{Definitions and results}
\label{Subsec_DefinifionsAndResults}

\begin{definition}
  \label{Def_HausdorffMetric}
  Denote by $\SpaceOfCompactSubsets$ the set of compact subsets of $X$.
  For $(J,g) \in \SpaceOfHermitianStructures$ denote by $d \co X \times X \to [0,\infty)$ the metric induced by $g$.
  The \defined{Hausdorff metric} $d_H \co \SpaceOfCompactSubsets \times \SpaceOfCompactSubsets \to [0,\infty]$ is defined by
  \begin{equation*}
    d_H(A,B)
    \coloneq
    \max\set*{
      \adjustlimits \sup_{x \in A} \inf_{y \in B} d(x,y),
      \adjustlimits \sup_{x \in B} \inf_{y \in A} d(x,y)
    }.
    \qedhere
  \end{equation*}
\end{definition}

\begin{theorem}[{\citeauthor{Blaschke1956} \cites{Blaschke1956}[Theorem 7.3.8]{Burago2001}}]
  \label{Thm_Blaschke}
  If $(X,d)$ is compact,
  then so is $(\SpaceOfCompactSubsets,d_H)$.
\end{theorem}

\begin{remark}
  The topology induced by $d_H$ depends only the \edited{topology of $X$ induced by $d$, that is: if $d_H'$ is defined in terms of a metric $d'$ on $X$ which induces the same topology as $d$,
  then $d_H$ and $d_H'$ induce the same topology on $\SpaceOfCompactSubsets$.}  
\end{remark}

The following notion of pseudo-holomorphic cycles and their geometric convergence goes back to \citet{Taubes1996b}.

\begin{definition}
  \label{Def_JHolomorphicCycles}
  \begin{enumerate}
  \item
    Let $(J,g) \in \SpaceOfHermitianStructures$.
    An \defined{irreducible $J$--holomorphic curve} is a subset $C \subset X$ which is the image of a simple $J$--holomorphic map $u \co (\Sigma,j) \to X$ from a connected, closed Riemann surface.
    \label{Def_JHolomorphicCycle}
    A \defined{$J$--holomorphic cycle} $C$ is a formal finite sum
    \begin{equation}
      \label{Eq_JHolomorphicCycle_Decomposition}
      C = \sum_{i=1}^I m_i C_i
    \end{equation}
    of distinct irreducible $J$--holomorphic curves $C_1,\ldots,C_I$ with coefficients $m_1,\ldots,m_I \in \N$.
  \item
    \label{Def_JHolomorphicCycles_Support}
    Let $(J,g) \in \SpaceOfHermitianStructures$.
    Let $C$ be a $J$--holomorphic cycle. 
    The \defined{support of $C$} and the \defined{current associated with $C$} are the closed subset $\supp C$ and the linear map $\delta_C \co \Omega_c^2(X) \to \R$ defined by
    \begin{equation*}
      \supp C \coloneq  \bigcup_{i=1}^I C_i
      \qandq
      \delta_C(\alpha) \coloneq \sum_{i=1}^I m_i \int_{\Sigma_i} u_i^*\alpha.
    \end{equation*}
  \item
    \label{Def_JHolomorphicCycles_Topology}
    Denote by $\SpaceOfCycles$ the set of pairs consisting of an almost Hermitian structure $(J,g) \in \SpaceOfHermitianStructures$ and a $J$--holomorphic cycle $C$ in $X$.
    The \defined{geometric \edited{convergence} topology} on $\SpaceOfCycles$ is the coarsest topology with respect to which the maps
    \begin{equation*}
      \pr_\SpaceOfHermitianStructures \co \SpaceOfCycles \to \SpaceOfHermitianStructures, \quad
      \supp\co \SpaceOfCycles \to \SpaceOfCompactSubsets, \qandq
      \delta\co \SpaceOfCycles \to \Hom(\Omega_c^2(X),\R)
    \end{equation*}
    are continuous.
    \edited{Here $\Omega_c^2(X)$ denotes the space of compactly supported smooth $2$--forms on $X$ with the topology of smooth convergence over compact sets and $\Hom(\Omega_c^2(X),\R)$ is equipped with the weak--$*$ topology; see, e.g., \cite[§6.1, 6.2]{Simon1983}.}
  \item
    \label{Def_JHolomorphicCycles_Mass}
    Let $(J,g;C) \in \SpaceOfCycles$.
    The \defined{mass of $C$} with respect to $g$ is
    \begin{equation*}
      \bM(C) = \bM_g(C) \coloneq \sum_{i=1}^I m_i \area_g(C_i).
    \end{equation*}
    The \defined{homology class} of $C$ is
    \begin{equation*}
      [C] \coloneq \sum_{i=1}^I m_i [C_i]
      \qwithq
      [C_i] \coloneq \paren{u_i}_*[\Sigma_i] \in \rH_2(X,\Z).
      \qedhere
    \end{equation*}
  \end{enumerate}
\end{definition}


\begin{remark}
  If $X$ is compact, then $\supp$ can be dropped from \autoref{Def_JHolomorphicCycles_Topology} because its continuity follows from the monotonicity formula;
  see, e.g., \cite[Lemma 5.6]{Doan2018a}.
\end{remark}

\begin{remark}
  \label{Rmk_TamedAlmostHermitianStructures}
  Every $(J,g) \in \SpaceOfHermitianStructures$ defines a \defined{Hermitian form} $2$--form $\sigma(\cdot,\cdot) \coloneq g(J\cdot,\cdot)$.
  It defines a semi-calibration.
  If $C$ is $J$--holomorphic,
  then $\delta_C$ is semi-calibrated by $\sigma$.  
  \edited{In particular, $\bM_g(C) = \delta_C(\sigma)$.
    
    For every $\epsilon > 0$, 
    if $(\tilde J,\tilde g)$ is sufficiently $C^0$--close to $(J,g)$ on a compact subset $K \subset X$,
    then for every $x \in K$ and every $J_x$--invariant linear subspace $L \subset T_xX$ 
    $\abs*{\frac{\tilde\sigma|_L}{\sigma|_L} - 1} < \epsilon$;
    in particular: if $C$ is $J$--holomorphic and $\supp C \subset K$,
    then $\bM_g(C) \leq (1+\epsilon)\delta_C(\tilde \sigma)$.
    As a consequence of this observation,
    if $(J_n,g_n;C_n) \in \SpaceOfCycles^\N$ converges to $(J,g;C)$,
    then $\lim_{n \to \infty} \bM_{g_n}(C_n) = \bM(C)$.
    Therefore, $\bM$ is a continuous function on $\SpaceOfCycles$.
  }
 
  If $(X,\omega)$ is a symplectic manifold and $\SpaceOfHermitianStructures = \SpaceOfAlmostComplexStructures(\omega)$ as in \autoref{Ex_SymplecticAlmostHermitianStructures}~\autoref{Ex_SymplecticAlmostHermitianStructures_Compatible},
  then $\sigma = \omega$.
  Therefore, it is a calibration.
  If $\SpaceOfHermitianStructures = \SpaceOfAlmostComplexStructures_\tau(\omega)$ as in \autoref{Ex_SymplecticAlmostHermitianStructures}~\autoref{Ex_SymplecticAlmostHermitianStructures_Tamed},
  then $\sigma = \frac12\paren*{\omega + \omega(J\cdot,J,\cdot)}$ and need not be a calibration;
  nevertheless: $\bM(C) = \Inner{[\omega],[C]}$;
  cf.~\cite[Lemma 2.2.1]{McDuff2012}.
  Therefore, $[C]$ determines $\bM(C)$.
\end{remark}

In light of the following,
arguments regarding the geometric convergence topology on $\SpaceOfCycles$ can be carried out using sequences.
\edited{%
In fact, the proof below shows that pseudo-holomorphic cycles can be regarded as compactly supported integral currents and the notion of geometric convergence agrees with that of convergence in the corresponding flat topology for currents.

\begin{prop}
  \label{Prop_SpaceOfCycles_Metrizable}
  $\SpaceOfCycles$ is metrizable. 
\end{prop}

\begin{proof}
  By the definition of the geometric convergence topology on $\SpaceOfCycles$,
  its suffices to prove that the weak $*$–topology on the space of integral pseudo-holomorphic $2$--currents is metrizable.
  
  The map $\delta$ factors through the subspace $\bI_{2,c}(X) \subset \Hom(\Omega_c^2(X),\bR)$ of closed integral currents of dimension two with compact support;
  see, e.g., \cites[§4.1.24]{Federer1969}[§27]{Simon1983}.  
  The \defined{support} of an integral current $T \in \bI_2(X)$ is the smallest closed subset $\supp(T) \subset X$ such that $T(\alpha) = 0$ for every $\alpha \in \Omega_c^2(X)$ with $\supp(\alpha) \cap \supp(T) = \emptyset$.
  This extends $\supp$ from \autoref{Def_JHolomorphicCycles}~\autoref{Def_JHolomorphicCycles_Support}.
  The \defined{mass} of an integral current $T \in \bI_2(X)$ is defined by
  \begin{equation*}
    \bM(T) \coloneq \sup\set{ T(\alpha) : \alpha \in \Omega_c^2(X) ~\text{with}~ \Abs{\alpha} \leq 1 }.
  \end{equation*}
  Here $\Abs{\cdot}$ denotes the comass norm.
  By Wirtinger's inequality \cite{Wirtinger1936},
  this extends $\bM$ from \autoref{Def_JHolomorphicCycles}~\autoref{Def_JHolomorphicCycles_Mass}.
  Observe that $\bM$ is lower semi-continuous with respect to the weak--$*$ topology.

  For every compact subset $K \subset X$, define the seminorm $\bF_K \co \bI_{2,c}(X) \to [0,\infty]$ by
  \begin{equation*}
    \bF_K(T) \coloneq \sup\set*{ T(\alpha) : \alpha \in \Omega_c^2(X) ~\text{with}~ \sup_{K} \Abs{\alpha} \leq 1, ~\text{and}~ \sup_{K} \Abs{\rd \alpha} \leq 1 }.
  \end{equation*}
  Observe that $\bF_K(T) < \infty$ implies $\supp(T) \subset K$. 
  The \defined{flat topology} $\cO_\flat$ on $\bI_{2,c}(X)$ is generated by these seminorms. In particular, $\cO_\flat$ is metrizable.
  Denote by $\cO_w$ the weak--$*$ topology.
  Observe that $\cO_w $ is Hausdorff. 

  It suffices to show that the two topologies  $\cO_\flat$ and $\cO_w $ agree on $\bI_{2,c}(X)$.
  Since the identity $\id \co (\bI_{2,c}(X),\cO_\flat) \to (\bI_{2,c}(X),\cO_w)$ is continuous, it remains to prove it is closed.
  To prove this we combine various results from geometric measure theory with the fact that a proper continuous map $f \co X \to Y$ between two topological spaces is closed whenever $Y$ is locally compact and Hausdorff.

  Consider the subsets 
  \begin{equation*}
    \bI_{2,c}^{K,M_0}(X)
    \coloneqq
    \set{
      T \in \bI_{2,c}(X)
      :
      \supp(T) \subset K,
      \bM(T) \leq M_0
    }
  \end{equation*}
  with $K \subset X$ compact and $M_0 \geq 0$.
  The Federer--Fleming Compactness Theorem for integral currents \cites{Federer1960}{White1989}[Theorem 27.3]{Simon1983} together with \autoref{Thm_Blaschke} implies that $ \bI_{2,c}^{K,M_0}(X)$ is sequentially-compact with respect to $\cO_w$. Indeed,
  the notions of convergence of sequences in $\bI_{2,c}^{K,M_0}(X)$ with respect to $\cO_\flat$ and $\cO_w$ agree \cite[§31]{Simon1983}.  
  Thus $\bI_{2,c}^{K,M_0}(X)$ is sequentially-compact with respect to $\cO_\flat$, and thus compact (since $\cO_\flat$ is metrizable).

  Since $\id$ is continuous,
  $\bI_{2,c}^{K,M_0}(X)$ is compact with respect to $\cO_w$ as well.
  Moreover, it is a neighborhood of $T\in \bI_{2,c}(X)$ provided $K$ is a compact neighborhood of $\supp(T)$ and $M_0 > \bM(T) $.
  Therefore, $\bI_{2,c}(X)$ is locally compact with respect to $\sO_w$.

  It remains to prove that $\id$ is proper.
  Suppose that $\bK \subset \bI_{2,c}(X)$ is compact with respect to $\cO_w$.
  Since $\cO_w$ is Hausdorff,
  $\bK$ is closed with respect to $\cO_w$ and, by continuity of $\id$,
  also with respect to $\sO_\flat$.  
  By continuity of $\supp$ and lower semi-continuity of $\bM$,
  there are $K \subset X$ and $M_0 \geq 0$ with $\bK \subset \bI_{2,c}^{K,M_0}(X)$.
  Since $\bI_{2,c}^{K,M_0}(X)$ is compact and $\bK$ is closed with respect to $\cO_\flat$,
  $\bK$ is compact with respect to $\cO_\flat$.
\end{proof}
}

The Federer--Fleming Compactness Theorem for integral currents
and the regularity theory for $2$--dimensional semi-calibrated integral currents developed by \citet{DeLellis2017b,DeLellis2017} lead to the following compactness theorem for pseudo-holomorphic cycles.

\begin{theorem}[{\cite[Proposition 1.9]{Doan2018a}}]
  \label{Thm_SpaceOfCycles_Proper}
  The map
  \begin{equation*}
    (\pr_\SpaceOfHermitianStructures,\supp,\bM) \co \SpaceOfCycles \to \SpaceOfHermitianStructures \times \SpaceOfCompactSubsets \times [0,\infty)
  \end{equation*}
  is continuous and proper.
  \qed
\end{theorem}

\begin{remark}
  If $\SpaceOfHermitianStructures = \SpaceOfAlmostComplexStructures(\omega)$ is as in \autoref{Ex_SymplecticAlmostHermitianStructures}~\autoref{Ex_SymplecticAlmostHermitianStructures_Compatible},
  then the proof of \autoref{Thm_SpaceOfCycles_Proper} can be based---instead of \cite{DeLellis2017b,DeLellis2017}---on the earlier work of \citet{Riviere2009} and, in dimension four, on the seminal work of \citet{Taubes1996}.
\end{remark}

\begin{remark}
  By \autoref{Rmk_TamedAlmostHermitianStructures} in either case of \autoref{Ex_SymplecticAlmostHermitianStructures} the map $\bM \co \SpaceOfCycles \to [0,\infty)$ in \autoref{Thm_SpaceOfCycles_Proper} can be replaced by $[\cdot] \co \SpaceOfCycles \to \rH_2(X,\Z)$. 
\end{remark}

To understand the relation between Gromov's compactness theorem and \autoref{Thm_SpaceOfCycles_Proper} it is enlightening to introduce the following map.

\begin{definition}
  \label{Def_MapToCycle}
  Define the map $\fz\co \ModuliSpaceOfNodalMaps \to \SpaceOfCycles$ by
  \begin{equation*}
    \fz(J,\edited{g},[u]) \coloneq \paren*{J,\edited{g}, C} 
    \qwithq C \coloneq \sum_{i=1}^I \deg \pi_i \cdot\im v_i.
  \end{equation*}
  Here $[u_1],\ldots,[u_I]$ denote the \edited{non-constant} irreducible components of $[u]$ and, 
  for every $i = 1,\ldots,I$, $u_i = v_i \circ \pi_i$ with $v_i$ a simple $J$-holomorphic map.
\end{definition}

\edited{The map $\fz$ is continuous. 
Indeed, if a sequence of stable $J_n$--holomorphic maps $(J_n,u_n)$ converges in Gromov topology to $(J, u)$, then $(u_n)$ converges to $u$ uniformly in $C^\infty_\loc$ away from the nodes of $u$, with exponential decay in the neck regions around the nodes, and $(\im u_n)$ converges to $\im u$ in Hausdorff distance. 
It follows that $(C_n=\fz(u_n))$ geometrically converges to $C = \fz(u)$ when regarded as pseudo-holomorphic cycles.
However, $\fz$ is not proper.}
The failure of properness again is due to ghosts components and branched covers with the same degree but different numbers of ramification points.

The following definitions and results concern certain important subsets of the space of cycles.

\begin{definition}
  \label{Def_PseudoHolomorphicCurve}
  \begin{enumerate}
  \item
    Let $(J,g) \in \SpaceOfHermitianStructures$.
    A \defined{$J$--holomorphic curve} is a $J$--holomorphic cycle all of whose multiplicities $m_i$ in \autoref{Eq_JHolomorphicCycle_Decomposition} are equal to one.
    Set
    \begin{equation*}
      \SpaceOfCurves
      \coloneq
      \set{ (J,g;C) \in \SpaceOfCycles : C ~\text{is a $J$--holomorphic curve}}.
    \end{equation*}
  \item
    Let $(J,g) \in \SpaceOfHermitianStructures$.
    A $J$--holomorphic curve $C$ is \defined{embedded} if its components $C_i$ in \autoref{Eq_JHolomorphicCycle_Decomposition} are disjoint and embedded.
    Set
    \begin{equation*}
      \SpaceOfEmbeddedCurves
      \coloneq
      \set{ (J,g;C) \in \SpaceOfCurves : C ~\text{is embedded}}.
    \end{equation*}
  \item
    \label{Def_PseudoHolomorphicCurve_Connected}
    Set
    \begin{gather*}
      \SpaceOfConnectedCycles
      \coloneq
      \set{
        (J,g;C) \in \SpaceOfCycles
        :
        \supp C ~\text{is connected}
      }, \\
      \SpaceOfConnectedCurves \coloneq \SpaceOfConnectedCycles \cap \SpaceOfCurves,
      \qandq
      \SpaceOfConnectedEmbeddedCurves \coloneq \SpaceOfConnectedCycles \cap \SpaceOfEmbeddedCurves.
      \qedhere
    \end{gather*}
  \end{enumerate}
\end{definition}

\begin{remark}
  A moment's thought shows that $\SpaceOfConnectedCycles = \im \fz$.
\end{remark}

Since the subset of connected, compact subsets is closed in $\SpaceOfCompactSubsets$,
$\SpaceOfConnectedCycles$ is closed in $\SpaceOfCycles$.

\begin{prop}
  \label{Prop_CurvesAreOpenInCycles}
  $\SpaceOfCurves$ is open in $\SpaceOfCycles$.
\end{prop}

\begin{theorem}
  \label{Thm_EmbeddedCurvesAreOpenInCycles}
  $\SpaceOfConnectedEmbeddedCurves$ and $\SpaceOfEmbeddedCurves$ are open in $\SpaceOfCurves$ (and, therefore, in $\SpaceOfCycles$).
\end{theorem}

The proof of \autoref{Prop_CurvesAreOpenInCycles} requires the monotonicity formula and is discussed in \autoref{Sec_MonotonicityFormula}.
\autoref{Thm_EmbeddedCurvesAreOpenInCycles} is proved in \autoref{Sec_AllardRegularity}---%
using Allard's regularity theorem \cite{Allard1972} and an observation due to \citet{Gray1965}.

The following results compare the geometric convergence topology on $\SpaceOfEmbeddedCurves$ with the $C^1$ topology.

\begin{definition}
  \begin{enumerate}
  \item
    Denote by $\SpaceOfSubmanifolds$ the set of $C^2$ submanifolds of $X$.
  \item
    Let $S \in \SpaceOfSubmanifolds$.
    A \defined{tubular neighborhood of $S$} consists of an open neighborhood $U \edited{\subset NS}$ of the zero section in the normal bundle $NS$,
    an open neighborhood $V$ of $S$ in $X$, and
    a $C^1$ diffeomorphism $\jmath \co U \to V$ which restricts to identity along the zero section. 
  \item
    For $S \in \SpaceOfSubmanifolds$, a tubular neighborhood $\jmath \co U \to V$ of $S$, and $\epsilon > 0$
    set
    \begin{equation*}
      \cU(S,\jmath,\epsilon)
      \coloneq
      \set*{
        \jmath(\graph{\xi})
        :
        \xi \in \Gamma(NS)
        ~\textnormal{with}~
        \im \xi \subset U \textnormal{ and }
        \Abs{\xi}_{C^1} < \epsilon
      }.
    \end{equation*}
    The \defined{$C^1$ topology on $\SpaceOfSubmanifolds$} is the coarsest topology with respect to which the subsets $\cU(S,\jmath,\epsilon)$ are open.
    \edited{Here the $C^1$ norm is with respect to data induced by some choice of Riemannian metric on $X$.}
    \qedhere
  \end{enumerate}
\end{definition}

\begin{theorem}
  \label{Thm_SpaceOfEmbeddedCurvesEmbedsIntoSpaceOfSubmanifolds}
  The map $(\pr_\SpaceOfHermitianStructures,\supp) \co \SpaceOfEmbeddedCurves \to \SpaceOfHermitianStructures \times \SpaceOfSubmanifolds$ is an embedding.
\end{theorem}

The proof is presented in \autoref{Sec_ConvergenceOfEmbeddedCurves};
it is based on an observation due to \citet{White2005}.

\begin{prop}
  \label{Cor_ModuliSpaceOfEmbeddedMapsEmbedsIntoSpaceOfSubmanifolds}
  The map $\fz \co \ModuliSpaceOfEmbeddedMaps \to \SpaceOfCycles$ is an open embedding;
  its image is $\SpaceOfConnectedEmbeddedCurves$.
  In particular,
  the Gromov topology on $\ModuliSpaceOfEmbeddedMaps$ agrees with the geometric convergence topology on $\SpaceOfConnectedEmbeddedCurves$.
\end{prop}

\begin{proof}
  Evidently,
  the image of $\fz \co \ModuliSpaceOfEmbeddedMaps \to \SpaceOfCycles$ is $\SpaceOfConnectedEmbeddedCurves$.
  By \autoref{Prop_CurvesAreOpenInCycles} and \autoref{Thm_EmbeddedCurvesAreOpenInCycles},
  the latter is open in $\SpaceOfCycles$. 
  Since the Gromov topology on $\ModuliSpaceOfEmbeddedMaps$ agrees with the $C^1$ topology on the space of maps,
  the composition
  \begin{equation*}
    \ModuliSpaceOfEmbeddedMaps
    \xrightarrow{\fz}
    \SpaceOfEmbeddedCurves
    \xrightarrow{(\pr_\SpaceOfHermitianStructures,\supp)}
    \SpaceOfHermitianStructures \times \SpaceOfSubmanifolds
  \end{equation*}
  is an embedding; that is: a homeomorphism on its image.
  Therefore, by \autoref{Thm_SpaceOfEmbeddedCurvesEmbedsIntoSpaceOfSubmanifolds},
  $\fz$ is an embedding.
\end{proof}

\begin{remark}\label{Rem.cont.inj.not.emb} 
  The reader should be warned that the map $\fz \co \ModuliSpaceOfSimpleMaps \to \SpaceOfCurves$ is a continuous injection but fails to be an embedding.
  To see this,
  consider a sequence $(u_n \co (\Sigma,j) \to X)$ of simple $J$--holomorphic maps which Gromov converges to a nodal $J$--holomorphic map $u \co \widehat \Sigma \to X$ with
  $\widehat \Sigma = \Sigma \vee S^2$ such that $u|_\Sigma$ is constant and $v \coloneq u|_{S^2}$ is simple.
  The sequence of $J$--holomorphic curves $(\im u_n) \in (\SpaceOfCurves)^\N$ geometrically converges to $\im v$;
  however, $(u_n)$ does not converge to $v$.
  By \autoref{Cor_ModuliSpaceOfEmbeddedMapsEmbedsIntoSpaceOfSubmanifolds},
  $v$ cannot be an embedding.
  Indeed, this can also be proved by analyzing the obstruction map in the Kuranishi model of a neighborhood of $[u] \in \ModuliSpaceOfNodalMaps$;
  cf.~\cites{Ionel1998,Zinger2009,Doan2019} 
  \edited{or by different means \cite{Bai2021,EkholmShende2022}.}
\end{remark}

The following result compares the geometric convergence topology on $\SpaceOfCurves$ with the topology induced by the Hausdorff metric.

\begin{definition}
  For $A \in \rH_2(X,\Z)$ and $\Lambda > 0$ set
  \begin{equation*}
    \SpaceOfCurves_{A,\Lambda}
    =
    \set{ (J,g;C) \in \SpaceOfCurves : [C] = A ~\textnormal{and}~ \bM(C) \leq \Lambda };
  \end{equation*}
  furthermore,
  in either situation of \autoref{Ex_SymplecticAlmostHermitianStructures},
  abbreviate
  \begin{equation*}
    \SpaceOfCurves_A \coloneq \SpaceOfCurves_{A,\Lambda}
    \qwithq
    \Lambda = \Inner{[\omega],A}.
    \qedhere
  \end{equation*}
\end{definition}

\begin{prop}
  \label{Prop_SpaceOfCurvesEmbedsIntoSpaceOfCompactSubsets}
  If there exists no $(J,g;C) \in \SpaceOfCycles$ with $[C] = 0$ but $C \neq 0$,  
  then the map $(\pr_\SpaceOfHermitianStructures,\supp) \co \SpaceOfCurves_{A,\Lambda} \to \SpaceOfHermitianStructures \times \SpaceOfCompactSubsets$ is an embedding.
  In particular,
  the geometric convergence topology on $\SpaceOfCurves_{A,\Lambda}$ agrees with the topology induced by the Hausdorff metric.
\end{prop}

\begin{remark}
  The hypothesis of \autoref{Prop_SpaceOfCurvesEmbedsIntoSpaceOfCompactSubsets} holds in either situation of \autoref{Ex_SymplecticAlmostHermitianStructures};
  cf.~\autoref{Rmk_TamedAlmostHermitianStructures}.
\end{remark}

\begin{remark}
  \label{Rem_ZeroAreaCurves}
  The hypothesis of \autoref{Prop_SpaceOfCurvesEmbedsIntoSpaceOfCompactSubsets} is necessary.
  Consider $S^6$ with the almost Hermitian structure $(J,g)$ induced by the octonions;
  \edited{that is: by regarding $S^6 \subset \Im \bO \subset \bO$ with $J$ at $x \in S^6$ given by octonionic left-multiplication by $x$ and $g$ induced by the standard Euclidean metric.}  
  Choose a sequence of distinct geodesic $J$--holomorphic $2$--spheres $(S_n)$ converging to a $J$--holomorphic geodesic $2$--sphere $S$.
  $(S_n\amalg S)$ converges to $S$ with respect to the Hausdorff metric,
  but $(S_n + S)$ does not geometrically converge to $S$: it geometrically converges to $2S \notin \SpaceOfCurves$.
  This issue also occurs with irreducible $J$--holomorphic curves;
  cf.~\citet{Hashimoto2004}.
\end{remark}

\begin{remark}\label{Rem.Z.si.cont.inj.not.emb} 
  The reader should be warned that the map $(\pr_\SpaceOfHermitianStructures,\supp) \co \SpaceOfCurves \to \SpaceOfHermitianStructures \times \SpaceOfCompactSubsets$ is a continuous injection but fails to be an embedding.
  To see this,
  consider a sequence of pseudo-holomorphic curves $(C_n)$ geometrically converging to a pseudo-holomorphic cycle $m C$ with $m \geq 2$.
  The sequence $(\supp C_n)$ converges to $\supp C$, but $(C_n)$ does not geometrically converge to $C$.
\end{remark}

\begin{proof}[Proof of \autoref{Prop_SpaceOfCurvesEmbedsIntoSpaceOfCompactSubsets}]
  The map $(\pr_\SpaceOfHermitianStructures,\supp)$ is continuous and injective \cite[Proposition 2.4.4, Corollary 2.5.3, Theorem E.1.2]{McDuff2012}.
  To prove that it is an embedding,
  let $(J_n,g_n;C_n) \in \paren{\SpaceOfCurves_{A,\Lambda}}^\N$ be such that $(J_n,g_n;\supp C_n)$ converges to $(J,g;\supp C)$ with $(J,g;C)\in \SpaceOfCurves_{A,\Lambda}$.
  By \autoref{Thm_SpaceOfCycles_Proper},
  $(J_n,g_n;C_n)$ converges to $(J,g;C') \in \SpaceOfCycles$.
  By continuity, $\supp C' = \supp C$.
  Therefore, if $C = \sum_{i = 1}^I C_i$,
  then $C' = \sum_{i = 1}^I m_i C_i$ with $m_1,\ldots,m_I \in \N$.
  Since $[C'] = A = [C]$, 
  and by the hypothesis, $m_1 = \ldots = m_I = 1$;
  hence: $C' = C$.
\end{proof}

Finally, here is a partial summary of the above results in the symplectic setting.

\begin{definition}
  For $A \in \rH_2(X,\Z)$ and $\genus \in \N_0$ set
  \begin{align*}
    \SpaceOfConnectedEmbeddedCurves_A
    &\coloneq
      \set{
      (J,h;C) \in \SpaceOfConnectedEmbeddedCurves : [C] = A
      }, \\     
    \ModuliSpaceOfEmbeddedMaps_A
    &\coloneq
      \set{
      (J,h;[u \co \Sigma \to X]) \in \ModuliSpaceOfEmbeddedMaps : u_*[\Sigma] = A
      },
      \qand
    \\
    \ModuliSpaceOfEmbeddedMaps_{A,\genus}
    &\coloneq
      \set{
      (J,h;[u \co \Sigma \to X]) \in \ModuliSpaceOfEmbeddedMaps : u_*[\Sigma] = A, \genus(\Sigma) = \genus
      }.
      \qedhere
  \end{align*}
\end{definition}

\begin{prop}
  \label{Cor_ModuliSpaceOfEmbeddedMapsEmbedsIntoSets}
  If $(X,\omega)$ is symplectic and $\SpaceOfHermitianStructures = \SpaceOfAlmostComplexStructures_\tau$ as in \autoref{Ex_SymplecticAlmostHermitianStructures}~\autoref{Ex_SymplecticAlmostHermitianStructures_Tamed},
  then for every $A \in \rH_2(X,\Z)$ the map
  \begin{equation*} 
    (\pr_\SpaceOfAlmostComplexStructures,\im) \co \ModuliSpaceOfEmbeddedMaps_A \to \SpaceOfConnectedEmbeddedCurves_A
  \end{equation*} 
  is a homeomorphism and $ \SpaceOfConnectedEmbeddedCurves_A$ is open in $\SpaceOfCycles$.
  In particular:
  \begin{enumerate}
  \item
    The Gromov topology on $\ModuliSpaceOfEmbeddedMaps_A$ agrees with the geometric convergence topology as well as with the topology induced by the Hausdorff metric.
  \item
    If $C$ is an irreducible, embedded $J$--holomorphic curve representing $A$ and of genus $\genus$,
    then there is an open neighborhood of $(J,C) \in \SpaceOfAlmostComplexStructures_\tau(\omega) \times \SpaceOfCompactSubsets$ which contains no other images of pseudo-holomorphic cycles representing $A$ except for those in the image of $\ModuliSpaceOfEmbeddedMaps_{A,\genus}$.
    \qedhere
  \end{enumerate}
\end{prop}

The remainder of this section contains the proofs of \autoref{Prop_CurvesAreOpenInCycles}, \autoref{Thm_EmbeddedCurvesAreOpenInCycles}, and \autoref{Thm_SpaceOfEmbeddedCurvesEmbedsIntoSpaceOfSubmanifolds}.
A reader who is solely interested in the applications of these results to symplectic geometry might proceed to the next section.

\subsection{The monotonicity formula}
\label{Sec_MonotonicityFormula}

The proofs of \autoref{Prop_CurvesAreOpenInCycles} and \autoref{Thm_SpaceOfEmbeddedCurvesEmbedsIntoSpaceOfSubmanifolds} require the following monotonicity formula.
This result is standard and can be derived from \cite[Theorem 2.1]{DeLellis2018a} and \cite[Proposition 5.3]{Gray1965}.
Variants of this result can be found in the literature on pseudo-holomorphic curves---e.g.: \cite[Proposition 3.12]{Zinger2020}.
For the readers' convenience a proof is included below.

\begin{lemma}[Monotonicity formula]
  \label{Lem_MonotonicityFormula}
  For every $\epsilon \geq 0$ and $\delta = \frac12\epsilon$ the following holds.
  Let $(J,g)$ be an almost Hermitian structure on $X$.
  Let $x \in X$ and $r_1,r_2 \in (0,\inj_g(x))$ with $r_1 \leq r_2$.
  Let $C \subset B_{r_2}(x)$ be a $J$--holomorphic submanifold.
  If
  \begin{equation*}
    \Abs*{r^{-2}\paren{J - J_x}}_{C^0(B_{r_2}(x))}
    \leq
    r_2^{-2}\delta
    \qandq
    \Abs*{r^{-2}\paren{g - g_x}}_{C^0(B_{r_2}(x))}
    \leq
    r_2^{-2}\delta,
  \end{equation*}
  then
  \begin{equation*}
    (1+\epsilon r_2^2)\frac{\area\paren{C \cap B_{r_2}(x)}}{r_2^2}
    -
    (1-\epsilon r_1^2)\frac{\area\paren{C \cap B_{r_1}(x)}}{r_1^2}
    \geq
    \int_{C\cap\paren{B_{r_2}(x) \setminus B_{r_1}(x)}} \frac{\abs{\nabla r^\perp}^2}{r^2} \, \vol_C.
  \end{equation*}
  Here $r \coloneq d(\cdot,x)$ and $(\cdot)^\perp$ denotes the projection onto the orthogonal complement of $T_yC$.
\end{lemma}

\begin{proof}
  The following argument is essentially due to \citet[§3]{Imagi2015}.
  The cognizant reader will realize that the proof immediately carries over to semi-calibrated cycles.
  
  It suffices the prove the statement with $X = B_1(0) \subset \C^m$, $x = 0$, $r_2 = 1$, $r_1 = s$, $J_x = i$, $g_x = g_0$, $\exp_x^g = \id_{B_1(0)}$, and $\nabla r = \del_r$.
  \edited{Here $g_x$, $J_x$ indicate the Euclidean inner product, almost complex structure on $T_xX = \C^m$.
  Let $\sigma_0$ be the Hermitian form of $g_0$ and $i$.}
  By hypothesis, $C$ is semi-calibrated by $\sigma \coloneq g(J\cdot,\cdot)$.

  Define $f_0 \co (0,1] \to [0,\infty)$ by
  \begin{equation*}
    f_0(s)
    \coloneq
    s^{-2}\int_{B_s(0) \cap C} \sigma_0.
  \end{equation*}
  A moment's thought shows that
  \begin{equation*}
    \sigma_0 = \frac12 \rd(r^2\alpha)
    \qandq
    i(\del_r)\rd r \wedge \sigma_0
    = \frac12 r^2\rd \alpha    
    \qwithq
    \alpha \coloneq r^{-1}i(\del_r)\sigma_0.
  \end{equation*}
  Therefore,
  \begin{equation*}
    f_0(s)
    =
    \frac12 \int_{\del B_s(0) \cap C} \alpha
  \end{equation*}
  and
  \begin{equation*}
    f_0(1)-f_0(s)
    = \frac12 \int_{\paren{B_1(0)\setminus B_s(0)} \cap C} \rd \alpha
    = \int_{\paren{B_1(0)\setminus B_s(0)} \cap C} r^{-2} i(\del_r) \rd r \wedge \sigma_0.
  \end{equation*}

  Since $C$ is semi-calibrated by $\sigma$,
  if $\nu \perp T_yC$,
  then $i(\nu)\sigma|_C = 0$.
  Therefore, with $(e_1,e_2)$ denoting a local orthonormal frame of $C$
  \begin{equation*}
    \paren{i(\del_r) \rd r \wedge \sigma}|_C
    =
    \Inner{\rd r \wedge \sigma,\del_r\wedge e_1\wedge e_2} \cdot \vol_C
    =
    \Inner{\rd r \wedge \sigma,\del_r^\perp\wedge e_1\wedge e_2} \cdot \vol_C
    =
    \abs{\del_r^\perp}^2 \cdot \vol_C
  \end{equation*}
  This proves the assertion with $\epsilon = 0$. 
  
  The function $f \co (0,1] \to [0,\infty)$ defined by
  \begin{equation*}
    f(s)
    \coloneq
    s^{-2}\int_{B_s(0) \cap C} \sigma
  \end{equation*}
  satisfies
  \begin{equation*}
    (1-\delta s^2) \cdot f(s)
    \leq
    f_0(s)
    \leq
    (1+\delta s^2) \cdot f;
  \end{equation*}
  moreover,
  \begin{equation*}
    i(\del_r) \rd r \wedge \sigma_0|_C
    \geq
    i(\del_r) \rd r \wedge \sigma|_C
    - \delta r^2 \cdot \vol_C.
  \end{equation*}
  Therefore,
  \begin{align*}
    (1+\delta s^2)\cdot f(1) - (1-\delta s^2) \cdot f(s)
    &\geq
      f(1) - f(s) \\
    &=
      \int_{\paren{B_1(0)\setminus B_s(0)} \cap C} r^{-2} i(\del_r) \rd r \wedge \sigma_0 \\
    &\geq
      \int_{\paren{B_1(0)\setminus B_s(0)} \cap C} \frac{\abs{\del_r^\perp}}{r^2} \cdot\vol_C - \delta s^2 \cdot f(1).
  \end{align*}
  This proves the assertion.
\end{proof}

\begin{cor}
  \label{Cor_Monotonicity_Rigidity}
  If $C$ is an $i$--holomorphic $2$--dimensional submanifold of $\C^n$ satisfying
  \begin{equation*}
    \area\paren*{C \cap B_r(x)} = \pi r^2
  \end{equation*}
  for every $r > 0$,
  then $C$ is a complex line.
\end{cor}

\begin{proof}[Proof of \autoref{Prop_CurvesAreOpenInCycles}]
  Suppose $(J_n,g_n;C_n) \in \SpaceOfCycles^\N$ geometrically converges to $(J,g;C) \in \SpaceOfCurves$.
  For every $n \in \N$ decompose $C_n$ as
  \begin{equation*}
    C_n = D_n + E_n
    \qwithq
    D_n \coloneq \sum_{i=1}^{I_n} C_{n,i} \qandq
    E_n \coloneq \sum_{i=1}^{I_n} (m_{n,i}-1) C_{n,i}.
  \end{equation*}
  \edited{By construction,
    $\bM(D_n) + \bM(E_n) = \bM(C_n)$, $\supp D_n = \supp C_n$, and $\supp E_n \subset \supp C_n$.
    Moreover, $\lim_{n\to \infty} \bM(C_n) = \bM(C)$ by the continuity of mass, see \autoref{Rmk_TamedAlmostHermitianStructures}.}
  By \autoref{Thm_SpaceOfCycles_Proper} every subsequence of $(J_n,g_n;D_n)$ has a subsequence which geometrically converges to a limit $(J,g;D)$.
  By construction $(\supp D_n)$ converges to $\supp C$;
  hence:  $\supp D = \supp C$.
  A further moment's thought shows that $D = C$.
  Therefore, $(J_n,g_n;D_n)$ geometrically converges to $(J,g;C)$ and $\lim_{n \to \infty} \bM(E_n) = 0$.
  The latter contradicts \autoref{Lem_MonotonicityFormula}.
\end{proof}

\edited{
  \begin{remark*}
    A referee pointed out that it is possible to replace the above use of \autoref{Thm_SpaceOfCycles_Proper},
    which relies on the delicate regularity theory for semi-calibrated currents,
    by an application of Federer--Fleming's Compactness Theorem combined with the Constancy Theorem \cite[Theorems 3.11 and 2.34]{Simon1983}.
  \end{remark*}
}


\subsection{Allard's regularity theorem}
\label{Sec_AllardRegularity}

\begin{definition}
  Let $g$ be Riemannian metric on $X$.
  Let $d \in \N_0$.
  Denote by $\sH^d$ the $d$--dimensional Hausdorff measure.
  \begin{enumerate}
  \item
    A Borel subset $S \subset X$ is \defined{rectifiable of dimension $d$} if there is a countable set $\set{ S_i : i \in I }$ of \edited{$d$--dimensional} $C^1$ submanifolds with
    \begin{equation*}
      \sH^d\paren[\Big]{S \setminus \bigcup_{i \in I} S_i} = 0.
    \end{equation*}
  \item
    An \defined{integral varifold of dimension $d$} is a pair $V = (S,m)$ consisting of a rectifiable subset $S$ of dimension $d$ and a Borel function $m \co S \to \N$.
  \end{enumerate}
  Let $V = (S,m)$ be an integral varifold of dimension $d$.
  \begin{enumerate}[resume]
  \item
    The \defined{measure associated with $V$} and the \defined{mass of $V$} are defined by
    \begin{equation*}
      \mu_V \coloneq m \, \sH^d|_S
      \qandq
      \bM(V) \coloneq \mu_V(X) = \int_S m \, \sH^d.
    \end{equation*}
  \item
    Let $H_V$ be a Borel vector field over $S$.
    $V$ has \defined{mean curvature $H_V$} if for every compactly supported $C^1$ vector field $v$
    \begin{equation*}
      \int \Inner{H_V,v} \,\mu_V
      =
      -\left.\frac{\rd}{\rd t}\right|_{t=0}
      \int_{\flow_v^t(S)} m \circ \flow_v^{-t} \, \sH^d.
     \end{equation*}
     Here $\flow_v^t$ denotes the flow of $v$.
   \item
     For $x \in X$ and $r > 0$ set
     \begin{equation*}
       \theta_V(x,r)
       =
       \theta_V(x,r;g)
       \coloneq
       \frac{\mu_V(B_r(x))}{\omega_d r^d}.
     \end{equation*}
     Here $\omega_d \coloneq \vol(B_1^d(0))$.
     \qedhere
  \end{enumerate}
\end{definition}

\begin{theorem}[{\citet[§8]{Allard1972}; see also \cites[Theorem 24.3]{Simon1983}[Theorem 3.2]{DeLellis2018a}}]
  \label{Thm_AllardRegularity}
  Let $m,d \in \N_0$ with $d \leq m$ and $\alpha \in (0,1)$.
  There are $\epsilon = \epsilon(m,d,\alpha) > 0$ and \edited{$0 < \gamma = \gamma(m,d,\alpha) < 1$} such that the following holds for every $r > 0$.
  If $V$ is an integral varifold of dimension $d$ in $(B_r^m(0),g_0)$ satisfying
  \begin{equation*}
    \theta_V(0,r) \leq 1+\epsilon \qandq
    \Abs{H_V}_{L^\infty(B_r(0))} \leq \epsilon/r,
  \end{equation*}
  then $V \cap B_{\gamma r}(0)$ is a $C^{1,\alpha}$ submanifold of $\R^m$.
  \qed
\end{theorem}

\begin{remark}
  This implies a corresponding result for Riemannian manifolds.
  Indeed,
  Nash proved that every Riemannian manifold $(X,g)$ admits an isometric embedding
  $\iota\co (X,g) \incl (\R^m,g_0)$ with $m = m(\dim X)$.
  Moreover,
  if $\rII_\iota$ denotes the second fundamental form of this embedding,
  then
  \begin{equation*}
    \abs{H_{\iota(V)}} \leq \abs{H_V} + \abs{\rII_\iota}.
    \qedhere
  \end{equation*}  
\end{remark}

\begin{remark}
  It is a nuisance that the dependence of $\epsilon$ on $g$ is not explicit.
  It should be possible to prove \autoref{Thm_AllardRegularity} directly for $g = g_0 + O(r^2)$ on $B_r^m(0)$.
  By careful bookkeeping in the proof of Nash's (local) isometric embedding theorem,
  it should also be possible to obtain bounds on the second fundamental form $\rII_\iota$ depending on $g - g_0$ and its derivatives.
  Unfortunately,
  the authors failed to locate proofs of either result in the literature.
\end{remark}

\begin{theorem}[{\citet[Proposition 5.5]{Gray1965}}]
  \label{Thm_Gray}
  Let $(J,g)$ be an almost Hermitian structure on $X$.
  For every $J$--holomorphic cycle $C$
  \begin{equation*}
    \abs{H_C} \leq \abs{\nabla J}.
  \end{equation*}
\end{theorem}

\edited{
\begin{proof}
  This observation is essentially due to \citet[Proposition 5.5]{Gray1965} \emph{if $C$ is smooth}.
  If $C$ is a $J$--holomorphic cycle,
  then by the variation formula for semi-calibrated cycles \cite[Proposition 1.2]{DeLellis2017}
  \begin{equation*}
    \int \Inner{H_C,v} \mu_C = \delta_C(i_v\rd\sigma).
  \end{equation*}
  Since $\abs{\rd\sigma} \leq \abs{\nabla J}$,
  the assertion follows.
\end{proof}
}

\begin{prop}
  \label{Prop_AllardInSequencesOfAlmostHermitianManifolds}
  Let $K \subset X$ be compact.
  Let $(J_n,g_n)$ be sequence of almost Hermitian structures converging to an almost Kähler structure $(J,g)$ in the $C_\loc^2$ topology.
  There are constants $r,\epsilon > 0$ (depending on the above data) such that the following holds for every $n \in \N$.
  If $C$ is \edited{a} $J_n$--holomorphic cycle with $\supp C \subset K$ and such that for every $x \in C$ there is an $s \in (0,r)$ with
  \begin{equation*}
    \theta_C(x,s;g_n) \leq 1 + \epsilon,
  \end{equation*}
  then $C$ is smooth.
\end{prop}

\begin{proof}
  Choose an open neighborhood $U$ of $K$ such that
  \begin{equation*}
    \epsilon_n \coloneq \sum_{m=n-1}^{n+1} \Abs{g_m-g}_{C^2(U)}
  \end{equation*}
  converges to zero.
  After passing to a subsequence,
  $\limsup_{n\to \infty} n^{-4}  \epsilon_n \leq 1$.
  Choose $\chi \in C^\infty(\R,[0,1])$ with $\chi|_{[-1/3,1/3]} = 1$, $\supp(\chi) \subset [-2/3,2/3]$, and $\sum_{n \in \Z} \chi(\cdot+n) = 1$.
  Define a Riemannian metric $G$ on $(0,1]  \times X$ by
  \begin{equation*}
    G \coloneq \rd t \otimes \rd t + \sum_{n=1}^\infty \chi(1/t - n) g_n.
  \end{equation*}
  By construction,
  for $k,\ell \in \set{0,1,2}$ and $t \in [1/(n-1),1/(n+1)]$
  \begin{equation*}
    \sup_{x \in U} \abs{\del_t^k\nabla_x^\ell G}(t,x)
    \lesssim
    \Abs{\chi}_{C^k}
    n^{-2k}
    \epsilon_n
    \leq
    \Abs{\chi}_{C^k}.
  \end{equation*}
  Therefore, $G$ extends to a $C^2$ Riemannian metric on $[0,1]\times X$.
  
  For every $n \in \N$ the map $\iota_n \coloneq (1/n,\id_X)$ defines an isometric embedding $(U,g_n) \incl ([0,1]\times U,G)$ with $\rII_{\iota_n}$ bounded independently of $n$.
  Choose an isometric embedding $\jmath \co ([0,1] \times U,G) \incl (\R^n,g_0)$.
  By \autoref{Thm_Gray},
  \begin{equation*}
    \abs{H_{\jmath \circ \iota_n(C)}} \leq \Lambda
    \qwithq
    \Lambda
    \coloneq
    \sup_{n \in \N}
    \paren*{
      \Abs{\nabla J_n}_{L^\infty(K)}
      +
      \Abs{\rII_{\iota_n}}_{L^\infty(K)}
    }
    +
    \Abs{\rII_\jmath}_{L^\infty([0,1]\times K)}
    < \infty.
  \end{equation*}
  Therefore,
  the assertion follows from \autoref{Thm_AllardRegularity}.  
\end{proof}

\begin{prop}
  \label{Prop_UpperBoundOnMassRatio}
  If $(J_n,g_n;C_n) \in \SpaceOfCycles^\N$ geometrically converges to $(J,g;C) \in \SpaceOfCycles$,
  then for every $r > 0$
  \begin{equation*}
    \limsup_{n \to \infty}
    \max_{x \in \supp{C_n}}
    \theta_{C_n}(x,r;g_n)
    \leq
    \max_{x \in \supp{C}}
    \theta_C(x,r;g).
  \end{equation*}
\end{prop}

\begin{proof}
  If not,
  then for every $n \in \N$ there is a $x_n \in C_n$ with $(x_n)$ converging to $x \in C$ and
  \begin{equation*}
    \limsup_{n \to \infty}
    \theta_{C_n}(x_n,r;g_n)
    >
    \theta_C(x,r;g).
  \end{equation*}
  However, this is in contradiction to geometric convergence by the continuity of mass on $\SpaceOfCycles$, \edited{see \autoref{Rmk_TamedAlmostHermitianStructures}.}
\end{proof}

\begin{proof}[Proof of \autoref{Thm_EmbeddedCurvesAreOpenInCycles}]
  Suppose $(J_n,g_n;C_n) \in \edited{(\SpaceOfCurves)}^\N$ geometrically converges to $(J,g;C) \in \SpaceOfEmbeddedCurves$.
  Let $r,\epsilon > 0$ be as in \autoref{Prop_AllardInSequencesOfAlmostHermitianManifolds}.
  Choose $s \in (0,r)$ such that
  \begin{equation*}
    \max_{x \in \supp{C}} \theta_{C_n}(x,s) \leq 1 + \frac12\epsilon.
  \end{equation*}
  By \autoref{Prop_UpperBoundOnMassRatio},
  for $n \gg 1$
  \begin{equation*}
    \max_{x \in \supp{C_n}} \theta_{C_n}(x,s) \leq 1 + \epsilon.
  \end{equation*}
  Therefore, by \autoref{Prop_AllardInSequencesOfAlmostHermitianManifolds},
  $C_n$ is embedded.
  This proves that $\SpaceOfEmbeddedCurves$ is open in $\edited{\SpaceOfCurves}$.

  Suppose that $\supp C$ is connected but $\supp C_n$ fails to be connected for $n \gg 1$.
  Decompose $C_n = D_n + E_n$ with $\supp D_n$ and $\supp E_n$ disjoint.
  After passing to a subsequence,
  $(D_n)$ converges to $D$ with $\supp D \subset \supp C$;
  hence: $\supp C = \supp D$.
  Similarly, $(E_n)$ converges to $E$ with $\supp E = \supp C$.
  This contradicts $C \in \SpaceOfConnectedEmbeddedCurves$.
\end{proof}



\subsection{Convergence of submanifolds}
\label{Sec_ConvergenceOfSubmanifolds}

The proof of \autoref{Thm_SpaceOfEmbeddedCurvesEmbedsIntoSpaceOfSubmanifolds} requires the following discussion of the convergence of submanifolds.
This material is entirely standard and elementary.
It is spelled out in detail for the readers' convenience.
Throughout this subsection,
set $m \coloneq \dim X$,
let $d \in \N_0$ with $d \leq m$, and $k \in 2+\N_0$.

\begin{notation}
~
  \begin{enumerate}
  \item
    The \defined{graph} of $f \in C^k(B_1^d(0),\R^{m-d})$ is defined by
    \begin{equation*}
      \graph f \coloneq \set{ (x,f(x) : x \in B_1^d(0) } \subset B_1^d(0) \times \R^{m-d}.
    \end{equation*}
  \item
    For $r > 0$ define $s_r \co \R^m \to \R^m$ by
    \begin{equation*}
      s_r(x) \coloneq r \cdot x.
    \end{equation*}
  \item
    Set
    \begin{equation*}
      Q^d \coloneq B_1^d(0) \times B_1^{m-d}(0).
    \end{equation*}
  \item 
    Let $x \in X$.
    A \defined{frame} of $(T_x X,g_x)$ is a linear isometry $\phi \co (\R^m, g_0) \to (T_x X,g_x)$.
    The space of frames is denoted by
    \begin{equation*}
      \Fr(T_x X,g_x).
      \qedhere
    \end{equation*}
  \end{enumerate}
\end{notation}

\begin{definition}
  Let $g$ be a $C^{k+1}$ Riemannian metric on $X$.
  Denote by $\SpaceOfSubmanifolds$ the set of closed $C^k$ submanifolds of $X$ of dimension $d$.
  Let $(S_n) \in \SpaceOfSubmanifolds^\N$ and $S \in \SpaceOfSubmanifolds$.
  $(S_n)$ \defined{weakly $C_\loc^k$ converges to $S$} if:
  \begin{enumerate}
  \item
    For every compact $K \subset X$ the sequence $(S_n \cap K)$ converges to $S \cap K$ with respect to the Hausdorff metric.
  \item
    For every $x \in S$ there are \edited{$r > 0$,} $\phi \in \Fr(T_xX,g_x)$, $(f_n) \in C^k(B_1^d(0),\R^{m-d})^\N$, and $f \in C^k(B_1^d(0),\R^{m-d})$ such that:
    \begin{enumerate}
    \item
      $\paren{{\exp_x^g} \circ \phi \circ s_r}^{-1}(S) \cap Q^d = \graph f$,
    \item
      $\paren{{\exp_x^g} \circ \phi \circ s_r}^{-1}(S_n) \cap Q^d = \graph f_n$
      for $n \gg 1$,
    \item
      $\limsup_{n \to \infty} \Abs{f_n}_{C^k} < \infty$, and
    \item
      $\lim_{n \to \infty} \Abs{f_n-f}_{C^{k-1,\alpha}} = 0$ for every $\alpha \in (0,1)$.
      \qedhere
    \end{enumerate}   
  \end{enumerate}
\end{definition}

\begin{definition}
  \label{Def_RegularityScale}
  Let $g$ be a $C^{k+1}$ Riemannian metric on $X$.
  Let $S \subset X$.
  The \defined{$C^k$ regularity scale} is the map $r_S^k(\,\cdot\,;g) \co S \to [0,\infty]$ defined by  
  \begin{equation*}
    r_S^k(x;g)
    \coloneq
    \sup \set[\big]{
      r_S^k(x,\phi;g) : \phi \in \Fr(T_xX,g_x) 
    }
  \end{equation*}
  with
  $r_S^k(x,\phi;g)$
  denoting the supremum of those $r \in (0,\inj_g(x)/2]$ for which
  \begin{equation*}
    \paren{{\exp_x^g} \circ \phi \circ s_r}^{-1}(S) \cap Q^d = \graph f
  \end{equation*}
  with $f \in C^k(B_1^d(0),\R^{m-d})$ satisfying $\Abs{f}_{C^k} \leq 1$;
  if there is no such $r$ (that is: if $S$ fails to be a $C^k$ submanifold in every neighborhood of $x$),
  then 
  \begin{equation*}
    r_S^k(x,\phi;g) \coloneq 0.
    \qedhere
  \end{equation*}
\end{definition}

\begin{prop}
  \label{Prop_RegularityScale_Convergence}
  Let $g$ be a $C^{k+1}$ Riemannian metric on $X$.
  Let $(g_n)$ be a sequence of $C^{k+1}$ Riemannian metrics on $X$ converging to $g$ in the $C_\loc^{k+1}$ topology.
  Let $(S_n) \in \SpaceOfSubmanifolds^\N$ and let $S \subset X$ be a closed subset.
  If for every compact $K \subset X$ the sequence $(S_n \cap K)$ converges to $S \cap K$ with respect to the Hausdorff metric and
  \begin{equation*}
    \liminf_{n \to \infty} \inf \set[\big]{ r_{S_n}^k(x,g_n) : x \in S_n \cap K } > 0,
  \end{equation*}
  then
  $S$ is a $C^k$ submanifold, and
  $(S_n)$ weakly $C_\loc^k$ converges to $S$.
\end{prop}

The proof requires the following preparation.

\begin{prop}
  \label{Prop_LipschitzImagesOfLipschitzGraphsAreLipschitzGraphs}
  For every $\epsilon_0 \in (0,1)$ there is a constant $c = c(k,\epsilon_0) > 0$ such that the following holds.
  Let
  $f \in C^k(B_1^d(0), \R^{m-d})$ and
  $\Phi \in C^k(Q, \R^m)$.
  If
  \begin{equation*}
    \Abs{f}_{C^k} \leq 1 \qandq
    \epsilon \coloneq \Abs{\Phi-\id}_{C^k} \leq \epsilon_0,
  \end{equation*}
  then there is an $\tilde f \in C^k(B_{1-\epsilon}^d(0), \R^{m-d})$ such that
  \begin{equation*}
    \Phi(\graph f) \cap \paren{B_{1-\epsilon}^d(0) \times \R^{m-d}} = \graph \tilde f
    \qandq
    \Abs{\tilde f-f}_{C^k} \leq c\epsilon.
  \end{equation*}
\end{prop}

\begin{proof}
  Define $\xi \in C^k(B_1^d(0),\R^d)$ and $\phi \in C^k(B_1^d(0),\R^{m-d})$ by
  \begin{equation*}
    (\xi(x), \phi(x)) \coloneq \Phi(x,f(x)).
  \end{equation*}
  A moment's thought shows that
  \begin{equation*}
    \Abs{\xi - \id}_{C^k} \leq \epsilon \qandq
    \Abs{\phi - f}_{C^k} \leq \epsilon.
  \end{equation*}
  By the inverse function theorem,
  $\xi$ is injective,
  $B_{1-\epsilon}(0) \subset \im \xi$,
  and
  \begin{equation*}
    \Abs{\xi^{-1}-\id}_{C^k} \leq c\epsilon.
  \end{equation*}
  Define $\tilde f \co \bar B_{1-\epsilon}^d(0) \to \R^{m-d}$ by
  \begin{equation*}
    \tilde f \coloneq \phi \circ \xi^{-1}.
  \end{equation*}
  By construction,
  \begin{equation*}
    \Phi(\graph f) \cap \paren{B_{1-\epsilon}^d(0) \times \R^{m-d}} = \graph \tilde f
  \end{equation*}
  and
  \begin{equation*}
    \Abs{\tilde f-f}_{C^k}
    \leq
    \Abs{\phi \circ \xi^{-1}-\phi}_{C^k}
    +
    \Abs{\phi-f}_{C^k}
    \leq
    c\epsilon.
    \qedhere
  \end{equation*}
\end{proof}

\begin{proof}[Proof of \autoref{Prop_RegularityScale_Convergence}]
  Let $x \in S$.
  For every $n\in\N$ choose $x_n \in S_n$ such that $x = \lim_{n \to \infty} x_n$.
  By hypothesis,
  \begin{equation*}
    r \coloneq
    \liminf_{n \to \infty}
    r_{S_n}(x_n;g_n)
    >
    0.
  \end{equation*}
  By \autoref{Def_RegularityScale},
  for every $n\in\N$
  there are
  $r_n \in (0,\inj_{g_n}(x)/2]$,
  $\phi_n \in \Fr(T_{x_n}X,g_n)$, and
  $f_n \in C^k(B_1^d(0),\R^{m-d})$  with
  $\Abs{f_n}_{C^k} \leq 1$ 
  such that
  \begin{equation*}
    \tilde S_n
    \coloneq
    \iota_n^{-1}(S_n) \cap Q^d
    =
    \graph{f_n}
    \qwithq
    \iota_n
    \coloneq
    \exp_{x_n}^{g_n}{} \circ \phi_n \circ s_{r_n}
  \end{equation*}
  and
  \begin{equation*}
    \liminf_{n \to \infty} r_n
    =
    r.
  \end{equation*}
  
  By the Arzelà–Ascoli theorem,
  after passing to a subsequence (without relabelling),
  $(r_n)$ converges to $r$,
  $(\phi_n)$ converges to $\phi \in \Fr(T_xX,g)$, and
  $(f_n)$ converges to $f \in C^k(B_1^d(0),\R^{m-d})$ with $\Abs{f}_{C^k} \leq 1$ in the $C^{k-1,\alpha}$ topology for every $\alpha \in (0,1)$.
  The sequence $(\iota_n)$ converges to
  \begin{equation*}
    \iota \coloneq {\exp_x^g} \circ \phi \circ s_r
    \co B_2^d(0) \to X
  \end{equation*}
  in the $C^k$ topology.  
  Set
  \begin{equation*}
    \tilde S
    \coloneq
    \iota^{-1}(S)
    \cap
    Q^d
  \end{equation*}
 and for $\rho \in (0,1)$ set
  \begin{equation*}
    Q^d_\rho
    \coloneq
    B_\rho^d(0) \times B_\rho^{m-d}(0).
  \end{equation*}
  On the one hand, by hypothesis, 
  $\paren{\iota_n\paren{\overline{Q^d_\rho}}}$ converges to $\iota\paren{\overline{Q^d_\rho}}$ with respect to the Hausdorff metric,
  and,
  therefore
  $\paren{\tilde S_n \cap \overline{Q^d_\rho}}$ converges to $\tilde S \cap \bar Q^d_\rho$ with respect to the Hausdorff metric. 
  On the other hand,
  evidently,
  $\paren{\graph{f_n}}$ converges to $\graph{f}$ with respect to the Hausdorff metric.
  Therefore, since $\rho \in (0,1)$ was arbitrary,
  \begin{equation*}
    \tilde S = \graph{f}.  
  \end{equation*}    
  In particular, $S \cap B_r(x)$ is a $C^k$ submanifold.
  
  Since $(\iota_n)_{n\in\N}$ converges to $\iota$ in the $C^{k}$ topology, 
  by \autoref{Prop_LipschitzImagesOfLipschitzGraphsAreLipschitzGraphs},
  for every $n \gg 1$ there is an $\tilde f_n \in C^k(B_{1/2}^d(0),\R^{m-d})$ with $\Abs{\tilde f_n}_{C^k} \leq 2$ such that
  \begin{equation*}
    \iota^{-1}(S_n) \cap Q^d_{1/2} = \graph{\tilde f_n}
  \end{equation*}
  and for every $\alpha \in (0,1)$
  \begin{equation*}
    \lim_{n \to \infty} \Abs{\tilde f_n - f}_{C^{k-1,\alpha}} = 0.
  \end{equation*}
  Therefore, $(S_n)$ weakly $C_\loc^k$ converges to $S$.
\end{proof}

\begin{prop}
  \label{Prop_CoordinateConvergence=>NormalBundleConvergence}
  Let $(S_n) \in \SpaceOfSubmanifolds^\N$ and $S \in \SpaceOfSubmanifolds$.
  Suppose that $S$ is compact.
  If $(S_n)$ weakly $C_\loc^k$ converges to $S$,
  then for every $n \gg 1$ there is a $\xi_n \in \Gamma(NS)$ with $\abs{\xi_n} < \inj_g$ such that
  \begin{equation*}
    S_n = \graph(\xi_n) \coloneq \set{ \exp_x^g \xi_n(x) : x \in S}
  \end{equation*}
  and
  \begin{equation*}
    \lim_{n\to\infty} \Abs{\xi_n}_{C^{k-1}} = 0.
  \end{equation*}
\end{prop}

\begin{proof}
  Let $x \in S$ and $r \in (0,\inj_g(x)/4]$.
  Choose a frame $\phi \in \Fr(T_xX,g_x)$ with $\phi(T_xS) = \R^d \subset \R^m$.
  Define $\iota \co B_2^m(0) \to X$ by
  \begin{equation*}
    \iota \coloneq \exp_x^g{} \circ \phi \circ s_r
  \end{equation*}
  and define $\jmath \co B_2^d(0) \times B_2^{m-d}(0) \to X$ by
  \begin{equation*}
    \jmath \circ \phi^{-1} \circ s_r^{-1} (v,w) \coloneq \exp_{\exp_x^g(v)}^g(\tilde w)
  \end{equation*}
  with $\tilde w$ denoting the parallel transport of $w$ along the geodesic $t \mapsto \exp_x^g(tv)$.
  The map $\Phi \coloneq \jmath^{-1} \circ \iota \co Q^d \to \R^m$ can be made arbitrarily $C^{k-1}$--close to $\id$ by choosing $r \ll 1$.
  Therefore, the assertion follows from \autoref{Prop_LipschitzImagesOfLipschitzGraphsAreLipschitzGraphs}.
\end{proof}

\begin{cor}
  \label{Cor_SmoothConvergence}
  Let $(S_n)$ be a sequence of $C^2$ submanifolds and let $S$ be a $C^2$ submanifold.
  If $(S_n)$ weakly $C^2$ converges to $S$,
  then it $C^1$ converges to $S$.
\end{cor}



\subsection{Convergence of embedded pseudo-holomorphic curves}
\label{Sec_ConvergenceOfEmbeddedCurves}

\begin{prop}
  \label{Prop_RegularityScale_Density}
  Let $k \in 2 + \N_0$ and $m\in \N$. 
  For every $\Lambda > 0$ there are $\epsilon = \epsilon(m,k,\Lambda) > 0$ and $\delta = \delta(m,k,\Lambda) > 0$ such that the following holds.
  Let $X$ be a smooth manifold of dimension $2m$,
  let $(J,g)$ be a $C^{k+1}$ almost Hermitian structure,
  let $x \in X$,
  let $r \in (0,\inj_g(x))$, and
  let $C \subset B_r(x)$ be a $J$--holomorphic submanifold.
  If
  \begin{equation*}
    \Abs{\paren{{\exp^g} \circ s_r}^*J - J_x}_{C^{k+1}(B_1(0))}
    \leq
    \Lambda
    \qandq
    \Abs{r^{-2}\paren{{\exp^g} \circ s_r}^*g - g_x}_{C^{k+1}(B_1(0))}
    \leq
    \Lambda,
  \end{equation*}
  and for every $y \in C$ and every $0 < s < d(y,\del B_r(x))$
  \begin{equation*}
    \theta_C(y,s;g) \leq 1 + \epsilon,
  \end{equation*}
  then
  \begin{equation*}
    r_C^k(y;g) \geq \delta \cdot d(y,\del B_r(x)).
  \end{equation*}
\end{prop}

\edited{The $C^{k+1}$ estimates are needed to improve the $C^{k-1,\alpha}$ convergence of embedded pseudo-holomorphic curves to $C^k$ convergence.}

\begin{proof}[\edited{Proof of \autoref{Prop_RegularityScale_Density}}]
  It suffices to prove the statement with $X = B_1(0) \subset \C^m$, $x = 0$, $r = 1$, $J_x = i$, and $g$ satisfying $g_x = g_0$ and $\exp_x^g = \id$.
  Here $g_0$ and $i$ are the standard Euclidean metric and complex structure on $\C^m$. 
  
  If the statement fails to hold,
  then for every $n \in \N$ there are
  a $C^{k+1}$ almost Hermitian structure $(g_n,J_n)$ on $B_1(0)$ and
  a $J_n$--holomorphic submanifold $C_n \subset B_1(0)$ such that
  \begin{equation*}
    \Abs{J_n - i}_{C^{k+1}(B_1(0))}
    \leq
    \Lambda
    \qandq
    \Abs{g_n - g_0}_{C^{k+1}(B_1(0))}
    \leq
    \Lambda 
  \end{equation*}
  and for every $x \in C_n$ and every $0 < s < 1-\abs{x}$
  \begin{equation*}
    \theta_C(x,s;g_n) \leq 1 + \epsilon_n \qwithq \epsilon_n \coloneq 1/n,
  \end{equation*}
  but the sequence $(\delta_n)_{n \in \N}$ defined by    
  \begin{equation*}
    \delta_n \coloneq \inf_{x \in B_1(0)} \frac{r_{C_n}^k(x;g_n)}{1-\abs{x}}
  \end{equation*}
  converges to zero.
  Since $C_n$ is a submanifold, $\delta_n > 0$.

  For every $n \in \N$ choose $x_n \in B_1(0)$ such that
  \begin{equation*}
    \frac{r_{C_n}^k(x_n;g_n)}{1-\abs{x_n}} \leq 2\delta_n,
  \end{equation*}
  and rescale by declaring that
  \begin{equation*}
    R_n \coloneq 1/r_{C_n}^k(x_n;g_n), \quad
    \tilde J_n \coloneq s_{1/R_n}^*J_n, \quad
    \tilde g_n \coloneq R_n^2 \cdot s_{1/R_n}^*g_n, \qandq
    \tilde C_n \coloneq s_{1/R_n}^{-1}(C_n).
  \end{equation*}
  The following hold:
  \begin{enumerate}
  \item
    For every $n \in \N$ the submanifold $\tilde C_n$ is $\tilde J_n$--holomorphic.
  \item
    Since $(R_n)$ converges to $\infty$,
    $(\tilde J_n,\tilde g_n)$ converges to $(i,g_0)$ in the $C_\loc^{k+1}$ topology.
  \item
    For every $n \in \N$ and $x \in B_{R_n}(0)$
    \begin{equation*}
      r_{\tilde C_n}^k(x;\tilde g_n) = R_n \cdot r_{\tilde C_n}(s_{1/R_n}(x);g_n);
    \end{equation*}
    in particular:
    \begin{equation*}
      r_{\tilde C_n}^k(\tilde x_n;g_n) = 1
      \qwithq
      \tilde x_n \coloneq s_{1/R_n}^{-1}(x_n).
    \end{equation*}
  \item
    For every $n \in \N$, $x \in B_{R_n}(0)$, and $0 < s < R_n - \abs{x}$
    \begin{equation*}
      \theta_{\tilde C_n}(x,s;\tilde g_n) \leq 1 + \epsilon_n.
    \end{equation*}
  \item
    The sequence $(\tilde R_n)$ defined by
    \begin{equation*}
      \tilde R_n \coloneq \frac12\cdot\paren{R_n - \abs{\tilde x_n}}
      =
      \frac12 \cdot \frac{1 - \abs{x_n}}{r_{C_n}^k(x_n;g_n)}
    \end{equation*}
    converges to $\infty$.
  \item
    For every $n \in \N$ and $x \in B_{R_n}(0)$
    \begin{equation*}
      r_{\tilde C_n}^k(x;\tilde g_n)
      \geq 
      \delta_n \cdot \paren{R_n - \abs{x}}
      \geq 
      \frac12 \cdot \frac{R_n - \abs{x}}{R_n - \abs{\tilde x_n}}
      \geq
      \frac12 \cdot \paren*{1 - \frac{d(x,x_n)}{R_n - \abs{\tilde x_n}}};
    \end{equation*}
    in particular:
    \begin{equation*}
      \inf \set[\big]{ r_{\tilde C_n}^k(x;\tilde g_n) : x \in B_{\tilde R_n}(\tilde x_n) }\geq \frac14.
    \end{equation*}
  \end{enumerate}
  
  Translate by $-\tilde x_n$ in order to assume that $\tilde x_n = 0$.
  By \autoref{Prop_RegularityScale_Convergence},
  after passing to a subsequence (without relabeling),
  $(\tilde C_n)$ weakly $C_\loc^k$ converges to an $i$--holomorphic submanifold $C$.
  For every $x \in C$ and $s > 0$,
  \begin{equation*}
    \theta_C(x,s;g_0) = 1.
  \end{equation*}
  Therefore,
  by \autoref{Cor_Monotonicity_Rigidity},
  $C$ is a complex line.
  Without loss of generality, $C = \C \times \set{0}$.

  Since $(\tilde C_n)$ weakly $C^k$ converges to $C$,
  by \autoref{Prop_LipschitzImagesOfLipschitzGraphsAreLipschitzGraphs},
  there are $f_n \in C^k(B_4(0),\C^{m-1})$ for every $n\in\N$
  such that 
  $\limsup_{n\to\infty} \Abs{f_n}_{C^k} < \infty$,
  $\lim_{n\to\infty} \Abs{f_n}_{C^{k-1,\alpha}} = 0$ for every $\alpha \in (0,1)$,
  and for $n \gg 1$,
  \begin{equation*}
    \tilde C_n \cap (B_4(0) \times \C^{m-1}) = \graph{f_n}.
  \end{equation*}
  \edited{The upcoming argument based on Schauder estimates proves that the convergence can be improved to $C^k$, that is: $\lim_{n\to\infty} \Abs{f_n}_{C^k(B_2(0))} = 0$.
  This contradicts $r_{\tilde C_n}^k(0;g_n) = 1$, cf. \autoref{Def_RegularityScale}.}
  
  The map $F_n \in C^k(B_4(0), \C^m)$ defined by $F_n(z) \coloneq (z, f_n(z))$ satisfies
  \begin{equation}
    \label{Eq_FnPseudoHolomorphic}
    \paren{\tilde J_n \circ F_n} \cdot \rd F_n - \rd F_n\cdot j_n = 0
  \end{equation}
  with $j_n$ denoting the $C^{k-1}$ complex structure on $B_4(0)$ associated with $F_n^*\tilde g_n$.
  For every $\alpha\in(0,1)$
  \begin{equation*}
    \lim_{n\to\infty}\Abs{\tilde J_n \circ F_n -i}_{C^{k-1,\alpha}} = 0 \qandq \lim_{n\to\infty}\Abs{j_n-j}_{C^{k-2,\alpha}} = 0
  \end{equation*}
  with $i$ and $j$ denoting the standard complex structures on $\C^m$ and $\C$ respectively.  
  With $\delbar F = i \rd F - \dF j$ denoting the standard Cauchy--Riemann operator,
  \autoref{Eq_FnPseudoHolomorphic} is rewritten as
  \begin{equation*}
    \delbar F_n + (\tilde J_n \circ F_n - i) \cdot \rd F_n - \rd F_n \cdot (j_n-j) = 0.
  \end{equation*}
  Since $\delbar F_n = (0,\delbar f_n)$ and $\lim_{n\to\infty}\Abs{\nabla f_n}_{C^{k-1,\alpha}} = 0$,
  this implies a PDE of the form
  \begin{equation*}
    \Delta f_n + \fp(\tilde J_n,f_n,\nabla f_n)\nabla^2 f_n + \fq(\tilde J_n,f_n,\nabla f_n) = 0
  \end{equation*}
  with
  \begin{equation*}
    \lim_{n\to\infty}
    \Abs{\fp(\tilde J_n,f_n,\nabla f_n)}_{C^{k-2,\alpha}}
    =
    0
    \qandq
    \lim_{n\to\infty}
    \Abs{\fq(\tilde J_n,f_n,\nabla f_n)}_{C^{k-2,\alpha}}
    =
    0
  \end{equation*}
  for every $\alpha \in (0,1)$.
  Therefore, by interior Schauder estimates \cite[Theorem 6.6]{Gilbarg2001}, 
  \begin{equation*}
    \lim_{n\to\infty} \Abs{f_n}_{C^k(B_2(0))}
    \leq
    \lim_{n\to\infty} \Abs{f_n}_{C^{k,\alpha}(B_2(0))} = 0.
    \qedhere
  \end{equation*}
\end{proof}

\begin{proof}[Proof of \autoref{Thm_SpaceOfEmbeddedCurvesEmbedsIntoSpaceOfSubmanifolds}]
  The map $(\pr_\fH,\supp) \co \SpaceOfEmbeddedCurves \to \SpaceOfHermitianStructures \times \SpaceOfSubmanifolds$ is injective.
  To see that it is continuous,
  suppose that $(C_n) \in \paren{\SpaceOfEmbeddedCurves}^\N$ geometrically converges to $C \in \SpaceOfEmbeddedCurves$.
  Let $\epsilon > 0$ be as in \autoref{Prop_RegularityScale_Density}.
  Since $C$ is embedded,
  there is an $r \in (0,\inj_g)$ such that
  \begin{equation*}
    \max_{x \in \supp C} \theta_C(x,r;g)
    \leq
    1 + \epsilon/2.
  \end{equation*}  
  By \autoref{Prop_UpperBoundOnMassRatio} and \autoref{Lem_MonotonicityFormula},
  for $n \gg 1$ and $0 < s < r$
  \begin{equation*}
    \max_{x \in \supp C_n} \theta_{C_n}(x,s;g_n)
    \leq
    1 + \epsilon.
  \end{equation*}
  Therefore,
  by \autoref{Prop_RegularityScale_Density},
  \begin{equation*}
    \liminf_{n \to \infty} \inf_{x \in C_n} r_{C_n}^2(x;g) > 0.
  \end{equation*}
  Therefore,
  by \autoref{Prop_RegularityScale_Convergence} and \autoref{Cor_SmoothConvergence},
  $(C_n)$ $C^1$ converges to $C$.
  Evidently, if $(C_n)$ $C^1$ converges to $C$,
  then it geometrically converges to $C$.
\end{proof}

The same argument also proves the following.

\begin{definition}
  Let $(J,g;C) \in \SpaceOfCycles$.
  A point $x \in \supp C$ is \defined{smooth} if
  \begin{equation}
    \limsup_{r \downarrow 0} \theta_C(x,r;g) = 1.
    \qedhere
  \end{equation}  
\end{definition}

\begin{prop}
  \label{Prop_LocalSmoothConvergence}
  If $(J_n,g_n;C_n) \in \SpaceOfCycles^\N$ geometrically converges to $(J,g;C) \in \SpaceOfCycles$ and $x \in \supp C$ is smooth,
  then there is a neighborhood $U$ of $x \in X$,
  such that,
  for every $n \gg 1$,
  $C_n \cap U$ is embedded
  and $(C_n \cap U)$ $C^1$ converges to $C \cap U$.
  \qed
\end{prop}




\section{The proof of the Gopakumar--Vafa conjecture}
\label{Sec_GopakumarVafaConjecture}

Throughout this section,
assume the following.

\begin{situation}
  \label{Sit_SymplecticSixManifolds}
  Let $(X,\omega)$ be a closed symplectic $6$--manifold.
  Denote by $\SpaceOfAlmostComplexStructures \coloneq \SpaceOfAlmostComplexStructures_\tau(\omega)$ the space of smooth almost complex structures $J$,
  which are tamed by $\omega$,
  equipped with the $C^\infty$ topology;
  cf.~\autoref{Ex_SymplecticAlmostHermitianStructures}~\autoref{Ex_SymplecticAlmostHermitianStructures_Tamed}.
\end{situation}

This section carries forward the notation from \autoref{Sec_SpacesOfPseudoHolomorphicCurves} with $\SpaceOfHermitianStructures = \SpaceOfAlmostComplexStructures$.
In particular, $\ModuliSpaceOfNodalMaps$ denotes the universal moduli space over $\SpaceOfAlmostComplexStructures$ of \edited{stable pseudo-holomorphic maps};
moreover, $\ModuliSpaceOfSimpleMaps$ and $\ModuliSpaceOfEmbeddedMaps$ denote the subspaces consisting of the equivalence classes of simple maps and of embeddings.
For $A\in\rH_2(M,\Z)$ and $\genus\in\N_0$ denote by $\ModuliSpaceOfNodalMaps_{A,\genus}$ the subspace of nodal pseudo-holomorphic maps representing $A$ and of genus $\genus$.
For $J \in\SpaceOfAlmostComplexStructures$ and $S \subset \SpaceOfAlmostComplexStructures$ set
\begin{equation*}
  \ModuliSpaceOfNodalMaps(J) \coloneq \pr_\SpaceOfAlmostComplexStructures^{-1}(J)
  \qandq
  \ModuliSpaceOfNodalMaps(S) \coloneq \pr_\SpaceOfAlmostComplexStructures^{-1}(S)
\end{equation*}
with $\pr_\SpaceOfAlmostComplexStructures \co  \ModuliSpaceOfNodalMaps \to \SpaceOfAlmostComplexStructures$ denoting the projection map.
Analogous notation is used for the subspaces of $\ModuliSpaceOfNodalMaps$ introduced above.

The infinitesimal structure of the moduli space is controlled by the linearization of the Cauchy--Riemann operator.

\begin{definition}
  \label{Def_CauchyRiemann}
  Let $J\in\SpaceOfAlmostComplexStructures$.
  Let $u \co (\Sigma,j) \to (X,J)$ be a $J$--holomorphic map.  
  \begin{enumerate}
  \item
    \label{Def_CauchyRiemann_D}
    Let $\cS$ be an $\Aut(\Sigma,j)$--invariant slice of the Teichmüller space $\cT(\Sigma)$ through $j$.
    The linearization of the Cauchy--Riemann operator defines a linear map
    \begin{equation*}
      \fd_{u,J} \co T_j\cS \oplus \Gamma(u^*TX) \to \Omega^{0,1}(\Sigma,u^*TX).
    \end{equation*}
    If $u$ is the inclusion of a $J$--holomorphic curve $C$,
    then $\fd_{C,J} \coloneq \fd_{u,J}$.
  \item
    \label{Def_CauchyRiemann_Index}
    The \defined{index of $u$} is
    \begin{align*}
      \ind u
      &\coloneq
        \ind \fd_{u,J} - \dim \Aut(\Sigma,j) \\
      &=
        (\dim X - 6)(1-g) + 2c_1(A)
        =
        2c_1(A)
    \end{align*}
    with $A \coloneq u_*[\Sigma] \in \rH_2(X,\Z)$ and $c_1(A) \coloneq \Inner{c_1(X,\omega),A}$. 
    If $u$ is the inclusion of a $J$--holomorphic curve $C$,
    then the \defined{index of $C$} is $\ind u$.
  \item
    \label{Def_CauchyRiemann_Unobstructed}
    The map $u$ is \defined{unobstructed with respect to $J$} if $\coker \fd_{u,J} = 0$.
    If $u$ is the inclusion of a $J$--holomorphic curve $C$,
    then $C$ is \defined{unobstructed with respect to $J$} if $u$ is.
    \qedhere
  \end{enumerate}
\end{definition}

 
\subsection{Gromov--Witten invariants of symplectic $6$--manifolds}
\label{Sec_GromovWittenInvariants}

For every $J \in\SpaceOfAlmostComplexStructures$, $A \in \rH_2(X,\Z)$, and $\genus \in \N_0$
the moduli space $\ModuliSpaceOfNodalMaps_{A,\genus}(J)$ carries a \defined{virtual fundamental class (VFC)}
\begin{equation*}
  [\ModuliSpaceOfNodalMaps_{A,\genus}(J)]^{\vir} \in \check\rH^\vdim (\ModuliSpaceOfNodalMaps_{A,\genus}(J), \Q)^\vee.
\end{equation*}
Here $\check\rH^*(\cdot,\Q)$ denotes  \edited{\v{C}ech cohomology} with rational coefficients,
$(\cdot)^\vee \coloneq \Hom(\cdot,\Q)$ denotes the dual vector space, and $\vdim$ is the \defined{virtual dimension} of the moduli space
\begin{equation}
  \label{Eq_VirtualDimension}
  \vdim
  \coloneq
  (\dim X - 6)(1-g) + 2c_1(A)
  =
  2c_1(A);
\end{equation}
cf.~\autoref{Def_CauchyRiemann}.
The VFC is independent of $J$ in the following sense.
If  $\bJ = (J_t)_{t\in[0,1]}$ is a path in $\SpaceOfAlmostComplexStructures$,
then
\begin{equation}
  \label{Eq_VFCPaths}
  [\ModuliSpaceOfNodalMaps_{A,\genus}(J_0)]^\vir =  [\ModuliSpaceOfNodalMaps_{A,\genus}(J_1)]^\vir
  \quad\text{in}~
  \check\rH^\vdim (\ModuliSpaceOfNodalMaps_{A,\genus}(\bJ),\Q)^\vee
\end{equation}
The reader can find the details of this in \cite[Section 9.3]{Pardon2016}.

If $A \in \rH_2(X,\Z)$ is a \defined{Calabi--Yau class};
that is: $c_1(A)=0$, then 
\begin{equation*}
  \vdim = 0
  \qandq
  [\ModuliSpaceOfNodalMaps_{A,\genus}(J)]^{\vir} \in \check\rH^0(\ModuliSpaceOfNodalMaps_{A,\genus}(J), \Q)^\vee.
\end{equation*}
In that case, the \defined{Gromov--Witten invariant} is obtained by pairing the VFC with $1 \in \check\rH^0(\ModuliSpaceOfNodalMaps_{A,\genus}(J),\Q)$:
\begin{equation*}
  \GW_{A,\genus} = \GW_{A,\genus}(X,\omega) \coloneq \int_{[\ModuliSpaceOfNodalMaps_{A,\genus}(J)]^{\vir}} 1 \in \Q
\end{equation*}
for $J\in\SpaceOfAlmostComplexStructures$.
Since $\SpaceOfAlmostComplexStructures$ is path-connected and by \autoref{Eq_VFCPaths}, $\GW_{A,\genus}$ is independent of $\SpaceOfAlmostComplexStructures$.
It is convenient to package these into the \defined{Gromov--Witten series}:
\begin{equation*}
  \GW = \GW(X,\omega) \coloneq \sum_{A \in \Gamma}\sum_{g=0}^\infty \GW_{A,\genus}\cdot t^{2\genus-2} q^A
\end{equation*}
with
\begin{equation*}
  \Gamma
  \coloneq
  \set{
    A \in \rH_2(X,\Z)
    :
    A \neq 0,
    c_1(A) = 0
  }
\end{equation*}
denoting the set of non-zero Calabi--Yau classes.

For $A \in \Gamma$ and $\genus \in \N_0$
if
\begin{equation*}
  \ModuliSpaceOfNodalMaps_{A,\genus}(J) = \coprod_{i \in I} \cA_i
\end{equation*}
is a finite decomposition into open and closed subsets,
then the VFC decomposes accordingly
\begin{equation}
  \label{Eq_DecompositionOfVFC}
  [\ModuliSpaceOfNodalMaps_{A,\genus}(J)]^\vir = \sum_{i \in I} [\cA_i]^\vir;
\end{equation}
see \cite[Lemma 5.2.3]{Pardon2016}.
Therefore,
\begin{equation*}
  \GW_{A,\genus}
  =
  \sum_{i \in I} \GW_{A,\genus}(\cA_i)
  \qwithq
  \GW_{A,\genus}(\cA_i) \coloneq \int_{[\cA_i]^\vir} 1.
\end{equation*}
The number $\GW_{A,\genus}(\cA_i)$ is the \defined{Gromov--Witten contribution of $\cA_i$}.

For the purpose of this article it is convenient to truncate the Gromov--Witten series $\GW$ according to an upper bound $\Lambda$ on the mass, or energy, of pseudo-holomorphic maps. 
For $A \in \rH_2(X,\Z)$ set $\bM(A) \coloneq \Inner{[\omega],A}$.
The \defined{$\Lambda$--truncated Gromov--Witten series} is
\begin{equation*}
  \GW_\Lambda
  =
  \GW_\Lambda(X,\omega)
  \coloneq
  \sum_{A \in \Gamma_\Lambda} \sum_{g=0}^\infty
  \GW_{A,\genus} \cdot t^{2g-2} q^A
\end{equation*}
with
\begin{equation}
  \label{Eq_GammaLambda}  
  \Gamma_\Lambda
  \coloneq
  \set{
    A \in \Gamma
    :
    \bM(A) \leq \Lambda
  }
\end{equation}
denoting the set of non-zero Calabi--Yau classes of mass at most $\Lambda$.
Denote by
\begin{equation*}
  \ModuliSpaceOfNodalMaps_\Lambda
  \coloneq
  \coprod_{A \in \Gamma_\Lambda} \coprod_{\genus = 0}^\infty \ModuliSpaceOfNodalMaps_{A,\genus}
\end{equation*}
the universal moduli space of stable nodal pseudo-holomorphic maps of index zero and mass at most $\Lambda$;
this is an open and closed subset of $\ModuliSpaceOfNodalMaps$.
Moreover,
the subspaces $\ModuliSpaceOfNodalMaps_\Lambda(J)$, $\ModuliSpaceOfSimpleMaps_\Lambda$, $\ModuliSpaceOfEmbeddedMaps_\Lambda$, etc. are defined analogously.
By the preceding discussion,
if
\begin{equation*}
  \ModuliSpaceOfNodalMaps_\Lambda(J) = \coprod_{i \in I} \cA_i
\end{equation*}
is a finite decomposition into open and closed subsets,
then $\GW_\Lambda$ decomposes accordingly
\begin{equation*}
  \GW_\Lambda = \sum_{i \in I} \GW_\Lambda(\cA_i).
\end{equation*}

The upcoming topological lemma describes a method for decomposing $\ModuliSpaceOfNodalMaps_\Lambda(J)$ into open and closed subsets.
It is the foundation of the concept of $\Lambda$--cluster introduced in \autoref{Sec_Clusters} which lies at the heart of the proof of \autoref{Thm_GopakumarVafa}.

\begin{definition}
  Let $\Lambda > 0$, and let $\Gamma_\Lambda$ be as in \autoref{Eq_GammaLambda}. 
  Let $\SpaceOfConnectedCycles$ be the universal space of pseudo-holomorphic cycles with connected support over $\SpaceOfAlmostComplexStructures$,  as in \autoref{Def_PseudoHolomorphicCurve}~\autoref{Def_PseudoHolomorphicCurve_Connected}.
  \begin{enumerate}
  \item
    Set 
    \begin{equation*}
      \SpaceOfConnectedCycles_\Lambda
      \coloneq
      \set{
        (J,C) \in \SpaceOfConnectedCycles : [C] \in \Gamma_\Lambda
      }.
    \end{equation*}
  \item
    For $J \in \SpaceOfAlmostComplexStructures$ and $S \subset  \SpaceOfAlmostComplexStructures$ set
    \begin{equation*}
      \SpaceOfConnectedCycles_\Lambda(J) \coloneq \SpaceOfConnectedCycles_\Lambda \cap \pr_\SpaceOfAlmostComplexStructures^{-1}(J)
      \qandq
      \SpaceOfConnectedCycles_\Lambda(S) \coloneq \SpaceOfConnectedCycles_\Lambda \cap \pr_\SpaceOfAlmostComplexStructures^{-1}(S)
    \end{equation*}
  \item
    Let   $\SpaceOfCompactSubsets$  be the space of compact subsets of $X$, as in \autoref{Def_HausdorffMetric}, and let
    $\supp \co \SpaceOfConnectedCycles \to \SpaceOfCompactSubsets$ be the map from \autoref{Def_JHolomorphicCycles}~\autoref{Def_JHolomorphicCycles_Support}.
    For $J \in \SpaceOfAlmostComplexStructures$ and $S \subset \SpaceOfAlmostComplexStructures$, and $\cU \subset \SpaceOfCompactSubsets$ set
    \begin{equation*}
      \SpaceOfConnectedCycles_\Lambda(J,\cU)
      \coloneq
      \SpaceOfConnectedCycles_\Lambda(J) \cap \supp^{-1}(\cU)
      \qandq
      \SpaceOfConnectedCycles_\Lambda(S,\cU)
      \coloneq
      \SpaceOfConnectedCycles_\Lambda(S) \cap \supp^{-1}(\cU).
      \qedhere
    \end{equation*}
  \end{enumerate}
\end{definition} 

\begin{lemma}[Open-Closed Contribution]
  \label{Lemma_OpenClosedContribution}  
  Let $\Lambda > 0$, $S \subset\SpaceOfAlmostComplexStructures$, and $\cU \subset \SpaceOfCompactSubsets$.
  If $\SpaceOfConnectedCycles_\Lambda(S,\cU)$ is open and closed in $\SpaceOfConnectedCycles_\Lambda(S)$,
  then the following hold:
  \begin{enumerate}
  \item
    \label{Lemma_OpenClosedContribution_Defined}
    For every $J\in S$
    \begin{equation*}
      \ModuliSpaceOfNodalMaps_\Lambda(J;\cU)
      \coloneq
      \fz^{-1}(\SpaceOfConnectedCycles_\Lambda(J,\cU))
    \end{equation*}
    is open and closed in $\ModuliSpaceOfNodalMaps_\Lambda(J)$.
    Here $\fz\co \ModuliSpaceOfNodalMaps \to \SpaceOfConnectedCycles$ is as in \autoref{Def_MapToCycle}.
    In particular, for every $J \in S$, $\cU$ has a Gromov--Witten contribution
    \begin{equation*}
      \GW_\Lambda(\cU, J) \coloneq \GW_\Lambda(\ModuliSpaceOfNodalMaps_\Lambda(\cU;J)).
    \end{equation*}
  \item
    \label{Lemma_OpenClosedContribution_Invariant}
    The Gromov--Witten contribution $\GW_\Lambda(\cdot,\cU)$ of $\cU$ is constant in paths in $S$, that is:
    \begin{equation*}
      \GW_\Lambda(\cU, J_0) = \GW_\Lambda(\cU, J_1)
    \end{equation*}
    for every path $\bJ = (J_t)_{t\in[0,1]}$ in $S$. 
  \end{enumerate}
\end{lemma}

\begin{proof}
  Since $\fz$ is continuous,
  $\fz^{-1}(\SpaceOfConnectedCycles_\Lambda(S,\cU))\subset \ModuliSpaceOfNodalMaps_\Lambda(S)$ is open and closed.
  The same holds for $\set{J}$ and $\bJ$ instead of $S$.
  This implies \autoref{Lemma_OpenClosedContribution_Defined} and,
  together with \autoref{Eq_VFCPaths} and \autoref{Eq_DecompositionOfVFC},
  also \autoref{Lemma_OpenClosedContribution_Invariant}.
\end{proof}



\subsection{Cluster formalism}
\label{Sec_Clusters}

\edited{While the Gromov--Witten generating function $\GW_\Lambda$ is naturally decomposed by decomposing the moduli space $\ModuliSpaceOfNodalMaps_\Lambda(J)$ into open and closed subsets with respect to the Gromov topology, \autoref{Lemma_OpenClosedContribution} allows for a decomposition using open and closed subsets of $\SpaceOfConnectedCycles_\Lambda$, the space of {\em connected pseudo-holomorphic cycles} (of mass at most $\Lambda$). In this section, we construct such a decomposition using the cluster formalism from \cite{Ionel2018}, suitably adapted to framework of pseudo-holomorphic cycles.
The results of this section rely} on the results from \autoref{Sec_SpacesOfPseudoHolomorphicCurves}.
The reader might find it helpful to review the definitions and results from \autoref{Subsec_DefinifionsAndResults}.
In particular, the following facts will be used:
\begin{enumerate}
\item
  For every $\Lambda > 0$, $\SpaceOfConnectedCycles_\Lambda $ is open and closed in $\SpaceOfConnectedCycles$ and 
  the projection map $\pr_{\SpaceOfAlmostComplexStructures} \co\SpaceOfConnectedCycles_\Lambda \to\SpaceOfAlmostComplexStructures$ is proper and, therefore, also closed with respect to the geometric convergence topology;
  see \autoref{Thm_SpaceOfCycles_Proper}.
\item
  The map $\fz \co \ModuliSpaceOfEmbeddedMaps \to \SpaceOfConnectedEmbeddedCurves$ is a homeomorphism with respect to the Gromov topology and the geometric convergence topology respectively;
  see \autoref{Cor_ModuliSpaceOfEmbeddedMapsEmbedsIntoSpaceOfSubmanifolds}.
\item
  For every $A \in \rH_2(X,\Z)$ the map
  $(\pr_\SpaceOfAlmostComplexStructures,\supp) \co \SpaceOfConnectedCurves_{A} \to \SpaceOfAlmostComplexStructures \times \SpaceOfCompactSubsets$ is an embedding with respect to the geometric convergence topology and the topology induced by the Hausdorff metric respectively;
  see \autoref{Prop_SpaceOfCurvesEmbedsIntoSpaceOfCompactSubsets}. 
\end{enumerate} 

\begin{definition}
  \label{Def_Cluster}
  Let $\Lambda > 0$.
  A \defined{$\Lambda$--cluster} is a triple $\cO = (\cU,J,C)$ consisting of
  an open subset $\cU \subset \SpaceOfCompactSubsets$,
  an almost complex structure $J \in \SpaceOfAlmostComplexStructures$, and
  an irreducible, embedded $J$--holomorphic curve $C$,
  the \defined{core} of $\cO$,
  such that:
  \begin{enumerate}
  \item
    \label{Def_Cluster_Boundary}
    There is no $J$--holomorphic curve $C'$ with $\bM(C') \leq \Lambda$ and $\supp C' \in \del \cU \coloneq  \overline\cU\setminus\cU$.
  \item
    \label{Def_Cluster_Homology}
    There is a Calabi--Yau class $A \in \Gamma$ such that for every $J$--holomorphic curve $C'$  with $\bM(C') \leq \Lambda$ and $\supp C' \in \overline\cU$ there is a $k \in \N$ with $[C'] = kA$.
    
    (In particular, every such $C'$ is of index zero.)
  \item
    \label{Def_Cluster_Core}
    $C$ is the unique $J$--holomorphic curve with $\supp C \in \cU$ and $[C] = A$.
    \qedhere
  \end{enumerate} 
\end{definition}

\edited{
This definition should be contrasted with the definition of a cluster in \cite[Definition 22]{Ionel2018} which considers curves of bounded genus and imposes the extra condition that all curves in $\cU$ other than $C$ have strictly greater genus. 
}

\begin{prop}[Cluster Contribution]
  \label{Prop_Cluster_Contribution}
  Let $J \in \SpaceOfAlmostComplexStructures$ and $\Lambda > 0$.
  If an open set $\cU \subset \SpaceOfCompactSubsets$ satisfies \autoref{Def_Cluster}~\autoref{Def_Cluster_Boundary},
  then there is a connected open neighborhood $\cV$ of $J$ in $ \SpaceOfAlmostComplexStructures$ such that the subset $\SpaceOfConnectedCycles_\Lambda(\cV,\cU)$ is open and closed in $\SpaceOfConnectedCycles_\Lambda(\cV)$.  
  In particular, by \autoref{Lemma_OpenClosedContribution},
  for every $J'\in \cV$,
  $\cU$ has a Gromov--Witten contribution $\GW_\Lambda(\cU, J')$ satisfying
  \begin{equation*} 
    \GW_\Lambda(\cU, J') = \GW_\Lambda(\cU, J).
  \end{equation*}
\end{prop}

\begin{notation}
  Let $\Lambda > 0$.
  The Gromov--Witten contribution of a $\Lambda$--cluster $\cO = (\cU,J,C)$ is
  \begin{equation*}
    \GW_\Lambda(\cO) \coloneq \GW_\Lambda(\cU, J).
    \qedhere
  \end{equation*}
\end{notation}

\begin{proof}[Proof of \autoref{Prop_Cluster_Contribution}]
  Since $\partial\cU$ is closed in $\SpaceOfCompactSubsets$, 
  $\SpaceOfConnectedCycles_\Lambda(\SpaceOfAlmostComplexStructures,\del \cU)$ is closed in $\SpaceOfConnectedCycles_\Lambda$ and thus in $\SpaceOfConnectedCycles$.
  Since \edited{$\pr_\SpaceOfAlmostComplexStructures \co \SpaceOfConnectedCycles_\Lambda \to \SpaceOfAlmostComplexStructures$ is proper} by \autoref{Thm_SpaceOfCycles_Proper}, and therefore closed, the set
  \begin{equation*}
    \cV \coloneq \SpaceOfAlmostComplexStructures \setminus \pr_\SpaceOfAlmostComplexStructures(\SpaceOfConnectedCycles_\Lambda(\SpaceOfAlmostComplexStructures,\del \cU))
  \end{equation*}
  is open;
  moreover, $J \in \cV$ because it satisfies \autoref{Def_Cluster}~\autoref{Def_Cluster_Boundary}.
  By construction,
  $\SpaceOfConnectedCycles_\Lambda(\cV,\del \cU) = \emptyset$.
  Therefore,
  \begin{equation}\label{contrib.constant} 
    \SpaceOfConnectedCycles_\Lambda(\cV,\cU)
    =
    \SpaceOfConnectedCycles_\Lambda(\cV,\overline{\cU})
  \end{equation}
  is open and closed.
  Finally, replace $\cV$ with its connected component containing $J$.
\end{proof}

The goal of the cluster formalism is to decompose the space of $J$--holomorphic cycles of mass at most $\Lambda$ into finitely many $\Lambda$--clusters with the aim of analysing the Gromov--Witten contribution of each cluster.
This can be done provided $J$ belongs to the following class of generic almost complex structures.

\begin{definition}
  \label{Def_JIsol}
  Denote by $\SpaceOfAlmostComplexStructuresDiscrete$ the subset of those $J \in \SpaceOfAlmostComplexStructures$ for which:
  \begin{enumerate}
  \item
    \label{Def_JIsol_NonNegative}
    Every simple $J$--holomorphic map has non-negative index.
  \item
    \label{Def_JIsol_Embedded}
    Every simple $J$--holomorphic map of index zero is an embedding.
  \item
    \label{Def_JIsol_Disjoint}
    Every pair of distinct simple $J$--holomorphic maps of index zero have disjoint images or are related by a reparametrization.
  \item
    \label{Def_JIsol_Discrete}
    The moduli space of simple  $J$--holomorphic maps of index zero is discrete with respect to the Gromov topology.
  \end{enumerate}
  Denote by $\SpaceOfAlmostComplexStructuresRegular$ the subset of those $J \in \SpaceOfAlmostComplexStructures$ satisfying \autoref{Def_JIsol_NonNegative}, \autoref{Def_JIsol_Embedded}, \autoref{Def_JIsol_Disjoint}, and---instead of \autoref{Def_JIsol_Discrete}---the stronger condition:
  \begin{enumerate}
  \item[(4$^+$)]
    \label{Def_J*}    
    Every simple $J$--holomorphic map of index zero is unobstructed;
    cf.~\autoref{Def_CauchyRiemann}.
    \qedhere
  \end{enumerate}
\end{definition}

\begin{prop}[{\cite[Lemma 1.2]{Ionel2018}}]
  \label{Prop_JReg_Comeager}
  The subset $\SpaceOfAlmostComplexStructuresRegular$ is comeager in $\SpaceOfAlmostComplexStructures$.
\end{prop}

Recall that a subset of a topological space is \defined{comeager} if it contains a countable intersection of open dense subsets.
By the Baire category theorem, a comeager subset of a complete metric space is dense.
This applies, in particular, to $\SpaceOfAlmostComplexStructures$.

The significance of $\SpaceOfAlmostComplexStructuresDiscrete$ stems from the following results and the fact that it is path-connected---while the complement of $\SpaceOfAlmostComplexStructuresRegular$ in $\SpaceOfAlmostComplexStructures$ is of codimension one and, therefore, $\SpaceOfAlmostComplexStructuresRegular$ is not path-connected; cf.~\cite[Corollary 6.6]{Ionel2018}.

\begin{notation}
  For $A,B \in \rH_2(X,\Z)$ write
  \begin{equation*}
    B | A
  \end{equation*}
  if there is a $k \in \N$ with $A = kB$. 
\end{notation}

\edited{Recall that $\SpaceOfConnectedCycles$ is the space of connected pseudo-holomorphic cycles, with subspaces $\SpaceOfConnectedCurves$ of connected simple cycles and $\SpaceOfConnectedEmbeddedCurves$ of connected embedded cycles; see \autoref{Def_PseudoHolomorphicCurve}.}

\begin{lemma}[Clustering Behaviour]
  \label{Lem_Cluster}
  For every $J \in \SpaceOfAlmostComplexStructuresDiscrete$ and $A \in \Gamma$ the following hold:
  \begin{enumerate}
  \item
    \label{Lem_Cluster_Embedded}
    $\SpaceOfConnectedCurves_A(J) = \SpaceOfConnectedEmbeddedCurves_A(J)$.    
  \item
    \label{Lem_Cluster_Countable}
    $\SpaceOfConnectedEmbeddedCurves_A(J)$ is countable.
  \item
    \label{Lem_Cluster_Discrete}
    $\supp{\SpaceOfConnectedEmbeddedCurves_A(J)}$ is discrete.
  \item    
    \label{Lem_Cluster_Decomposition+Compact}
    $\supp{\SpaceOfConnectedCycles_A(J)}
    =
    \bigcup_{B | A} \supp \SpaceOfConnectedEmbeddedCurves_B(J)$;
    in particular, the latter is compact.
  \end{enumerate}
\end{lemma}

\begin{proof}
  Let $(J, C)\in \SpaceOfConnectedCycles$ with $C= \sum_{i=1}^I m_iC_i$ and $[C]=A\in \Gamma$,
  that is, $[C]$ is a Calabi--Yau class.
  By \autoref{Def_JIsol}~\autoref{Def_JIsol_NonNegative} and \autoref{Eq_VirtualDimension}, $[C_1],\ldots,[C_I]$ must also be Calabi--Yau classes.
  By \autoref{Def_JIsol}~\autoref{Def_JIsol_Embedded} and \autoref{Def_JIsol_Disjoint}, $I=1$;
  that is: $C = m_1C_1$ and $C_1$ is embedded.
  This implies \autoref{Lem_Cluster_Embedded} and \autoref{Lem_Cluster_Decomposition+Compact}.

  By \autoref{Def_JIsol} \autoref{Def_JIsol_Discrete},
  $\ModuliSpaceOfEmbeddedMaps_A(J) = \ModuliSpaceOfSimpleMaps_A(J)$ is discrete and, therefore, countable.
  This implies \autoref{Lem_Cluster_Countable}.  
  By \autoref{Cor_ModuliSpaceOfEmbeddedMapsEmbedsIntoSpaceOfSubmanifolds} and \autoref{Prop_SpaceOfCurvesEmbedsIntoSpaceOfCompactSubsets},
  the map $\im \co \ModuliSpaceOfEmbeddedMaps_A(J) \to \supp \SpaceOfConnectedEmbeddedCurves_A(J)$ is a homeomorphism.
  This implies \autoref{Lem_Cluster_Discrete}.
\end{proof}

\edited{The following three propositions are adaptations of \cite[Lemma 2.3, Proposition 2.4, Corollary 2.5]{Ionel2018} to our setting.}

\begin{prop}[Cluster Existence]
  \label{Prop_Cluster_Existence}
  Let $J \in \SpaceOfAlmostComplexStructuresDiscrete$ and $\Lambda > 0$.
  Let $C$ be an irreducible, embedded $J$--holomorphic curve of index zero with $\bM(C) \leq \Lambda$.
  There is an $\epsilon_0 > 0$ such that the subset
  \begin{equation*}
    \set{
      \epsilon \in (0,\epsilon_0]
      :
      \cO = (B_\epsilon(C),J,C)
      ~\text{is a $\Lambda$--cluster}
    }
  \end{equation*}
  is open and dense in $(0,\epsilon_0]$.
  Here $B_\epsilon(C)$ denotes the ball of radius $\epsilon$ centered at $C$ in $\SpaceOfCompactSubsets$.
\end{prop}

\begin{proof}
  By \autoref{Lem_Cluster}~\autoref{Lem_Cluster_Discrete},
  there is an $\epsilon_0 > 0$ with $\SpaceOfConnectedCycles_A(J,B_{\epsilon_0}(C)) =\set{(J, C)}$.
  After possibly decreasing $\epsilon_0$,
  $C$ is a deformation retract of $\set{ x \in X : d(x,C) \leq \epsilon_0 }$.
  Therefore if $C'$ is a $J$--holomorphic curve with \edited{ $d_H(\supp C,\supp C') \leq \epsilon_0$},
  then $[C'] = k[C]$ with $k \in \N$.

  By \autoref{Lem_Cluster}~\autoref{Lem_Cluster_Countable} and \autoref{Thm_SpaceOfCycles_Proper},
  \edited{
  \begin{equation*}
    \Delta
    \coloneq
    \set{
      d_H(\supp C,\supp C')
      :
      (J, C') \in \SpaceOfConnectedEmbeddedCurves_\Lambda(J)
    }
  \end{equation*}
  }
  is countable and compact.
  Consequently,
  $(0,\epsilon_0] \setminus \Delta$ is open and dense.
\end{proof}

\begin{prop}[Cluster Decomposition]
  \label{Prop_Cluster_Decomposition}
  Let $J \in \SpaceOfAlmostComplexStructuresDiscrete$ and $\Lambda > 0$.
  Let $\cU \subset \SpaceOfCompactSubsets$ be such that $\SpaceOfConnectedCycles_\Lambda(J,\cU)$ is open and closed in $\SpaceOfConnectedCycles_\Lambda(J)$.
  There is a finite set $\set{ \cO_i = (\cU_i,J,C_i) : i \in I}$ of $\Lambda$--clusters such that
  \begin{equation*}
    \SpaceOfConnectedCycles_\Lambda(J,\cU) = \coprod_{i \in I} \SpaceOfConnectedCycles_\Lambda(J,\cU_i);
  \end{equation*}
  in particular, 
  \begin{equation*}
    \GW_\Lambda(\cO)
    =
    \sum_{i \in I} \GW_\Lambda(\cO_i).
  \end{equation*} 
\end{prop}

\begin{proof}
  For every $d \in \N$ set
  \begin{equation*}
    \cN_d
    \coloneq
    \bigcup \supp \SpaceOfConnectedEmbeddedCurves_A(J) \cap \cU
    \subseteq
    \supp \SpaceOfConnectedCycles_\Lambda(J, \cU).
  \end{equation*}
  with the union taken over those $A \in \Gamma_\Lambda$ with divisibility at most $d$.
  Since $\SpaceOfConnectedCycles_\Lambda(J)$ is compact by \autoref{Thm_SpaceOfCycles_Proper},
  only finitely many $A \in \Gamma_\Lambda$ are represented by $J$--holomorphic curves.
  Therefore, these unions are finite and the sequence $\cN_1 \subseteq \cN_2 \subseteq \ldots$ eventually becomes constant.
  
  By \autoref{Lem_Cluster}~\autoref{Lem_Cluster_Discrete} and \autoref{Lem_Cluster_Decomposition+Compact},
  $\cN_1$ is discrete and compact;
  hence: finite.
  Enumerate $\cN_1$ as $\set{ C_1,\ldots,C_{n_1} }$.  
  For every $i \in \set{ 1,\ldots, n_1 }$ choose $\epsilon_i > 0$---by means of \autoref{Prop_Cluster_Existence}---such that for $\cU_i \coloneq B_{\epsilon_i}(C_i) \subset \cU$ the triple $\cO_i \coloneq (\cU_i,J,C_i)$ is a $\Lambda$--cluster, and
  $\cU_i \cap \cU_j = \emptyset$ for $i \neq j \in \set{ 1,\ldots, n_1 } $.
  Set
  \begin{equation*}
    \hat \cN_2 \coloneq \cN_2  \setminus \coprod_{i=1}^{n_1} B_{\epsilon_i}(C_i).
  \end{equation*}  
  By \autoref{Lem_Cluster}~ \autoref{Lem_Cluster_Discrete} and \autoref{Lem_Cluster_Decomposition+Compact},
  $\hat\cN_2$ is discrete and compact;
  hence: finite.
  Enumerate $\hat \cN_2$ as $\set{ C_{n_1+1},\ldots,C_{n_2} }$.
  For $i \in \set{ n_1+1,\ldots,n_2 }$ choose $\epsilon_i > 0$ such that for $\cU_i \coloneq B_{\epsilon_i}(C_i)$ the triple $\cO_i \coloneq (\cU_i,J,C_i)$ is a $\Lambda$--cluster, and
  $\cU_i \cap \cU_j = \emptyset$ for $i \neq j \in \set{ 1,\ldots, n_2 }$.
  Continuing in this fashion constructs the desired decomposition.
\end{proof}

\begin{prop}[Cluster Refinement]
  \label{Prop_Cluster_Refinement}
  Let $J \in \SpaceOfAlmostComplexStructuresDiscrete$ and $\Lambda > 0$.
  The following hold:
  \begin{enumerate}
  \item
    \label{Prop_Cluster_Refinement_Common}
    If $\cO_0 = (\cU_0,J,C)$ and $\cO_0 = (\cU_1,J,C)$ are $\Lambda$--clusters with identical cores,
    then there exists a $\cU \subset \cU_0 \cap \cU_1$ such that $\cO = (\cU,J,C)$ is a $\Lambda$--cluster.
  \item    
    \label{Prop_Cluster_Refinement_Splitting}
    If $\cO_+ = (\cU_+,J,C)$ and $\cO_- = (\cU_-,J,C)$ are $\Lambda$--clusters with identical cores and $\cU_- \subset \cU_+$,
    then there is a finite set $\set{ \cO_i = (\cU_i,J,C_i) : i \in I }$ of $\Lambda$--clusters such that
    \begin{equation*}
      \SpaceOfConnectedCycles_\Lambda(J,\cU_+)
      =
      \SpaceOfConnectedCycles_\Lambda(J,\cU_-)
      \amalg
      \coprod_{i \in I} \SpaceOfConnectedCycles_\Lambda(J,\cU_i),
    \end{equation*}
    and, for every  $i \in I$, $[C_i] = d_i[C]$ with $d_i \geq 2$;
    in particular, 
    \begin{equation*} 
      \GW_\Lambda(\cO_+)
      =
      \GW_\Lambda(\cO_-) + \sum_{i \in I}\GW_\Lambda(\cO_i).
    \end{equation*} 
  \end{enumerate}
\end{prop} 

\begin{proof}
  Since $\cU_0,\cU_1$ are open neighborhoods of $C$,
  \autoref{Prop_Cluster_Existence} implies \autoref{Prop_Cluster_Refinement_Common}.

  To prove \autoref{Prop_Cluster_Refinement_Splitting},
  observe the following.
  Since $\SpaceOfConnectedCycles_\Lambda(J,\cU_+) \setminus \SpaceOfConnectedCycles_\Lambda(J,\cU_-)$ is open and closed in $\SpaceOfConnectedCycles_\Lambda(J)$, 
  \autoref{Prop_Cluster_Decomposition} constructs $\set{ \cO_i = (\cU_i,J,C_i): i \in I }$.
  Since $\cO_+$ is a $\Lambda$--cluster,
  the cores $C_i$ must satisfy $[C_i] = d_i[C]$ with $d_i \geq 2$.
\end{proof}

\begin{prop}[Cluster Stability]
  \label{Prop_Cluster_Stability}
  Let  $\Lambda > 0$ and $J\in \SpaceOfAlmostComplexStructures$.
  Let $\cO = (\cU,J,C)$ be a $\Lambda$--cluster.
  Set $A \coloneq [C]$ and $\genus \coloneq \genus(C)$.
  For every $0 < \epsilon \ll 1$
  there is an open neighborhood $\cV$ of $J$ in $\SpaceOfAlmostComplexStructures$ such that
  for every $J' \in \cV$ \autoref{Def_Cluster}~\autoref{Def_Cluster_Boundary} and \autoref{Def_Cluster_Homology} hold; that is:
  \begin{enumerate}
  \item
    \label{Prop_Cluster_Stability_Boundary}
    There is no $J'$--holomorphic curve $C'$ with $\bM(C') \leq \Lambda$ and $\supp{C'} \in \del\cU$.
  \item
    \label{Prop_Cluster_Stability_Homology}
    For every $J'$--holomorphic curve $C'$ with $\bM(C') \leq \Lambda$ and $\supp{C'} \in \overline\cU$ there is a $k \in \N$ with $[C'] = kA$.
  \end{enumerate}
  Moreover:
  \begin{enumerate}[resume]
  \item
    \label{Prop_Cluster_Stability_Model}
    \edited{Let $B_\epsilon(C)$ be the ball of radius $\epsilon$ centered at $C$ in $\SpaceOfCompactSubsets$.}
    The map $\fz \co \ModuliSpaceOfEmbeddedMaps_{A,\genus}(\cV,B_\epsilon(C)) \to \SpaceOfConnectedCycles_A(\cV,\cU)$ is a homeomorphism;
    in particular:
    $\SpaceOfConnectedCycles_A(\cV,\cU) = \SpaceOfConnectedEmbeddedCurves_A(\cV,\cU)$.
  \end{enumerate}
\end{prop}

\begin{proof}
  The subset
  \begin{equation*}
    \Delta
    \coloneq
    \set{
      (J',C') \in \SpaceOfConnectedCycles_\Lambda
      :
      \supp{C'} \in \del\cU
      ~\text{or}~
      (\supp{C'} \in \overline\cU
      ~\text{and}~
      [C'] \notin \N\cdot[C])
    }
  \end{equation*}
  is closed because the map $(\supp,[\cdot]) \co \SpaceOfConnectedCycles_\Lambda \to \SpaceOfCompactSubsets \times \rH_2(X,\Z)$ is continuous.
  Since $\pr_\SpaceOfAlmostComplexStructures \co \SpaceOfConnectedCycles_\Lambda \to \SpaceOfAlmostComplexStructures$ is closed by  \autoref{Thm_SpaceOfCycles_Proper}, 
  $\pr_\SpaceOfAlmostComplexStructures(\Delta)$ is closed.
  Set
  \begin{equation*}
    \cV
    \coloneq
    \SpaceOfAlmostComplexStructures \setminus \pr_\SpaceOfAlmostComplexStructures(\Delta).
  \end{equation*}
  By construction,
  $J \in \cV$,
  $\cV$ is open,
  and 
  $\Delta \cap \pr_\SpaceOfAlmostComplexStructures^{-1}(\cV) = \emptyset$.
  This proves \autoref{Prop_Cluster_Stability_Boundary} and \autoref{Prop_Cluster_Stability_Homology}.

  In light of \autoref{Cor_ModuliSpaceOfEmbeddedMapsEmbedsIntoSpaceOfSubmanifolds},
  it suffices to show that $\fz \co \ModuliSpaceOfEmbeddedMaps_{A,\genus}(\cV,B_\epsilon(C)) \to \SpaceOfConnectedCycles_A(\cV,\cU)$ is surjective to prove \autoref{Prop_Cluster_Stability_Model}.  
  Consider $(J_n,C_n) \in \paren{\SpaceOfConnectedCycles_A(\SpaceOfAlmostComplexStructures,\cU)}^\N$ with $(J_n)$ converging to $J$.  
  By \autoref{Thm_SpaceOfCycles_Proper}, $(C_n)$ geometrically converges to a $J$--holomorphic cycle $C'$ with $(J, C') \in \SpaceOfConnectedCycles_A$ and $\supp{C'} \in \overline\cU$.
  Since $\cO = (\cU,J,C)$ is a $\Lambda$--cluster,
  $C' = C$.
  Therefore, $\fz$ is surjective---possibly after shrinking $\cV$.
\end{proof}

\begin{prop}[Cluster Perturbation]
  \label{Prop_Cluster_Perturbation}
  Let $\Lambda > 0$ and $J\in \SpaceOfAlmostComplexStructures$.  
  Let $\cO = (\cU,J,C)$ be a $\Lambda$--cluster.
  There is a connected open neighborhood $\cV$ of $J$ in $\SpaceOfAlmostComplexStructures$ such that the subset $\SpaceOfConnectedCycles_\Lambda(\cV,\cU)$ is open and closed in $\SpaceOfConnectedCycles_\Lambda(\cV)$ and the following hold:
  \begin{enumerate}
  \item
    \label{Prop_Cluster_Perturbation_Unobstructed}
    If $C$ is unobstructed,
    then for every $J' \in \cV$ there is a unique $J'$--holomorphic curve $C'$ such that $\cO' = (\cU,J',C')$ is a $\Lambda$--cluster.
  \item
    \label{Prop_Cluster_Perturbation_Splitting}
    For every $J' \in \SpaceOfAlmostComplexStructuresRegular \cap \cV$ (which is non-empty by \autoref{Prop_JReg_Comeager}) there is a finite set $\set{ \cO_i = (\cU_i,J',C_i) : i \in I}$ of $\Lambda$--clusters such that
    \begin{equation*}
      \SpaceOfConnectedCycles_\Lambda(J',\cU)
      =
      \coprod_{i \in I}
      \SpaceOfConnectedCycles_\Lambda(J',\cU_i),
    \end{equation*}
    and, for every $i \in I$, $[C_i] = d_i [C]$ with $d_i \geq 1$ and
    $C_i$ is unobstructed with respect to $J'$;
    in particular,
    \begin{equation*}
      \GW_\Lambda(\cO)
      =\sum_{i\in I} \GW_\Lambda(\cO_i).
    \end{equation*}
  \end{enumerate}
\end{prop}

\begin{proof}
  Let $\cV$ be the connected component of $J$ of the open subset constructed in \autoref{Prop_Cluster_Stability}.
  By \autoref{Prop_Cluster_Stability}~\autoref{Prop_Cluster_Stability_Boundary} and the argument in the proof of \autoref{Prop_Cluster_Contribution},
  $\SpaceOfConnectedCycles_\Lambda(\cV,\cU)$ is open and closed in $\SpaceOfConnectedCycles_\Lambda(\cV)$.

  By the deformation theory of pseudo-holomorphic maps,
  if $C$ is unobstructed,
  then for $0 < \epsilon \ll 1$ and after possibly shrinking $\cV$
  the map
  $\pr_\SpaceOfAlmostComplexStructures \co \ModuliSpaceOfEmbeddedMaps_{[C],\genus}(\cV,B_\epsilon(C)) \to \cV$
  is a diffeomorphism. 
  Therefore,
  by \autoref{Prop_Cluster_Stability}~\autoref{Prop_Cluster_Stability_Model},
  for every $J' \in \cV$ there is a unique $J'$--holomorphic curve $C'$ with $\supp C' \in \cU$ and $[C'] = [C]$.
  This proves \autoref{Prop_Cluster_Perturbation_Unobstructed}.

  \autoref{Prop_Cluster_Perturbation_Splitting} is a consequence of \autoref{Prop_Cluster_Decomposition}.
\end{proof}

 In applications, it is convenient to choose the open sets $\cU$ and $\cV$ in \autoref{Prop_Cluster_Stability} and \autoref{Prop_Cluster_Perturbation} to be arbitrarily small.  
 
\begin{prop}
  \label{Cor_LocalBasisInCurveSpace}
  Let $J\in\SpaceOfAlmostComplexStructuresDiscrete$ and $\Lambda > 0$.
  Let $C$ be an irreducible, embedded $J$--holomorphic curve of index zero with $\bM(C) \leq \Lambda$.
  Set $A \coloneq [C]$ and $\genus \coloneq \genus(C)$.
  In this situation,
  there exists a basis of open neighborhoods of $(J,C)$ in $\SpaceOfConnectedCycles_A$ consisting of subsets of the form $\SpaceOfConnectedCycles_A(\cV, \cU)$ such that:
  \begin{enumerate}
  \item
    $\cV$ is open in $\SpaceOfAlmostComplexStructures$ and $\cU$ is open in  $\SpaceOfCompactSubsets$
   (we can take $\cU = B_\epsilon(C)$).
  \item
    $\cO = (\cU, J,C)$ is a $\Lambda$--cluster,
    and 
    \begin{equation}
      \GW_\Lambda(\cU,J') = \GW_\Lambda(\cO)
    \end{equation}
    for every $J'\in \cV$.
  \item
    The maps
    \begin{equation} \label{c.is.embedding}
      \ModuliSpaceOfEmbeddedMaps_{A, \genus}(\cV,\cU) \to \SpaceOfConnectedCycles_A(\cV,\cU) \qandq \SpaceOfConnectedEmbeddedCurves_A(\cV,\cU) \to \supp \SpaceOfConnectedEmbeddedCurves_A(\cV,\cU)
    \end{equation} 
    are homeomorphisms.
  \end{enumerate}
\end{prop}

\begin{proof}
  Since $C$ is embedded, by \autoref{Cor_ModuliSpaceOfEmbeddedMapsEmbedsIntoSets} 
  the basis of the topology on $\SpaceOfConnectedCycles_A$ at $(J,C)$ consists of subsets of the form
  \begin{equation*}
    \set{
      (J',C') \in \SpaceOfConnectedCycles_A
      :
      J' \in \cV ~\textnormal{and}~
      \supp C' \in B_\epsilon(C)
    }
  \end{equation*}
  with $\epsilon > 0$ and $\cV$ from a basis of open, connected neighborhoods of $J$ in $\SpaceOfAlmostComplexStructures$.  
  Therefore, the corollary follows from \autoref{Prop_Cluster_Stability} and \autoref{Prop_Cluster_Contribution}. Note that \autoref{Cor_ModuliSpaceOfEmbeddedMapsEmbedsIntoSpaceOfSubmanifolds} and \autoref{Prop_SpaceOfCurvesEmbedsIntoSpaceOfCompactSubsets}  imply that 
  $(\pr_\SpaceOfAlmostComplexStructures,\im)\co\ModuliSpaceOfEmbeddedMaps_A(\cV,\cU) \to \cV \times \cU$ is an embedding whose image  is $\supp \SpaceOfConnectedEmbeddedCurves_A(V,\cU)$ and whose domain is equal to $\ModuliSpaceOfEmbeddedMaps_{A,\genus}(\cV,\cU)$ by \autoref{Prop_Cluster_Stability}-\autoref{Prop_Cluster_Stability_Homology}. 
\end{proof}

The crucial result for the proof of \autoref{Thm_GopakumarVafa} is the following. 

\begin{theorem}[Cluster Isotopy]
  \label{Thm_Cluster_Isotopy}
  Let $\Lambda > 0$.
  Let $\cO_0 = (\cU_0, J_0, C)$ and $\cO_1 = (\cU_1,J_1,C)$ be $\Lambda$--clusters with identical cores.
  If $J_0, J_1 \in \SpaceOfAlmostComplexStructuresDiscrete$ and $C$ is unobstructed with respect to $J_0$ and $J_1$, 
  then there is a finite set $\set{ \cO_i = (\cU_i,J_i,C_i) : i \in I}$ of $\Lambda$--clusters such that 
  \begin{equation*}
    \GW_\Lambda(\cO_1)
    =
    \pm\GW_\Lambda(\cO_0) + \sum_{i\in I} \pm \GW_\Lambda(\cO_i);
  \end{equation*}
  moreover, for every $i \in I$, $J_i \in \SpaceOfAlmostComplexStructuresDiscrete$, $C_i$ is unobstructed with respect to $J_i$, and $[C_i] = d_i [C]$ with $d_i \geq 2$.
\end{theorem}


\subsection{Proof of the cluster isotopy theorem}
\label{Sec_ClusterIsotopyTheorem}

This section provides the proof of \autoref{Thm_Cluster_Isotopy}.
\edited{The results of \autoref{Sec_Clusters} are sufficient to carry out the argument from \cite[Section 7]{Ionel2018} with minor changes in notation. 
The main steps in the proof are \autoref{Prop_SimpleIsotopy}, \autoref{Prop_BirthDeathCrossing}, and \autoref{Prop_WallCrossingInMC} which are analogous to \cite[Lemma 7.2, 7.3, 7.4]{Ionel2018}. 
For completeness, we include their proofs with an emphasis on where the results of \autoref{Sec_Clusters} are used.}
In this section, $\Lambda>0$ is fixed and $A \in \rH_2(X,\Z)$ is a Calabi--Yau class satisfying $\bM(A) \leq \Lambda$. 

\begin{notation}
  Following  \cite[Section 7]{Ionel2018},
  given two $\Lambda$--clusters $\cO_0=(\cU_0,J_0,C_0)$ and $\cO_1=(\cU_1,J_1,C_1)$ with $[C_0]=[C_1]$ write
  \begin{equation}
    \label{Eq_EquivalenceOfClusters}
    \GW_\Lambda(\cO_0) \approx \GW_\Lambda(\cO_1)
  \end{equation}
  if there is a finite set $\set{ \cO_i = (\cU_i, J_i, C_i) : i \in I}$ of $\Lambda$--clusters such that
  \begin{equation*}
    \GW_\Lambda(\cO_1) =\GW_\Lambda(\cO_0) + \sum_{i\in I} \pm \GW_\Lambda(\cO_i),
  \end{equation*}
  and, for every $i \in I$, $[C_i] = d_i [C_0]$ with $d_i \geq 2$.
  Similarly, we will write  $\GW_\Lambda(\cO_0) \approx -\GW_\Lambda(\cO_1)$,  $\GW_\Lambda(\cO_0) \approx 0$, and so on, when the equality holds modulo finitely many contributions of $\Lambda$--clusters with cores representing homology classes $d[C_0]$ with $d\geq 2$. 
\end{notation}

\begin{remark}
  While the notation might suggest otherwise,  it is worth pointing out that \autoref{Eq_EquivalenceOfClusters} is actually an equivalence relation of $\Lambda$--clusters rather than of power series, as the homology class of the core plays an important role in the definition of \autoref{Eq_EquivalenceOfClusters}.
\end{remark}

\begin{notation}
  Following \cite[Section 5]{Ionel2018}, we consider the following subspaces of $\ModuliSpaceOfSimpleMaps$:
  \begin{align*}
    \cW &\coloneq \set{ (J,[u]) \in \ModuliSpaceOfSimpleMaps :  \dim\ker\fd_{J,u} > 0 } = \bigcup_{k\in\N}\cW^k, \\
    \text{with } \cW^k &\coloneq \set{ (J,[u]) \in \ModuliSpaceOfSimpleMaps : \dim\ker\fd_{J,u} = k},
  \end{align*}
  with $\fd_{J,u}$ being the linearization of the Cauchy--Riemann operator; see \autoref{Def_CauchyRiemann}.
  
  Denote by $\cA$ the set of points in $\cW^1$ where the projection map $\pi \co \cW^1 \to \SpaceOfAlmostComplexStructures$ fails to be an immersion, cf. \cite[Section 5.4]{Ionel2018}.
\end{notation}

By \cite[Proposition 5.3]{Ionel2018}, $\cW^1\cap\ModuliSpaceOfSimpleMaps_\Lambda$ is a codimension one submanifold of  $\ModuliSpaceOfSimpleMaps_\Lambda$.
By \cite[Lemma 5.6]{Ionel2018}, $\cA\cap\ModuliSpaceOfEmbeddedMaps_\Lambda$ is a codimension one submanifold of $\cW^1$. 
(Recall that, by definition, all pseudo-holomorphic maps in $\ModuliSpaceOfSimpleMaps_\Lambda$ have index zero.)

\begin{prop}[Simple Isotopy]
  \label{Prop_SimpleIsotopy}
  Let $(J_t,[u_t])_{t\in[0,1]}$ be a path in $\ModuliSpaceOfEmbeddedMaps_{A, \genus}$ disjoint from $\cW$ and such that $J_t \in \SpaceOfAlmostComplexStructuresDiscrete$ for all $t\in[0,1]$. 
  Let $C_i$ be the image of $u_i$ for $i=0,1$. 
  If $\cO_i = (\cU_i,J_i,C_i)$ is a $\Lambda$--cluster for $i=0,1$, 
  then 
  \begin{equation*}
    \GW_\Lambda(\cO_0) \approx \GW_\Lambda(\cO_1).
  \end{equation*} 
\end{prop}

\begin{proof}
  The proof is identical to that of \cite[Lemma 7.2]{Ionel2018}, except that we use a different definition of an $\Lambda$--cluster and invoke the results of \autoref{Sec_Clusters} to control sequences of curves without an a priori genus bound.
  
  Assume that $\cO_i = (\cU_i,J_i,C_i)$ is a $\Lambda$--cluster for $i=0,1$. 
  Let $C_t = \im u_t$; this is a family of curves of genus $\genus$ representing $A$; in particular, of index zero. 
  Since the path $(J_t,[u_t])_{t\in[0,1]}$ is disjoint from $\cW$, it follows from the standard deformation theory for pseudo-holomorphic maps that there is an open neighborhood $\cQ$ of the path $(J_t,[u_t])_{t\in[0,1]}$ in $\ModuliSpaceOfEmbeddedMaps_{A, g}$ such that for every $t\in[0,1]$, 
  \begin{equation}
    \label{Eq_SimpleIsotopyLocalModel}
    \ModuliSpaceOfEmbeddedMaps_{A,\genus}(J_t) \cap \cQ = \set{ (J_t, [u_t])};
  \end{equation}
  see, for example, \cite[Proposition 5.3]{Ionel2018}. 
  Since $\bJ$ is contained in $\SpaceOfAlmostComplexStructuresDiscrete$, for every $t\in[0,1]$ we can apply  \autoref{Cor_LocalBasisInCurveSpace}  to $J_t$ and $C_t$ to produce open subsets $\cU_t \subset \SpaceOfCompactSubsets$ and $\cV_t \subset \SpaceOfAlmostComplexStructures$ with all the properties listed in   \autoref{Cor_LocalBasisInCurveSpace}.
  We may, moreover, choose them in such a way that 
  \begin{equation}
    \label{Eq_SimpleIsotopyNeighborhood}
    \ModuliSpaceOfEmbeddedMaps_{A, \genus}(\cV_t,\cU_t) \subset \cQ.
  \end{equation}
  In particular, for every $t\in[0,1]$:
  \begin{enumerate}
  \item The triple $(\cU_t,J_t,C_t)$ is a $\Lambda$--cluster.
  \item \label{ItSimpleIsotopyHomeomorphism} The map
    \begin{equation*}
      \ModuliSpaceOfEmbeddedMaps_{A, \genus}(\cV_t,\cU_t)\to \SpaceOfConnectedCycles_A(\cV_t,\cU_t)
    \end{equation*}
    is a homeomorphism.
  \item Since $\bJ\subset \SpaceOfAlmostComplexStructuresDiscrete$, by \autoref{Eq_SimpleIsotopyLocalModel}, \autoref{Eq_SimpleIsotopyNeighborhood} and property \autoref{ItSimpleIsotopyHomeomorphism} above, we have
    \begin{equation*}
      \SpaceOfConnectedCycles_A(J_s,\cU_t) = \SpaceOfConnectedEmbeddedCurves_A(J_s,\cU_t) = \set{ (J_s,C_s)}
    \end{equation*}
    for every $s \in [0,1]$ such that $J_s \in \cV_t$. 
  \item Moreover, for every such $s \in [0,1]$,
    \begin{equation*}
      \GW_\Lambda(\cU_t,J_s) = \GW_\Lambda(\cU_t,J_t).
    \end{equation*}
  \end{enumerate} 
  
  Since $[0,1]$ is compact, there are 
  \begin{equation*}
    0 = t_0 < t_1 < \ldots < t_m = 1
  \end{equation*} 
  and $\delta_{t_0},  \ldots, \delta_{t_m} > 0$ such that the intervals $I_i = \set{s \in [0,1] : \abs{s-t_i} < \delta_{t_i} }$ cover $[0,1]$ and for $s \in I_i$ we have $J_s \in \cV_{t_i}$.
  Let $s \in I_i \cap I_{i+1}$.
  It follows from the preceding discussion and \autoref{Prop_Cluster_Refinement} that
  \begin{equation*}
    \GW_\Lambda(\cU_{t_i}, J_{t_i}) = \GW_\Lambda(\cU_{t_i}, J_s) \approx \GW_\Lambda(\cU_{t_{i+1}}, J_s) = \GW_\Lambda(\cU_{t_{i+1}}, J_{t_{i+1}}). 
  \end{equation*}
  We conclude that $ \GW_\Lambda(\cO_0) \approx  \GW_\Lambda(\cU_{t_0},J_{t_0}) \approx \cdots \approx  \GW_\Lambda(\cU_{t_m},J_{t_m}) \approx \GW(\cO_1)$. 
\end{proof}

\begin{prop}[Wall-crossing in $\SpaceOfAlmostComplexStructures$] 
  \label{Prop_BirthDeathCrossing}
  Let $\bJ = (J_t)_{t\in[-1,1]}$ be a $C^1$ path in $\SpaceOfAlmostComplexStructures$, contained in $\SpaceOfAlmostComplexStructuresDiscrete$. 
  Suppose that $\pi\co \ModuliSpaceOfEmbeddedMaps_{A, \genus} \to \SpaceOfAlmostComplexStructures$ is transverse to the path $\bJ$ at a point $p_0=(J_0,[u_0]) \in \cW^1\setminus \cA$. Then:
  \begin{enumerate}
  \item 
    \label{It_BirthDeathCrossingLocalModel}
    There exist $\delta>0$, $\sigma = \pm$, and an open neighborhood  $\cQ$ of $p_0$ in 
    $\ModuliSpaceOfEmbeddedMaps_{A, \genus}$ such that for all $t\ne 0$, 
    \begin{equation}\label{moduli.simple.wall}
      \ModuliSpaceOfEmbeddedMaps_{A, \genus}(J_t)\cap\cQ =
      \begin{cases}
        \set{ p_t^+, p_t^- }  & \text{for } 0 < \abs{t} < \delta, \ \sign t = \sigma, \\
        \emptyset & \text{for } 0 < \abs{t} <\delta, \ \sign t = - \sigma,
      \end{cases}
    \end{equation}
    where $p_t^\pm=(J_t, [u_t^\pm]) \in \ModuliSpaceOfEmbeddedMaps_{A, \genus}\setminus \cW$  and $\lim_{t \to 0^\sigma} p_t^\pm = p_0$. 
  \item
    \label{It_BirthDeathCrossingGWContribution}
    Let $t$ be such that $0 < \abs{t} < \delta$ and $\sign t = \sigma$, and set $C^{\pm}_t = \im u^{\pm}_t$. 
    If  $\cO^{\pm}_t= (\cU_t^\pm, J_t, C_t^\pm)$ are $\Lambda$--clusters, then 
    \begin{equation*}
      \GW_\Lambda(\cO^+_t) \approx - \GW_\Lambda(\cO^-_t). 
    \end{equation*} 
    
    (Note that such $\cU_t^+$ and $\cU_t^-$ exist by \autoref{Prop_Cluster_Existence}.)
  \end{enumerate}
\end{prop}

\begin{proof}
  The proof is identical to that of \cite[Lemma 7.3]{Ionel2018}.
  Part \autoref{It_BirthDeathCrossingLocalModel} is a consequence of the standard local model for the birth-death bifurcation for simple pseudo-holomorphic maps; see, for example, \cite[Theorem 6.2, Corollary 6.3]{Ionel2018}. 

  It remains to prove part \autoref{It_BirthDeathCrossingGWContribution}.
  Suppose without loss of generality that $\sigma = -$ and set 
  \begin{equation*}
    C_0 = \im u_0, \quad C_t^\pm =\im u_t^\pm.
  \end{equation*}  
  Since $\bJ$ is contained in $\SpaceOfAlmostComplexStructuresDiscrete$, we can apply  \autoref{Cor_LocalBasisInCurveSpace}  to $J_0$ and $C_0$ to produce open neighborhoods $\cU \subset \SpaceOfCompactSubsets$ and $\cV \subset \SpaceOfAlmostComplexStructures$ with all the properties listed in \autoref{Cor_LocalBasisInCurveSpace}, and such that  
  \begin{equation}
    \label{Eq_SimpleWallNeighborhood}
    \ModuliSpaceOfEmbeddedMaps_{A, \genus}(\cV,\cU) \subset \cQ.
  \end{equation}
  where $\cQ$ is the open neighborhood in \autoref{moduli.simple.wall}.
  In particular, 
  \begin{enumerate}
  \item The triple $(\cU_0,J_0,C_0)$ is a $\Lambda$--cluster.
  \item The map
    \begin{equation*}
      \fz: \ModuliSpaceOfEmbeddedMaps_{A, \genus}(\cV,\cU) \to \SpaceOfConnectedCycles_A(\cV,\cU)
    \end{equation*}
    is a homeomorphism.
  \item Since $J_t\in \SpaceOfAlmostComplexStructuresDiscrete$, then \autoref{moduli.simple.wall} and \autoref{Eq_SimpleWallNeighborhood} imply
    \begin{equation*}
      \SpaceOfConnectedCycles_A(J_t,  \cU )=   \SpaceOfConnectedEmbeddedCurves_A(J_t, \cU)= 
      \begin{cases}
        \set{C^+_t,\;  C^-_t} & \text{for } -\delta < t < 0, \\
        \emptyset & \text{for } 0 < t < \delta.
      \end{cases} 
    \end{equation*}
    for some $\delta>0$ sufficiently small so that the path $(J_t)_{t\in[-\delta,\delta]}$ is contained in $\cV$. 
  \end{enumerate} 
  
  For $-\delta < t < 0$, let $\cU_t^\pm$ be an open neighborhood of $C_t^\pm$ in $\SpaceOfCompactSubsets$ such that the triples $\cO_t^\pm = (\cU_t^\pm, J_t, C_t^\pm)$ are $\Lambda$--clusters and
  \begin{equation*}
    \cU_t^+ \cap \cU_t^- = \emptyset, \quad \cU_t^\pm \subset \cU.
  \end{equation*}
  Such open neighborhoods exist by \autoref{Prop_Cluster_Existence}. 
  By \autoref{Prop_Cluster_Contribution}, $\SpaceOfConnectedCycles_\Lambda(J_t,\cU_t^\pm)$ is open and closed in $\SpaceOfConnectedCycles_\Lambda(J_t)$. 
  Since  $J_t \in \SpaceOfAlmostComplexStructuresDiscrete$,  \autoref{Prop_Cluster_Decomposition} implies that the set 
  \begin{equation*}
    \SpaceOfConnectedCycles_\Lambda(J_t) \setminus\( \SpaceOfConnectedCycles_\Lambda(J_t,\cU_t^+) \sqcup \SpaceOfConnectedCycles_\Lambda(J_t,\cU_t^-) \)
  \end{equation*}
  has a finite decomposition into $\Lambda$--clusters.
  The preceding discussion shows that the cores of the clusters appearing in this decomposition represent homology classes of the form $d A$ for $d \geq 2$.
  Therefore, for $-\delta < t < 0$,
  \begin{equation*}
    \GW_\Lambda(\cU,J_t) \approx \GW_\Lambda(\cO_t^+) + \GW_\Lambda(\cO_t^-).
  \end{equation*}
  On the other hand, for $0 < t < \delta$, we similarly prove that
  \begin{equation*}
    \GW_\Lambda(\cU,J_t) \approx 0.
  \end{equation*}
  By \autoref{Prop_Cluster_Contribution}, the contribution $\GW_\Lambda(\cU,J_t)$ does not depend on $t\in(-\delta,\delta)$.
  We conclude that
  \begin{equation*}
    \GW_\Lambda(\cO^+_t) \approx - \GW_\Lambda(\cO^-_t).
    \qedhere
  \end{equation*}
\end{proof}

\begin{definition}
  Given an embedded, oriented, closed surface $C \subset M$, denote by $\SpaceOfAlmostComplexStructures_C \subset\SpaceOfAlmostComplexStructures$ the subset consisting of all $J$ for which $C$ is $J$--holomorphic.
\end{definition}

\begin{prop}[{Wall-crossing in $\SpaceOfAlmostComplexStructures_C$}]
  \label{Prop_WallCrossingInMC}
  Let $C \subset M$ be an embedded, oriented, connected, closed surface;
  denote by $\iota \co C \to M$ the inclusion map.
  Let $\bJ= (J_{t})_{t\in[-1,1]}$ be a $C^1$ path in $\SpaceOfAlmostComplexStructures$, contained in $\SpaceOfAlmostComplexStructures_C\cap\SpaceOfAlmostComplexStructuresDiscrete$, and such that the path $(J_t,[\iota])_{t\in[-1,1]}$ in $\ModuliSpaceOfEmbeddedMaps_{A, \genus}$ is transverse to $\cW^1$ at the point $p_0=(J_0,[\iota])\in \cW^1\setminus \cA$. Then there exist a $\delta>0$ such that if $\cO_{\pm} = (\cU_\pm, J_{\pm\delta}, C)$ are  $\Lambda$--clusters, then 
  \begin{equation*}
    \GW_\Lambda(\cO_+) \approx - \GW_\Lambda(\cO_-).
  \end{equation*} 
\end{prop} 

\begin{proof} 
  The proof is identical to that of \cite[Lemma 7.4]{Ionel2018}.  
  Together with the path $\bJ$ we will consider its thickenings $\bbJ$, which can be seen either as a $2$--parameter family $\bbJ = (J_{s,t})_{s,t}$ in $\SpaceOfAlmostComplexStructures$, or as a $1$--parameter family $\bbJ = (\bJ_s)_s$ of paths $\bJ_s$ in $\SpaceOfAlmostComplexStructures$ such that $\bJ_0 = \bJ$. 
  
  The results of  \cite[Section 6, in particular Corollary 6.3]{Ionel2018} provide an explicit Kuranishi model for the family of moduli spaces $\ModuliSpaceOfEmbeddedMaps_{A,\genus}(\bJ_0)$ and $\ModuliSpaceOfEmbeddedMaps_{A,\genus}(\bbJ)$ in a neighborhood of the point $p_0=(J_{00},[\iota_C])$ for a generic thickening $ \bbJ$.
  In this situation, and after reparametrizing the path $\bJ$, the following hold for a generic thickening $\bbJ$:
  \begin{enumerate}
  \item 
    \label{It_KuranishiModel}
    There is an open neighborhood $\cQ$ of $p_0=(J_{00}, [\iota_C])$ in $\ModuliSpaceOfEmbeddedMaps_{A,\genus}$ such that $\cQ \cap \ModuliSpaceOfEmbeddedMaps_{A,\genus}(\bbJ)$ is diffeomorphic to a neighborhood of the point $(0,0,0)$ in the surface
    \begin{equation}
      \label{Eq_EquationForS}
      S = \set{ (s,t,x) \in [-1,1]^3 : s=x(x\pm t) }.
    \end{equation}  
    (Without loss of generality, assume that the above sign is negative).

  \item Under this diffeomorphism,   the projection $\ModuliSpaceOfEmbeddedMaps_{A,\genus}(\bbJ) \to \SpaceOfAlmostComplexStructures$ agrees with 
    \begin{equation*}
      (s,t,x) \mapsto (s,t),
    \end{equation*}
    with the path $\bJ = (J_{0,t})_{t\in[-1,1]}$ corresponding to $\set{ s=0, t\in[-1,1]}$.
  \item Under this diffeomorphism, the set
    \begin{equation*}
      \cW\cap\ModuliSpaceOfEmbeddedMaps_{A,\genus}(\bbJ)  = (\cW^1\setminus\cA)\cap\ModuliSpaceOfEmbeddedMaps_{A,\genus}(\bbJ)
    \end{equation*}
    is identified with the curve 
    \begin{equation*}
      S \cap \set{ (s,t,x) : 2x - t = 0 }
    \end{equation*}
    and its tangent space is identified with $TS \cap \ker( 2\rd x - \rd t )$.
  \end{enumerate}

  In the proof, we will need two additional properties of the generic thickening $\bbJ$.
  \begin{enumerate}
    \setcounter{enumi}{3}
  \item There exists a countable set $\Delta\subset [-1,1]$ such that $\bJ_s$ is a path in $\SpaceOfAlmostComplexStructuresDiscrete$ for all 
    $s \in [-1,1]\setminus \Delta$.
  \end{enumerate}
  
  This can be achieved as in \cite[Lemma 6.5]{Ionel2018}, by perturbing $\bbJ$ as a $1$--parameter family of paths $\bbJ = (\bJ_s)_s$ with $\bJ_0 = \bJ$ fixed.  
  First, as in the proof of \cite[Lemmas 6.4 and 6.5]{Ionel2018} the Sard--Smale Theorem implies that the restriction $(\bJ_{s})_{s\ne 0}$ of a generic thickening $\bbJ$ is transverse to the projection $\pi: \ModuliSpaceOfSimpleMaps_\Lambda\rightarrow \SpaceOfAlmostComplexStructures$ as well as to the restriction of  $\pi$ to all the strata of
  \begin{itemize}
  \item $\ModuliSpaceOfSimpleMaps_\Lambda \setminus \ModuliSpaceOfEmbeddedMaps$,
  \item $\cW^k$ for $k\geq 1$,
  \item $\cA$.
  \end{itemize}
  Therefore, for a generic thickening $\bbJ$ and generic $s\neq 0$, $\ModuliSpaceOfSimpleMaps_\Lambda(\bJ_s)$ is a 1-dimensional manifold, transverse to $\cW^1$ and disjoint from  $\ModuliSpaceOfSimpleMaps_\Lambda\setminus\ModuliSpaceOfEmbeddedMaps$, $\cW^k$ for $k\geq 2$ and $\cA$, which all have codimension at least $2$.
  The local Kuranishi models for $\ModuliSpaceOfSimpleMaps_\Lambda(\bJ_s)$ then imply that the path  $\bJ_s$ is contained in $\SpaceOfAlmostComplexStructuresDiscrete$.
  
  In addition, by \autoref{Prop_Cluster_Existence} and \autoref{Cor_LocalBasisInCurveSpace}, we can guarantee the following.
  
  \begin{enumerate}
    \setcounter{enumi}{4}
  \item 
    \label{It_KuranishiModelForCurveSpace}
    There are open neighborhoods $\cU$ of $C$ in $\SpaceOfCompactSubsets$ and $\cV$ of $J_{00}$ in $\SpaceOfAlmostComplexStructures$ with all the properties listed in \autoref{Cor_LocalBasisInCurveSpace} and such that
    \begin{equation}
      \label{Eq_KuranishiNeighborhood}
      \ModuliSpaceOfEmbeddedMaps_{A, \genus}(\cV,\cU) \subset \cQ,
    \end{equation}
    where $\cQ$ is the open neighborhood of $p_0 = (J_{00},[\iota_C])$ from property \autoref{It_KuranishiModel}.   In particular,  \begin{equation*}
      \fz: \ModuliSpaceOfEmbeddedMaps_{A, \genus}(\cV,\cU) \to \SpaceOfConnectedCycles_A(\cV,\cU)
    \end{equation*}
    is a homeomorphism.
    Therefore, a neighborhood of the point $(J_{00},C)$ in $\SpaceOfConnectedCycles_A(\bbJ,\cU)$ is homeomorphic to a neighborhood of $(0,0,0)$ in the surface $S$ from \autoref{Eq_EquationForS}.
    Without loss of generality we will assume that the entire family $\bbJ$ is contained in $\cV$.
  \end{enumerate}
  
  Since $\bJ \subset \SpaceOfAlmostComplexStructuresDiscrete$, it follows from  \autoref{It_KuranishiModelForCurveSpace}  and  \eqref{Eq_EquationForS} that $\SpaceOfConnectedCycles_A(\bJ,\cU)$ is homeomorphic to a neighborhood of the point $(0,0)$ in 
  \begin{equation*}
    \set{ (t,x) \in [-1,1]^2 : 0=x(x-t) }.
  \end{equation*}
  The curve $\set{x=0}$ corresponds to the path $(J_t, C)_{t\in[-1,1]}$, while the curve $\set{x = t}$ corresponds to another $1$--parameter family $C_t'$ of irreducible, embedded $J_t$--holomorphic curves representing the class $A$. 
  Note that $C_t'\ne C$ for $t\ne 0$, and $C_0'=C$. 
  It follows from \autoref{It_KuranishiModelForCurveSpace} and the Kuranishi model \eqref{Eq_EquationForS}  that for  sufficiently small $\delta>0$ and $s\neq 0$,  
  \begin{equation*}
    \SpaceOfConnectedCycles_A(J_{s, \pm \delta}, \cU)= \set{p_{s,\pm}, \; p'_{s,\pm}} 
  \end{equation*}
  consists of two points, which, as $s \to 0$, converge to 
  \begin{equation*} 
    p_\pm= (J_{0,\pm \delta}, C) \qandq p_\pm'= (J_{0,\pm \delta}, C_{\pm \delta}'). 
  \end{equation*} 
  Since the path $\bJ=\bJ_0$ is contained in $\SpaceOfAlmostComplexStructuresDiscrete$ and $C_t' \neq C$ for all $t\neq 0$,  by \autoref{Prop_Cluster_Existence} there are open neighborhoods $\cU_\pm$ of $C$ and $\cU_\pm'$ of $C_{\pm\delta}'$  in $\SpaceOfCompactSubsets$ such that the triples
  \begin{equation*}
    \cO_{\pm}= (\cU_\pm,J_{0,\pm\delta}, C) \qandq \cO_\pm' = (\cU_\pm' ,J_{0,\pm\delta}, C'_{\pm\delta})
  \end{equation*}
  are $\Lambda$--clusters and
  \begin{equation*}
    \cU_\pm \subset \cU, \quad \cU_\pm' \subset \cU, \qandq \cU_\pm \cap \cU_\pm' = \emptyset.
  \end{equation*}
  Combining  the local description of the cycle space $\SpaceOfConnectedEmbeddedCurves_A(\bbJ)$, given by properties \autoref{It_KuranishiModel} and \autoref{It_KuranishiModelForCurveSpace}, with \autoref{Prop_Cluster_Refinement} and \autoref{Prop_Cluster_Stability}, we obtain 
  \begin{equation}
    \label{Eq_TotalContribution}
    \GW_\Lambda(\cO_-) + \GW_\Lambda(\cO_-') \approx \GW_\Lambda(\cU, J_{0,-\delta}) = \GW_\Lambda(\cU,J_{0,\delta}) \approx \GW_\Lambda(\cO_+) + \GW_\Lambda(\cO_+').
  \end{equation} 
  We will show that 
  \begin{equation}
    \label{Eq_ClusterComparison}
    \GW_\Lambda(\cO_-) \approx \GW_\Lambda(\cO_+')  \approx - \GW_\Lambda(\cO_-'),
  \end{equation}
  which, in conjunction with \autoref{Eq_TotalContribution},  will complete the proof.
  
  We  prove \autoref{Eq_ClusterComparison} by considering the restrictions of the local  Kuranishi model from property \autoref{It_KuranishiModel} over the paths $\bJ_s$.
  It follows from \autoref{Eq_EquationForS} that for $s\neq 0$, $\ModuliSpaceOfEmbeddedMaps_A(\bJ_s,\cU)$ is a $1$--dimensional manifold such that: 
  \begin{enumerate}[i.] 
  \item For $s\ne 0$, $\ModuliSpaceOfEmbeddedMaps_A(\bJ_s,\cU)$ has two connected components, one corresponding to $x>0$ and the other to $x<0$ in the description provided by property \autoref{It_KuranishiModel}.
  \item
    \label{It_PositivePath}
    For $s>0$, the component with $x>0$ is a path in $\ModuliSpaceOfEmbeddedMaps_A$ disjoint from the wall $\cW,$ and the projection to $\SpaceOfAlmostComplexStructures$ is a injective when restricted to this path. For $s$ small, this path intersects the fiber over $J_{s, - \delta}$ in $p_{s,-}$ and the fiber over $J_{s, - \delta}$ in $p_{s,+}'$.
  \item 
    \label{It_NegativePath}
    For $s<0$, the component with $x>0$ is a path in $\ModuliSpaceOfEmbeddedMaps_A$ which intersects the wall transversally in precisely one point $q_s$ with 
    $t=x/2=\sqrt{-s}$, and $q_s\in \cW^1\setminus \cA$. For $s<0$ small, the intersection of this path with the fiber over $J_{s, \delta}$ consists of the points $p_{s,+}$ and $p_{s,+}'$. 
  \end{enumerate} 
  For small $s\not\in \Delta$, $\bJ_s$ is a path in $\SpaceOfAlmostComplexStructuresDiscrete$, as was the path $\bJ$.
  Let $\cU_\pm$ and $\cU_\pm'$ be four open sets in $\SpaceOfCompactSubsets$ appearing in the definition of the $\Lambda$--clusters $\cO_\pm$ and $\cO_\pm'$ above.  
  It follows from the preceding discussion that for $s$ sufficiently small,
  \begin{equation*}
    \SpaceOfConnectedCycles_A(J_{s, \pm \delta},\cU_\pm) = p_{s, \pm}, \qandq   \SpaceOfConnectedCycles_A(J_{s, \pm \delta},\cU_\pm') = p'_{s, \pm}. 
  \end{equation*} 
  Therefore, by part \autoref{Prop_Cluster_Stability_Boundary} of \autoref{Prop_Cluster_Stability}, the triples
  \begin{equation*}
    \cP_{s,\pm} = (\cU_\pm, J_{s,\pm\delta}, p_{s,\pm}) \qandq \cP_{s,\pm}' = (\cU_\pm', J_{s,\pm\delta}, p_{s,\pm}')
  \end{equation*}
  are $\Lambda$--clusters. 
  (We engage here in a slight abuse of notation by identifying the points $ p_{s,\pm}$ and $ p_{s,\pm}'$ in $\ModuliSpaceOfEmbeddedMaps_\Lambda$ with the corresponding pseudo-holomorphic curves in $\SpaceOfConnectedEmbeddedCurves$.) 
  Moreover,
  \begin{equation*}
    \GW_\Lambda(\cO_\pm) = \GW_\Lambda(\cP_{s,\pm})
    \qandq 
    \GW_\Lambda(\cO_\pm') = \GW_\Lambda(\cP'_{s,\pm}). 
  \end{equation*} 
  
  On the other hand, for $s>0$ small and $s\not\in \Delta$, we get a path of type \autoref{It_PositivePath} which misses  the wall $\cW$, thus  
  \autoref{Prop_SimpleIsotopy} implies that
  \begin{equation*}
    \GW_\Lambda(\cP_{s,+}) \approx \GW_\Lambda(\cP_{s,+}'). 
  \end{equation*}
  For $s<0$ small and $s\not\in \Delta$, we get a path of type \autoref{It_NegativePath}  which intersects the wall $\cW$ transversally at precisely one point $q_s$ in 
  $\cW^1\setminus \cA$, thus  \autoref{Prop_BirthDeathCrossing} implies 
  \begin{equation*}
    \GW_\Lambda(\cP_{s,-}) \approx - \GW_\Lambda(\cP_{s,+}'). 
  \end{equation*}
  Combining the last four displayed equations implies 
  \begin{equation*}
    \GW_\Lambda(\cO_-) \approx \GW_\Lambda(\cO_+')  \approx - \GW_\Lambda(\cO_-'),
  \end{equation*}
  and therefore completes the proof. 
\end{proof}

\begin{proof}[Proof of \autoref{Thm_Cluster_Isotopy}] 
  Set $A \coloneq [C] \in \Gamma$ and $\genus \coloneq \genus(C)$.
  Denote by $\iota \co C \hookrightarrow X$ the inclusion map.
  Since $\SpaceOfAlmostComplexStructures_C$ is path connected  it follows from \cite[proof of Lemma 6.7]{Ionel2018} that there exists a path $\bJ=(J_t)_{t\in[0,1]}$ in $\SpaceOfAlmostComplexStructures_C$ connecting $J_0$ and $J_1$ with the following properties. 
  \begin{itemize} 
  \item The path $(J_t,[\iota])_{t\in[0,1]}$ in $\ModuliSpaceOfEmbeddedMaps_{A, \genus}$ intersects $\cW^1\setminus\cA$ transversely at finitely many points and is otherwise disjoint from $\cW$.
  \item Away from the points of intersection with $\cW^1\setminus\cA$,  for all $d\ge 1$, the subset 
    $\ModuliSpaceOfSimpleMaps_{\Lambda}(\bJ)$ of the moduli space of index $0$ simple maps  is a 1-dimensional manifold, consisting of embeddings, and intersecting the wall transversely at points in $\cW^1\setminus \cA$. 
  \end{itemize} 
  In particular, as in \cite[proof of Lemma 6.7]{Ionel2018}, the local Kuranishi models for $\ModuliSpaceOfSimpleMaps_\Lambda(\bJ)$ imply that the path $\bJ$ is contained in $\SpaceOfAlmostComplexStructuresDiscrete$.
  Therefore the theorem follows by  \autoref{Prop_Cluster_Perturbation}~\autoref{Prop_Cluster_Perturbation_Splitting}, combined with \autoref{Prop_SimpleIsotopy} and \autoref{Prop_WallCrossingInMC} after dividing the path $(J_t, [\iota])_{t\in[0,1]}$  into finitely many paths, each either disjoint from $\cW$ or intersecting $\cW^1\setminus\cA$ transversally at one point.
\end{proof}


\subsection{Contributions of super-rigid curves}
\label{Sec_SuperRigidContributions}

\begin{definition}
  \label{Def_SuperRigid}
  Let $J \in \SpaceOfAlmostComplexStructures$.
  Let $C$ be an irreducible, embedded $J$--holomorphic curve.
  Set $j \coloneq J|_{TC}$.
  \begin{enumerate}
  \item
    The operator $\fd_{C,J}$ \edited{descends} to the \defined{normal Cauchy--Riemann operator}
    \begin{equation*}
      \fd_{C,J}^N \co \Gamma(NC) \to \Omega^{0,1}(C,NC).
    \end{equation*}
  \item
    If $\pi \co (\tilde C,\tilde j) \to (C,j)$ is a nodal $j$--holomorphic map,
    then $\fd_{C,J}^N$ induces
    \begin{equation*}
      \pi^*\fd_{C,J}^N \co \Gamma(\pi^*NC) \to \Omega^{0,1}(\tilde C,\pi^*NC)
    \end{equation*}
    by pulling back; cf.~\cites[§2.2]{Zinger2011}[Definition 1.2.1]{Doan2018}.
  \item
    $C$ is \defined{super-rigid with respect to $J$} if $\ker \pi^*\fd_{C,J}^N = 0$ for every $j$--holomorphic map $\pi$.
    \qedhere
  \end{enumerate}
\end{definition}

Assume the situation of \autoref{Def_SuperRigid}.
Denote by $\HurwitzSpace{d}{\genus}(C)$ the moduli space of stable degree $d$ genus $\genus$ nodal $j$--holomorphic maps to $C$.
The space  $\HurwitzSpace{d}{\genus}(C)$ is an orbispace and parametrizes the family of Fredholm operators
\begin{equation*}
  {\underline\fd}_{C,J}^N
  =
  \paren*{
    \pi^*\fd_{C,J}^N
  }_{[\pi]\in \HurwitzSpace{d}{\genus}}.
\end{equation*}
If $C$ is super-rigid,
then following hold:
\begin{enumerate}
\item
  The cokernels of $\pi^*\fd_{C,J}^N$ form an orbibundle $\coker \underline\fd_{C,J}^N$ over $\HurwitzSpace{d}{\genus}(C)$.
\item
  By \cite[Theorem 1.6]{Doan2018a},
  $\SpaceOfConnectedCycles_\Lambda(J,\set{C})$ is open and closed in $\SpaceOfConnectedCycles_\Lambda(J)$ for every $\Lambda \geq \bM(C)$.
  In particular,
  $\ModuliSpaceOfNodalMaps(J,\set{C})$ is open and closed in $\ModuliSpaceOfNodalMaps(J)$.
\item
  $\ModuliSpaceOfNodalMaps(J,\set{C})$ agrees with $\coprod_{d=1}^\infty \coprod_{\genus=0}^\infty \HurwitzSpace{d}{\genus}(C)$.
\item
  According to \citet[Theorem 1.2]{Zinger2011},
  the Gromov--Witten contribution of $C$ is
  \begin{equation}
    \label{Eq_SuperRigidContribution}
    \GW(C,J)
    \coloneq
    \sum_{\genus=0}^\infty \sum_{d=1}^\infty
    \int_{\sqparen[\big]{\HurwitzSpace{d}{\genus}(C)}^\vir} e(\coker \underline\fd_{C,J}^N)
    \cdot
     t^{2\genus-2} q^{d[C]}.
  \end{equation}
  Here $e(\cdot)$ denotes the Euler class.
\end{enumerate}

\begin{cor}
  \label{Cor_SuperRigidCluster}
  Let $J \in \SpaceOfAlmostComplexStructures$.
  Let $C$ be an irreducible, embedded $J$--holomorphic curve of index zero.
  Let $\Lambda \geq \bM(C)$
  If $C$ is super-rigid,
  then there is an $\epsilon_0 > 0$ such that for every $\epsilon \in (0,\epsilon_0)$ the triple $\cO = (B_\epsilon(C),J,C)$ is a $\Lambda$--cluster and $\GW_\Lambda(\cO) = \GW_\Lambda(C,J)$.
\end{cor}

Computing the contribution $\GW(C,J)$ in \autoref{Eq_SuperRigidContribution} is a formidable problem.
Fortunately, it has been studied extensively by \citet{Bryan2008}.

\begin{definition}
  For $h \in \N_0$ set
  \begin{equation*}
    G_h(q,t)
    \coloneq
    \log{
      \paren[\Bigg]{
        1
        +
        \sum_{d=1}^\infty \sum_{\mu \vdash d}
        \prod_{\square\in \mu}
        \paren*{2\sin(h(\square)\cdot t/2)}^{2h-2} q^d
      }}.
  \end{equation*}
  Here $\mu \vdash d$ indicates that the sum is taken over all partitions $\mu$ of $d$,
  $\square \in \mu$ indicates that $\square$ is a box in the Young diagram of $\mu$, and
  $h(\square)$ denotes the hook length of $\square$. 
\end{definition}

\begin{prop}[{\cites[§2]{Lee2009a}[Propositions 3.2 and 3.3]{Ionel2018}[Corollary 7.3]{Bryan2008}}]
  \label{Prop_J_elementary_exists}
  Let $J \in \SpaceOfAlmostComplexStructures$.
  Let $C$ be an irreducible, embedded $J$--holomorphic curve of index zero.
  There is a $J_L \in \SpaceOfAlmostComplexStructures_C$ with respect to which $C$ is super-rigid and
  \begin{equation*}
    \GW(C,J_L)
    =
    G_h(q^{[C]},t)
    \qwithq
    h
    \coloneq
    \genus(C).
    \pushQED{\qed}
    \qedhere
    \popQED
  \end{equation*}
\end{prop}

The following combinatorial result verifies the Gopakumar--Vafa conjecture for $G_h$.

\begin{prop}[{\cite[Proposition 3.4]{Ionel2018}}]
  \label{Prop_J_elementary_contribution}
  For every $h \in \N_0$ the coefficients $\BPS_{d,\genus}(h)$ defined by
  \begin{equation*}
    G_h(q,t)
    =
    \sum_{d=1}^\infty
    \sum_{\genus = 0}^\infty
    \BPS_{d,\genus}(h)
    \cdot    
    \sum_{k=1}^\infty
    \frac1k
    \paren*{2\sin\paren*{kt/2}}^{2\genus-2}
    q^{kd}
  \end{equation*}
  satisfy:
  \begin{description}[labelwidth=\widthof{\bfseries (integrality)}]
  \item[(integrality)]
    $\BPS_{d,\genus}(h) \in \Z$, and
  \item[(finiteness)]
    $\BPS_{d,\genus}(h) = 0$ for $\genus \gg 1$.
    \pushQED{\qed}
    \qedhere
    \popQED
  \end{description}
\end{prop}

The following structure result for the contribution of a super-rigid $J$--holomorphic curve is a byproduct of the proof of \autoref{Thm_GopakumarVafa} in \autoref{Sec_ProofOfGopakumarVafaConjecture}.

\begin{prop}[Super-rigid Contributions]
  \label{Thm_SuperRigidContribution}
  Let $J \in \SpaceOfAlmostComplexStructures$ and $C$ be a irreducible, embedded $J$--holomorphic curve of index zero and genus $\genus$.
  If $C$ is super-rigid with respect to $J$,
  then
  \begin{equation*}
    \GW(C,J)
    =
    \sign(C,J)\cdot G_\genus (q^{[C]}, t)
    +
    \sum_{d=2}^\infty \sum_{h=\genus}^\infty e_{d,h}(C,J) \cdot G_h(q^{d[C]},t) 
  \end{equation*}
  with
  \begin{description}[labelwidth=\widthof{\bfseries (integrality)}]
  \item[(integrality)]
    $e_{d,h}(C,J) \in \Z$, and
  \item[(finiteness)]
    $e_{d,h}(C,J) = 0$ for $\genus \gg 1$.
  \end{description}
\end{prop}

\begin{remark}
  \label{Rem_SuperRigidity}
  \citet[Theorem A]{Wendl2016} has recently proved that for a generic $J \in \SpaceOfAlmostComplexStructures$ every $J$--holomorphic curve of index zero in a symplectic $6$--manifold is super-rigid.
  Therefore,
  it is interesting to ask whether \autoref{Thm_SuperRigidContribution} can be proved directly.
  An obstacle to this appears to be the lack of understanding of the wall-crossing/bifurcation phenomena related to the failure of super-rigidity along a generic path $\bJ = (J_t)_{t\in[0,1]}$ in $\SpaceOfAlmostComplexStructures$; cf.~\cites[§2.4]{Wendl2016}[§2.7]{Doan2018}. 
  \edited{Since the first version of this article appeared, \citet{Bai2021} made progress in this direction by analyzing the bifurcation phenomena caused by double covers of bounded genus. 
  It is an interesting problem to extend their work to covers of higher degree and unbounded genus.}
\end{remark}


\subsection{Conclusion of the proof of the Gopakumar--Vafa conjecture}
\label{Sec_ProofOfGopakumarVafaConjecture}

\autoref{Thm_GopakumarVafa} is an immediate consequence of \autoref{Prop_J_elementary_contribution} and the following structure theorem. 

\begin{theorem}
  \label{Thm_GWStructure}
  There are unique coefficients $e_{A,\genus} = e_{A,\genus}(X,\omega)$ such that
  \begin{equation}
    \label{Eq_ElementaryExpansion}
    \GW
    =
    \sum_{A \in \Gamma} \sum_{\genus = 0}^\infty e_{A,\genus} \cdot G_\genus(q^A,t);
  \end{equation}
  moreover, they satisfy:
  \begin{description}[labelwidth=\widthof{\bfseries (integrality)}]
  \item[(integrality)]
    $e_{A,\genus} \in \Z$, and
  \item[(finiteness)]
    $e_{A,\genus} = 0$ for $\genus \gg 0$.
  \end{description}
\end{theorem}

\begin{remark}
  There is a version of the question raised in \autoref{Rmk_BPSCastelnuovo} with $\BPS_{A,\genus}$ replaced by $e_{A,\genus}$.
\end{remark}

The proof relies on the following result.

\begin{notation}
  \label{Not_Truncation}
  Consider a formal power series
  \begin{equation*}
    S = \sum_{A \in \Gamma} c_A \cdot q^A.
  \end{equation*}
  For every $\Lambda>0$ the \defined{$\Lambda$--truncation} of $S$ is the formal power series
  \begin{equation*}
    S_\Lambda \coloneq \sum_{A \in \Gamma_\Lambda} c_A \cdot q^A
  \end{equation*}
  with $\Gamma_\Lambda$ as in \autoref{Eq_GammaLambda}.
\end{notation}

\begin{prop}
  \label{Prop_GWStructure_Cluster}
  Let $\Lambda > 0$.
  Let $\cO = (\cU,J,C)$ be a $\Lambda$--cluster with $J\in  \SpaceOfAlmostComplexStructuresDiscrete$ and $C$ unobstructed with respect to $J$.
  Set $d^* \coloneq \floor{\bM(C)/\Lambda}$.
  There are unique coefficients $e_{d,\genus}(\cO)$ such that
  \begin{equation}
    \GW_\Lambda(\cO)=
    \sum_{d = 1}^{d^*} \sum_{\genus = 0}^\infty e_{d,\genus}(\cO) \cdot G_\genus(q^{d[C]}, t)_\Lambda;
  \end{equation}
  moreover, they satisfy:
  \begin{description}[labelwidth=\widthof{\bfseries (integrality)}]
  \item[(integrality)]
    $e_{d,\genus}(\cO) \in \Z$, and
  \item[(finiteness)]
    $e_{d,\genus}(\cO) = 0$ for $\genus \gg 0$.
  \end{description}
\end{prop}

\begin{proof}
  The uniqueness of the coefficients is a consequence of the fact that $G_\genus(q,t) = t^{2\genus-2}q + \textnormal{higher order terms}$.

 Since the core $C$ of $\cO = (\cU,J,C)$ is of index zero, by \autoref{Cor_SuperRigidCluster} and \autoref{Prop_J_elementary_exists} there are $J' \in \SpaceOfAlmostComplexStructures_C$ and an $\epsilon > 0$ such that
  $C$ is super-rigid with respect to $J'$,
  $\cO' \coloneq (B_\epsilon(C),J',C)$ is a $\Lambda$-cluster,
  and
  \begin{equation*}
    \GW_\Lambda(\cO') = G_\genus(q^{[C]}, t)_\Lambda.
  \end{equation*}
  Since $C$ is unobstructed with respect to $J'$,
  by \cite[Proof of Lemma 6.7]{Ionel2018} and \autoref{Prop_Cluster_Stability},
  there is a $J'' \in \SpaceOfAlmostComplexStructures_C\cap \SpaceOfAlmostComplexStructuresRegular$ such that $\cO'' \coloneq (B_\epsilon(C),J'',C)$ is a $\Lambda$-cluster,
  and
  \begin{equation*}
    \GW_\Lambda(\cO'') = \GW_\Lambda(\cO').
  \end{equation*}
  
  By \autoref{Thm_Cluster_Isotopy} with $\cO_0 = \cO$ and $\cO_1 = \cO''$,
  there are $e_{1,\genus}(\cO) \in \set{\pm 1}$ and a finite set $\set{ \cO_i = (\cU_i,J_i,C_i) : i \in I}$ of $\Lambda$--clusters such that 
  \begin{equation*}
    \GW_\Lambda(\cO)
    =
    e_{1,\genus}(\cO) \cdot G_\genus(q^{[C]}, t)_\Lambda
    +
    \sum_{i\in I} \pm \GW_\Lambda(\cO_i);
  \end{equation*}
  moreover, for every $i \in I$, $J_i \in  \SpaceOfAlmostComplexStructuresDiscrete$, $C_i$ is unobstructed with respect to $J_i$, and $[C_i] = d_i [C]$ with $d_i \geq 2$.

  This finishes the proof if $d^* = 1$.
  If $d^* \geq 2$,
  then $d_i^* \coloneq \floor{\bM(C_i)/\Lambda} \leq d^*-1$ and the assertion follows by induction.
\end{proof}

\begin{proof}[Proof of \autoref{Thm_GWStructure}]
  The uniqueness of the coefficients follows as in the proof of \autoref{Prop_GWStructure_Cluster} because 
  $G_\genus(q^A,t) = t^{2\genus-2}q^A + \textnormal{higher order terms}$.
    
  Let $A \in \Gamma$ be a non-zero Calabi-Yau class. 
  Set $\Lambda \coloneq \Inner{[\omega],A}$.
  Let $J \in \SpaceOfAlmostComplexStructuresRegular$.
  By \autoref{Prop_Cluster_Decomposition},
  there is a finite set $\set{ \cO_i = (\cU_i,J,C_i) : i \in I}$ of $\Lambda$--clusters such that
  \begin{equation*}
    \GW_\Lambda = \sum_{i \in I} \GW_\Lambda(\cO_i).
  \end{equation*}
  By \autoref{Prop_GWStructure_Cluster}, 
  \begin{equation*}
    \GW_\Lambda =\sum_{i\in I} \GW_\Lambda(\cO_i)= \sum_{i\in I}\sum_{d=1}^\infty \sum_{\genus=0}^\infty e_{d, \genus}(\cO_i) \cdot G_\genus(q^{d[C_i]}, t)_\Lambda.
  \end{equation*}
  Denote by $I^*$ the subset of those $i \in I$ for which there is a $d_i \in \N$ with $A = d_i[C_i]$.
  By uniqueness of coefficients in \eqref{Eq_ElementaryExpansion}, 
  \begin{equation*}
    e_{A,\genus} = \sum_{i \in I^*} e_{d_i,\genus}(\cO).
  \end{equation*}
  By \autoref{Prop_GWStructure_Cluster},
  these satisfy integrality and finiteness.
\end{proof}

\begin{proof}[Proof of \autoref{Thm_SuperRigidContribution}]
  Denote by $\iota \co C \to M$ the inclusion. Since $C$ is of index zero and is unobstructed with respect to $J$, the contribution of $[\iota] \in \ModuliSpaceOfNodalMaps_{A,\genus}$ to $\GW(C,J)$ is precisely $\sign(C,J) \cdot t^{2\genus-2}q^{[C]}$.  
  The remaining contributions to $\GW(C,J)$ arise from $\HurwitzSpace{d}{h}(C)$ and vanish unless $h \geq \genus$.
  Therefore, the assertion is a consequence of \autoref{Prop_GWStructure_Cluster}.
\end{proof}



\appendix


\section{The Gopakumar–Vafa conjecture for Fano classes} 
\label{Sec_GopakumarVafaFano}

There is an analogue of the Gopakumar--Vafa conjecture for \defined{Fano classes};
that is: $A \in \rH_2(X,\Z)$ with $c_1(A) > 0$.
(\edited{In complex dimension three}, Gromov--Witten theory is trivial for $A \in \rH_2(X,\Z)$ with $c_1(A) < 0$.)
Let $A \in \rH_2(X,Z)$ be a Fano class, $\genus \in \N_0$, and $k \in \N_0$.
Denote by $\ModuliSpaceOfNodalMaps_{A,\genus,k}$ the universal moduli space over $\SpaceOfAlmostComplexStructures$ of stable nodal pseudo-holomorphic maps representing $A$, of genus $\genus$, and with $k$ marked points.
Evaluation at the marked points defines a map
\begin{equation*}
  \ev\co \ModuliSpaceOfNodalMaps_{A,\genus,k} \to X^k. 
\end{equation*} 
As in \autoref{Sec_GromovWittenInvariants},
the fibers of $\ModuliSpaceOfNodalMaps_{A,\genus,k}$ carry a VFC of degree $2c_1(A)+2k$ and these are consistent along paths in $\SpaceOfAlmostComplexStructures$.
If $\gamma_1,\ldots,\gamma_k \in \rH^*(X,\Z)$ satisfy
\begin{equation}
  \label{eq.dim.gamma}
  c_1(A) - \sum_{i=1}^k (\deg \gamma_i-2) = 0,
\end{equation}
then the \defined{Gromov--Witten invariant} is defined by
\begin{equation}
  \label{GW.gamma}
  \GW_{A,\genus}(\gamma_1,\ldots,\gamma_k)
  \coloneq
  \int_{[\ModuliSpaceOfNodalMaps_{A,\genus,k}(J)]^\vir}  \ev^*(\gamma_1\times  \dots \times \gamma_k).
\end{equation}
These can be packaged into a linear map
\begin{equation}
  \label{Eq_GWFano}
  \GW_{A,\genus} = \GW_{A,\genus}(X,\omega) \co \Sym^* \rH^*(X,\Z) \to \Q.
\end{equation}
Here $\Sym^*\rH^*(X,\Z)$ denotes the \emph{graded} symmetric algebra on the graded abelian group $\rH^*(X,\Z)$.
This map satisfies, in particular, the following axioms; cf.~\cite[§7.5]{McDuff2012}:
\begin{description}[labelwidth=\widthof{\bfseries (vanishing)}]
\item[(grading)]
  $\GW_{A,\genus}(\gamma_1 \cdots \gamma_k)=0$ unless $\sum_{i=1}^k \deg \gamma_i = 2c_1(A) + 2k$.
\item[(vanishing)]
  For every $h \in \rH^i(X,\Z)$ with $i \in \set{0,1}$
  \begin{equation*}
    \GW_{A,\genus} (h \cdot \gamma)=0.
  \end{equation*}
\item[(divisor)]
  For every $h \in \rH^2(X,\Z)$
  \begin{equation*}
    \GW_{A,\genus}(h \cdot \gamma)= \Inner{h,A} \GW_{A,\genus}(\gamma).
  \end{equation*}
\end{description}

The \defined{Gopakumar--Vafa BPS invariant} $\BPS_{A,\genus} = \BPS_{A,\genus}(X,\omega) \co \Sym^* \rH^*(X,\Z) \to \Q$ is defined by
\begin{equation}
  \label{Eq_GopakumarVafaBPS_Fano}
  \sum_{\genus=0}^\infty \GW_{A,\genus}(\gamma) \cdot t^{2g-2}
  = 
  \sum_{\genus=0}^\infty \BPS_{A,\genus}(\gamma) \cdot \paren*{2\sin\paren*{t/2}}^{2\genus-2+2c_1(A)}.
\end{equation}
Evidently, it satisfies the same axioms as $\GW_{A,\genus}$.

\begin{theorem}[{\citet[Theorem 1.5]{Zinger2011} and \citet[Corollary 1.18]{Doan2019}}]
  \label{Thm_GopakumarVafa_Fano}
  Let $(X,\omega)$ be a closed symplectic $6$--manifold and let $A \in \rH_2(X,\Z)$ be a Fano class.
  The invariants $\BPS_{A,\genus}=\BPS_{A,\genus}(X,\omega)$ defined by \autoref{Eq_GopakumarVafaBPS_Fano}
  satisfy:  
  \begin{description}[labelwidth=\widthof{\bfseries (integrality)}]
  \item[(integrality)]
    $\BPS_{A,\genus}(\gamma)\in \Z$ for every $\gamma\in \Sym^*\rH^*(X,\Z)$.
  \item[(finiteness)]
    There exists $\genus_A \in \N_0$ such that
    $\BPS_{A,\genus}(\gamma)=0$ for every $\genus \geq \genus_A$ and $\gamma \in \Sym^*\rH^*(X,\Z)$.
  \end{description}
\end{theorem}

\begin{proof}
  The integrality statement was proved by \citet[Theorem 1.5]{Zinger2011}.

  By the vanishing and divisor axioms,
  it suffices to prove that there is a $\genus_A \in \N_0$ such that
  $\BPS_{A,\genus}(\gamma_1\cdots \gamma_k) = 0$
  whenever $\genus \geq \genus_A$ and $\deg \gamma_i \geq 3$.
  The latter implies $k \leq 2c_1(A)$.
  Since 
  \begin{equation*}
    \bigoplus_{k=0}^{2c_1(A)} \Sym^k \rH^*(X,\Z)
  \end{equation*}
  is a finitely generated abelian group,
  the finiteness statement follows from \cite[Corollary 1.18]{Doan2019}.
\end{proof}

\begin{remark}
  \label{Rmk_NoGhosts}
  The proof of \cite[Corollary 1.18]{Doan2019} relies on \cite[Theorems 1.1]{Doan2019}.
  To prove the latter, \citeauthor{Doan2019} carried out a somewhat delicate analysis of the Kuranishi model at nodal pseudo-holomorphic maps with ghost components---%
  following ideas of \citet{Ionel1998} and \citet[Theorem 1.2]{Zinger2009}.
  \cite[Theorems 1.1]{Doan2019}, however, also is an immediate consequence of \autoref{Prop_LocalSmoothConvergence}.
\end{remark}

\section{Castelnuovo's bound for primitive Calabi--Yau classes}
\label{Sec_CastelnuovoOneParameterFamilies}

Let $(X,\omega)$ be a closed symplectic $6$--manifold.
Denote by $\SpaceOfAlmostComplexStructures$ the space of almost complex structures tamed by (or compatible with) $\omega$;
cf.~\autoref{Ex_SymplecticAlmostHermitianStructures}.

\begin{definition}
  For $A \in \rH_2(X,\Z)$ and $J \in \SpaceOfAlmostComplexStructures$ the \defined{Castelnuovo number} $\gamma_A(X,J)$ is
  \begin{equation*}
    \gamma_A(X,J)
    \coloneq
    \sup\set{
      \genus(C) : C ~\text{is an irreducible $J$--holomorphic curve}
    }.
    \qedhere
  \end{equation*}
\end{definition}

\cite[Theorem 1.6]{Doan2018a} established that $\gamma_A(X,J) < \infty$ provided $J \in \SpaceOfAlmostComplexStructures$ is $k$--rigid and $A$ has divisibility at most $k$ and $c_1(A) = 0$.
The subset of these $J$ is comeager \cites[Theorem 1.2]{Eftekhary2016}[Theorem A]{Wendl2016}, but fails to be path-connected---even for $k=1$.
Therefore, \cite[Theorem 1.6]{Doan2018a} does not establish Castelnuovo bounds in generic $1$--parameter families.
The results of \autoref{Sec_SpacesOfPseudoHolomorphicCurves},
however, immediately yield such bounds for primitive Calabi--Yau classes $A \in \Gamma$.

\begin{definition}
  \label{Def_JEmb}  
  Denote by $\SpaceOfAlmostComplexStructures_\emb$ the subset of those $J \in \SpaceOfAlmostComplexStructures$ satisfying
  \autoref{Def_JIsol}~
  \autoref{Def_JIsol_NonNegative}, \autoref{Def_JIsol_Embedded}, and \autoref{Def_JIsol_Disjoint}.
\end{definition}

\begin{theorem}[{\cites[Theorem 1.1]{Oh2009}[Proposition A.4]{Ionel2018}}]
\edited{
  $\SpaceOfAlmostComplexStructures\setminus\SpaceOfAlmostComplexStructures_\emb$ has codimension two in $\SpaceOfAlmostComplexStructures$;
  in particular: $\SpaceOfAlmostComplexStructures_\emb$ is comeager and path-connected.
  }
\end{theorem}

\begin{theorem}
  \label{Thm_Castelnuovo}
  If $K \subset \SpaceOfAlmostComplexStructures_\emb$ is compact,
  then for every primitive Calabi--Yau class $A \in \Gamma$
  \begin{equation*}
    \sup_{J \in K} \gamma_A(X,J) < \infty.
  \end{equation*}
\end{theorem}

\begin{proof}
  By \autoref{Thm_SpaceOfCycles_Proper},
  $\SpaceOfConnectedCycles_A(K)$ is compact.

  Let $J \in K$ and $C = \sum_{i = 1}^I m_i C_i$ with $(J, C) \in \SpaceOfConnectedCycles_A(K)$.
  As in the proof of \autoref{Lem_Cluster},
  $I = 1$;
  that is:
  $C = m_1C_1$ and $C_1$ is embedded.
  Since $A$ is primitive, $m_1 = 1$.
  Therefore,
  \begin{equation*}
    \SpaceOfConnectedCycles_A(K)
    =
    \SpaceOfConnectedCurves_A(K)
    =
    \SpaceOfConnectedEmbeddedCurves_A(K).
  \end{equation*}

  By \autoref{Cor_ModuliSpaceOfEmbeddedMapsEmbedsIntoSpaceOfSubmanifolds},
  the map $\genus \co \SpaceOfConnectedEmbeddedCurves_A \to \N_0$ assigning to $(J,C)$ the genus of $C$ is continuous.
  Since $\SpaceOfConnectedEmbeddedCurves_A(K)$ is compact,
  this implies the assertion.  
\end{proof}


\printbibliography

\end{document}
